\documentclass[notitlepage]{article}
\usepackage[utf8]{inputenc}
\usepackage{amsmath}
\usepackage{titling}
\usepackage{amsthm}
\usepackage{geometry}
\usepackage{setspace}
\usepackage{enumitem}
\usepackage{amsmath}
\usepackage{amsfonts}
\usepackage{amssymb}
\usepackage{latexsym}
\usepackage{amsthm}
\usepackage{verbatim}
\usepackage{authblk}
\usepackage[english]{babel}
\usepackage{fancyhdr}
\usepackage{titlepic}
\usepackage{fancyhdr}
\usepackage{graphicx}
\usepackage{lipsum}
\usepackage{titling}
\usepackage{textcomp}
\usepackage{gensymb}
\usepackage{amsfonts}
\usepackage{multirow}
\usepackage{subcaption}
\usepackage{siunitx}

\usepackage{setspace}
\usepackage{enumitem}
\usepackage{amsmath}
\usepackage{amsfonts}
\usepackage{amssymb}
\usepackage{latexsym}
\usepackage{amsthm}
\usepackage{verbatim}
\usepackage{authblk}
\usepackage{setspace}
\usepackage{verbatim}
\usepackage{longtable}
\usepackage{flexisym}
 \usepackage{breqn}
 \usepackage{longtable}
 \usepackage{amsmath, amssymb, etoolbox, expl3, mathtools, pgfkeys, pgfopts, xparse, xstring,tikz}
\usepackage[edge-length=1.5cm,root-radius=.2cm,label,ordering=Bourbaki,mark=o]{dynkin-diagrams}

\usepackage{microtype}  
\usepackage[intoc]{nomencl}
\usepackage{color}

\usepackage{keystroke}
\usepackage{xcolor}
\usepackage{listings}
\usepackage{setspace}
\usepackage{amsmath}
\usepackage{bm}
\usepackage{upgreek}
\usepackage{units}
\usepackage{lscape}
\usepackage{pbox}
\usepackage[tableposition=top]{caption}
\usepackage{float}
\floatstyle{plaintop}
\restylefloat{table}
\usepackage[version=3]{mhchem}
\usepackage{hyperref,bookmark,refcount}
\hypersetup{
    colorlinks=true, 
    linktoc=all,     
    linkcolor=blue,  
}

\newtheorem{theorem*}{Theorem}
\newtheorem{corollary*}[theorem*]{Corollary}
\newtheorem{theorem}{Theorem}[section]
\newtheorem{corollary}[theorem]{Corollary}
\newtheorem{lemma}[theorem]{Lemma}
\newtheorem{proposition}[theorem]{Proposition}

\usepackage[parfill]{parskip}
\DeclareMathOperator{\codim} {codim}
\DeclareMathOperator{\im} {im}
\usepackage{xltabular}
\usepackage{mathtools}

\DeclarePairedDelimiter\floor{\lfloor}{\rfloor}
\title{Modules for Algebraic Groups with Finitely Many Orbits on Totally Singular $2$-spaces}
\author{Aluna Rizzoli}
\affil{Department of Mathematics, Imperial College, London SW7 2BZ\\
DPMMS, CMS, Wilberforce Road, Cambridge CB3 0WA\\ \texttt{ar2128@cam.ac.uk}}
\usepackage[maxbibnames=99,
backend=biber,
style=numeric,
sorting=nyt
]{biblatex}
\AtEveryBibitem{\clearfield{doi}}
\AtEveryBibitem{\clearfield{url}}
\AtEveryBibitem{\clearfield{eprint}}
\AtEveryBibitem{\clearfield{isbn}}
\AtEveryBibitem{\clearfield{issn}}
\begin{document}
\maketitle
\thispagestyle{empty}

\begin{abstract}
This is the author's second paper treating the double coset problem for classical groups. Let $G$ be an algebraic group over an algebraically closed field $K$. The double coset problem consists of classifying the pairs $H,J$ of closed connected subgroups of $G$ with finitely many $(H,J)$-double cosets in $G$. The critical setup occurs when $H$ is reductive and $J$ is a parabolic subgroup. Assume that $G$ is a classical group, $H$ is simple and $J$ is a maximal parabolic $P_k$, the stabilizer of a totally singular $k$-space. We show that if there are finitely many $(H,P_k)$-double cosets in $G$, then the triple $(G,H,k)$ belongs to a finite list of candidates. Most of these candidates have $k=1$ or $k=2$. The case $k=1$ was solved in \cite{rizzoli} and here we deal with $k=2$. We solve this case by determining all faithful irreducible self-dual $H$-modules $V$, such that $H$ has finitely may orbits on totally singular $2$-spaces of $V$. 
\end{abstract}

\section{Introduction}
In this paper we consider a question concerning double cosets in algebraic groups. Let $G$ be a simple algebraic group over an algebraically closed field $K$ of characteristic $p\geq 0$. The general problem consists of describing pairs of closed subgroups $H,J\leq G$  for which there are finitely many $(H,J)$-double cosets. This is a question that has attracted considerable interest thanks to the interesting range of examples coming from group theory and representation theory.

Parabolic subgroups produce multiple families of examples. As a consequence of the Bruhat decomposition $G=\bigcup_{w\in W}B\dot{w}B,$ for any two parabolic subgroups $H,J\leq G$ we have  $|H\backslash G/J|<\infty.$
For a second class of examples, consider a Levi decomposition $P=QL$ of a parabolic subgroup $P$, where $Q=R_u(P)$ is the unipotent radical of $P$ and $L$ is a Levi complement. Then $Q$ has a filtration by irreducible modules for $L$ (\cite[Thm.~17.6]{MT}). These are called \textit{internal modules}. 
A result of Azad, Barry and Seitz (\cite[Thm.~2(f)]{azad}) following work of Richardson says that $L$ has finitely many orbits on the vectors of each internal module, and the semisimple group $L'$ has finitely many orbits on subspaces of dimension $1$. With this setup, if $V$ is an internal module for $L'$, we have $ |L'\backslash SL(V)/P_1|<\infty,$ where $P_1$ is a maximal parabolic subgroup of $SL(V)$ that stabilises a $1$-dimensional subspace. 

Another class of examples occurs when both $H$ and $J$ are reductive subgroups of $G$. A result of Brundan \cite[Thm. A]{brundan} says that if $H$ and $J$ are either maximal connected or a Levi subgroup of a parabolic, and the number of $(H,J)$-double cosets in $G$ is finite, then there is actually only $1$ double coset and the group $G$ has a factorization $G=HJ$. Such factorizations have been classified by Liebeck, Saxl and Seitz in \cite{factorizations}.

This implies that the double coset problem has been solved in the case where both $H$ and $J$ are reductive, as well as when both $H,J$ are parabolic subgroups. 
A well-known result of Borel and Tits \cite[Prop.~2.3]{BorelTits} implies that any closed connected subgroup of $G$ is either reductive, or lies in a parabolic subgroup of $G$. It is then natural to assume that $H$ is reductive and $J$ is contained in a parabolic subgroup. 

Most efforts in the literature have focused on the case where $J$ is not just contained in a parabolic, but is a maximal parabolic. This is the same setup that we will be using, though some results were achieved without the maximality assumption on $J$. For example, when $J$ is a Borel subgroup and $|H\backslash G/J|<\infty$, the group $H$ is called spherical. Spherical subgroups have been classified by Kr\"amer in characteristic $0$ and by Knop and R\"ohrle in \cite{spherical} for arbitrary characteristic. 

Assume that $H$ is reductive and $J$ is a maximal parabolic. The case where $G=SL(V)$, for a finite dimensional vector space $V$ over $K$, has been settled by Guralnick, Liebeck, Macpherson and Seitz. In \cite{finite} the authors determined all irreducible connected subgroups of $SL(V)$ with finitely many orbits on $k$-spaces of $V$. When $k=1$ these are called \textit{finite orbit modules}. Since a maximal parabolic subgroup of $SL(V)$ is precisely the stabilizer $P_k$ of a $k$-space, we have that $H\leq SL(V)$ has finitely many orbits on $k$-spaces of $V$ if and only if there is a finite number of $(H,P_k)$-double cosets in $SL(V)$. 

As pointed out in \cite{SeitzDoubleCosets} things are different when we consider other classical groups instead of $SL(V)$. For example, the group $H=G_2$ has infinitely many orbits on $1$-spaces on its $14$-dimensional Lie algebra $V=\mathrm{Lie}(G_2)$ for $p\neq 3$. However it preserves a non-degenerate quadratic form on $V$ (again see \cite{SeitzDoubleCosets}) and can therefore be regarded as a subgroup of $SO(V)$, and we showed that it has finitely many orbits on singular $1$-spaces \cite[Prop.~6.13]{rizzoli}.

If $\rho: H\rightarrow GL(V)$ is a representation of a group $H$ then we say that the $H$-module $V$ is \textit{orthogonal} if $\rho(H)\leq SO(V)$, and \textit{symplectic} if $\rho(H)\leq Sp(V)$. 

Given the classification for $G$ of type $A_n$, we let $G=Cl(V)$ be a classical group of type $B_n$, $C_n$ or $D_n$ with natural module $V$. 
The maximal parabolic subgroups of $G$ are then stabilizers of totally singular subspaces. We remark that the double coset problem for $G$ exceptional is still open, though some good progress can be found in \cite{duckworth}. The double coset problem for $G=Sp(V)$ and $J=P_1$ is trivially solved from the $G=SL(V)$ case, since a $P_1$-parabolic subgroup of $Sp(V)$ is the stabilizer of a $1$-space of $V$. An orthogonal $H$-module with finitely many orbits on singular $1$-spaces is a \textit{finite singular orbit module}. The case $G=SO(V)$ and $k=1$ was solved in \cite{rizzoli}, where we determined all orthogonal faithful irreducible finite singular orbit modules for simple algebraic groups (see \cite[Thm.~1]{rizzoli}).

In Proposition~\ref{prop k=2 list of candidates restated} we determine a finite list of triples $(H,V,k)$ for $k\geq 2$, with the property that if $H$ has finitely many orbits on totally singular $k$-spaces of $V$ then the triple $(H,V,k)$ is in the given list. 
The modules $V$ are denoted by their highest weight, the groups $H$ by their Dynkin diagram. 

As we can see from Table~\ref{tab:candidates higher dimensions restated} in Proposition~\ref{prop k=2 list of candidates restated}, most of the cases that need to be addressed occur when $k=2$. The case $k\geq 3$ is not dealt with in this paper, but is left for a future project. With the following theorem, we provide a classification for the case where $J=P_2$, the stabilizer of a totally singular $2$-space. We use $T_i$ to denote an $i$-dimensional torus, $Sym_n$ and $Alt_n$ to denote the symmetric group and the alternating group on a set of size $n$, and $Dih_{2n}$ for a dihedral group of order $2n$. By a field or graph twist of a module $V$, we mean a module $V^\alpha$ obtained from $V$ by twisting the action of the group by a field or graph automorphism $\alpha$. Note that such a module $V^\alpha$ is quasiequivalent to $V$, which by \cite[Lemma~2.10.14]{KL} means that the action groups of $V$ and $V^\alpha$ are $GL(V)$-conjugate.

\begin{theorem*}\label{main theorem k=2}
Let $H$ be a simple irreducible closed connected subgroup of $SO(V)$ or $Sp(V)$ such that $H$ has finitely many orbits on totally singular $2$-spaces. Then either $H$ has finitely many orbits on all $2$-spaces (see \cite[Thm.~2]{finite}), or up to field and graph twists $(H,V)$ is as in Table~\ref{tab:theorem 2}. Furthermore, for every case in Table~\ref{tab:theorem 2} the group $H$ has many orbits on totally singular $2$-spaces. 
\begin{center}
\begin{longtable}{l l l l l l}
\caption{Modules for simple groups with finitely many orbits on totally singular $2$-spaces} \label{tab:theorem 2} \\

\hline \multicolumn{1}{c}{\textbf{$H$}} & \multicolumn{1}{c}{\textbf{$V$}} & \multicolumn{1}{c}{\textbf{$\dim V$}}& \multicolumn{1}{c}{\textbf{$p$}} & \multicolumn{1}{c}{Generic Stabilizer} & \multicolumn{1}{c}{Reference}\\ \hline 
\endhead
$A_1$ & $\lambda_1+p^i\lambda_1$ & $4$ & $\neq 0$ & $U_1T_1$ & \ref{A1 lambda1+p lambda1} \\*
$A_1$ & $3\lambda_1$ & $4$ &  $\neq 2,3$ & $Alt_4$ & \ref{a1 3 and 4 prop}\\*
$A_1$ & $4\lambda_1$ & $5$ &  $\neq 2,3$ &  $Sym_3$ & \ref{a1 3 and 4 prop}\\*
 $A_2$ & $\lambda_1+\lambda_2$& $7$ & $3$ & $U_1$ & \ref{proposition a2 k=2}\\*

 $C_3$ & $ \lambda_2$ & $ 13$& $3$ & $ U_1(T_1.2)$ & \ref{c3 k=2 main prop}\\* 

  $B_4$ & $ \lambda_4$& $ 16$ & $\neq 2$ & $A_1(A_2.2) $ & \ref{intial b4 k=2 prop}  \\*
    $B_4$ & $ \lambda_4$& $ 16$ & $2$ & $ U_5A_1A_1$ & \ref{intial b4 k=2 prop}\\*

    $F_4$ & $ \lambda_4$& $ 25$ & $3$ & $ U_1(A_2.2)$ & \ref{F4 main prop} \\*
\hline
\end{longtable}
\end{center}
\end{theorem*}
With the exception of $V_{A_1}(4\lambda_1)$, for all the cases in the conclusion of Theorem~\ref{main theorem k=2}, we determine a complete list of orbits and stabilizers in Section~\ref{proof of main theorem section}.

It is the case that simple groups acting irreducibly have finitely many orbits on $1$-spaces if and only if they have a dense orbit. This turns out to still be true when looking at orbits on singular $1$-spaces (see \cite[Cor.~2]{rizzoli}) as well as totally singular $2$-spaces.

\begin{corollary*}\label{main corollary}
Let $H$ be a simple algebraic group over $K$ and let $V$ be a rational self-dual irreducible $KH$-module. Then $H$ has finitely many orbits on totally singular $2$-spaces if and only if $H$ has a dense orbit on totally singular $2$-spaces.
\end{corollary*}

In order to prove Theorem~\ref{main theorem k=2} and Corollary~\ref{main corollary} we proceed in the following manner. In Section~\ref{list of candidates section} we determine a finite list of triples $(H,V,k)$ for $k\geq 2$, with the property that if $H$ has finitely many orbits on totally singular $k$-spaces of $V$ then the triple $(H,V,k)$ is in the given list (Table~\ref{tab:candidates higher dimensions restated}). In Section~\ref{proof of main theorem section} we then proceed on a case by case basis for every module in Table~\ref{tab:candidates higher dimensions restated} with $k=2$, either explicitly finding a finite list of orbits for the action on totally singular $2$-spaces (Sections~\ref{a1 section}, \ref{a2 section}, \ref{c3 section}, \ref{f4 section k=2}, \ref{b4 section}), or showing that there is no dense orbit (Section~\ref{no dense orbit section}). This will conclude the proof of Theorem~\ref{main theorem k=2} and Corollary~\ref{main corollary}.

Before we start we set up some of the required notation in Section~\ref{notation section}, and list the needed preliminary results in Section~\ref{preliminary section}. 

\section*{Acknowledgements}
The author would like to thank Professor Martin Liebeck for the great support, and acknowledge his EPSRC funding.

\section{Notation}\label{notation section}

In this section we set up the general notation that we are going to use throughout. 
If $G$ is semisimple, let $T$ be a fixed maximal torus, and $\Phi$ the root system for $G$ with respect to $T$, described by its Dynkin diagram.
The root system has positive roots $\Phi^+$ and a base $\Delta=\{\alpha_1,\alpha_2,\dots,\alpha_n\}$ of \textit{simple roots}. For the simple algebraic groups the ordering of the simple roots is taken according to Bourbaki. We use $P$ to denote a parabolic subgroup containing a Borel subgroup $B$ and $P_k$ to denote the maximal parabolic subgroup obtained by deleting the $k$-th node of the Dynkin diagram for $G$. 

For a subsystem $\Phi_I$ of $\Phi$ let $P=BW_IB$ be the corresponding parabolic subgroup with unipotent radical $U_I=R_u(P)$ and Levi decomposition $P=U_IL_I$. 
The \textit{height} of a root $\beta=\sum_{\alpha\in\Delta}c_\alpha\in\Phi$ is $\mathrm{ht}(\beta):=\sum_{\alpha\in\Delta}c_\alpha $.
For every $\beta=\sum_{\alpha\in\Delta_I}c_\alpha\alpha+\sum_{\gamma\in\Delta\setminus\Delta_I}d_\gamma\gamma\in \Phi^+$, we say that $\beta$ is of \textit{level} $\sum d_\gamma$ and we call $(d_\gamma)_{\gamma\in\Delta\setminus\Delta_I}$ the \textit{shape} of $\beta$. Then $U_{(j)}$ denotes the normal subgroup of $P$ generated by the positive root subgroups of level at least $j$ (see \cite[§17.1]{MT}. As $U_{(j)}\trianglelefteq P$, the Levi subgroup $L$ acts on the quotient $U_{(j)}/U_{(j+1)}$. The quotient $U_{(j)}/U_{(j+1)}$ has a $K$-vector space structure, given by $cx_\beta(t)U_{(j+1)}=x_\beta(ct)U_{(j+1)}$, for $\beta\in \Phi^+\setminus \Phi_I$ of level $j$ and $c\in K$. The vector space structure on $U_{(j)}/U_{(j+1)}$ commutes with the action of $L$, and $U_{(j)}/U_{(j+1)}$ can therefore be regarded as a $KL$-module. Given a shape $\mathcal{S}$ of level $j$, the internal module $V_{\mathcal{S}}$ is a $KL$-submodule of $U_{(j)}/U_{(j+1)}$, given by $\prod_{\beta }U_\beta U_{(j+1)}/U_{(j+1)}$, where the product runs over roots of shape $\mathcal{S}$. 

The \textit{Weyl group} of $G$ is $N_G(T)/C_G(T)$ and is generated by the simple reflections $s_\alpha$ for $\alpha\in \Delta$. The root subgroup corresponding to a root $\alpha$ is denoted by $X_{\alpha}$, while $U_i$ denotes an $i$-dimensional unipotent group. The pair $X_{\alpha},X_{-\alpha}$ generates a subgroup of type $A_1$. Let $\phi_\alpha:  A_1\rightarrow \langle X_\alpha,X_{-\alpha}\rangle$ be an isomorphism as described in \cite[§6]{cartersimple}. We identify specific elements of $\langle X_{\alpha},X_{-\alpha}\rangle$ according to the corresponding element in $A_1=SL_2(K)$. The images under $\phi_\alpha$ of the elements $\left(\begin{matrix} 1 & t\\ 0 & 1 \\ \end{matrix}\right) $ and $\left(\begin{matrix} 1 & 0\\ t & 1 \\ \end{matrix}\right) $ are $x_\alpha(t)\in X_\alpha$ and $x_{-\alpha}(t)\in X_{-\alpha}$. The element $\phi_\alpha(\mathrm{diag}(\lambda,\lambda^{-1}))$ is called $h_{\alpha}(\lambda)$.

We denote by $Cl(V)$ a classical group $SL(V),Sp(V)$ or $SO(V)$, with natural module $V$.
Both $Sp(V)$ and $SO(V)$ preserve a non-degenerate bilinear form $(\cdot,\cdot):V\times V\rightarrow K$, which is respectively alternating or symmetric. The \textit{radical} of a subspace $U$ of $V$, denoted by $Rad(U)$, is the subspace $U\cap U^\perp$. Furthermore, $SO(V)$ fixes a quadratic form $Q:V\rightarrow K$ that satisfies $(u,v)=Q(u+v)-Q(u)-Q(v)$ for all $u,v\in V$. 

We say that a vector $v\in V$ is \textit{singular} if $(v,v)=0$ in the symplectic case, and $Q(v)=0$ in the orthogonal case. We say that a subspace $U\leq V$ is \textit{totally singular} if every vector in $U$ is singular. An even dimensional vector space $V$ with a non-degenerate bilinear form has a basis $\{e_i,f_i\}_i$ such that $(e_i,e_j)=(f_i,f_j)=0$ and $(e_i,f_j)=\delta_{ij}$ for all $1\leq i,j\leq \dim V/2$, and $Q(e_i)=Q(f_i)=0$ in the orthogonal case. The pairs $e_i,f_i$ are called \textit{hyperbolic pairs}. We denote the variety of $k$-spaces by $P_k(V)$, and the variety of totally singular $k$-spaces by $P_k^{TS}(V)$. Finally, if $V$ is an $H$-module and $W\leq V$, we denote by $C_H(W)$ the centralizer of $W$ in $H$.

\section{Preliminary results}\label{preliminary section}
In this section we list all the preliminary results that we are going to use to prove Theorem~\ref{main theorem k=2} and Corollary~\ref{main corollary}.
We begin with some technical results about the choice of algebraically closed field and restrictions from $D_{n+1}$ to $B_n$. 

\begin{proposition}\cite[Prop. 1.1]{finite}\label{field independent}
Let $k\leq K$ be two algebraically closed fields of characteristic $p$. Let $G=G(K)$ be a connected reductive algebraic group over $K$, defined over $k$. Denote by $G(k)$ the group of $k$-rational points of $G(K)$. Suppose that $G(K)$ acts algebraically on the affine variety $V(K)$, and the action is defined over $k$. Then $G(K)$ has finitely many orbits on $V(K)$ if and only if $G(k)$ has finitely many orbits on $V(k)$. If this holds the number of orbits is the same in each case.
\end{proposition}

As pointed out at the end of \cite{finite}, Proposition~\ref{field independent} can also be proved using model theory to deduce the $p=0$ case from the $p>0$ one. In the same fashion we have the following result.

\begin{proposition}\label{characteristic 0 prop}
For all $p\geq 0$, let $G_p$ denote an arbitrary simple algebraic group with a fixed root system $\Phi$, over an algebraically closed field of characteristic $p$. Let $\lambda$ be a dominant weight with respect to $\Phi$.
Let $k\in \mathbb{N}^*$. Suppose that there exists $m\in \mathbb{N}$ such that for all $p\geq m$, the group $G_p$ has finitely many orbits on $P_k^{TS}(V_{G_p}(\lambda))$.
Then the same is true for $p=0$. 
\end{proposition}

\begin{proof}
First note that $G_p$ acts with finitely many orbits on $P_k^{TS}(V_{G_p}(\lambda))$ independently of the isogeny type.  
By Robinson's Theorem, any statement in first order logic in the language of fields that is true over algebraically closed fields of arbitrarily high characteristic, is also true in characteristic $0$. The discussion at the end of \cite{finite} shows that this proposition is indeed equivalent to a statement in first order logic in the language of fields.
\end{proof}

We recall another general result from \cite{finite}, which will allow us to prove the existence of finitely many orbits for an algebraic group by showing that when passing to finite fields we have a uniform bound on the number of orbits. Assume $p>0$. For each power $q$ of $p$, let $\sigma_q$ be the Frobenius morphism of $SL(V)$, sending $x_\alpha(t)\rightarrow x_\alpha(t^q)$ for all roots $\alpha$ and $t\in K$, and acting in a compatible way on $V$. Assume $G$ is a closed connected subgroup of $SL(V)$ which is $\sigma_q$-stable for some $q$. Let $G(q^e)$ denote the group of fixed points of $\sigma_{q^e}$ acting on $G$ and $V(q^e)$ denote the fixed points of $\sigma_{q^e}$ acting on $V$.
\begin{lemma}\cite[Lemma 2.10]{finite}\label{finiteOrbits}
Under the above assumptions, $G\leq SL(V)$ has finitely many orbits on $P_k(V)$ if and only if there exists a constant $c$ such that $G(q^e)$ has at most $c$ orbits on $P_k(V(q^e))$ for all $e\geq 1$. Furthermore, if the latter statement is true, $G$ has at most $c$ orbits on $P_k(V)$.
\end{lemma}

If $V_{2n}$ is the natural module for $D_n$, then we can obtain $B_{n-1}$ as $(D_n)_{v }$, where $v$ is a non-singular vector in $V_{2n}$. Suppose that $X$ is a $D_n$-module. Let orb$(G,Y)$ denote the set of orbits of a group $G$ acting on a set $Y$. 
The following lemma describes how each $D_n$-orbit on $k$-spaces splits when restricting to $B_{n-1}$. 

\begin{lemma}\label{orbit correspondence Dn Bn}
Let $W_k\in P_k(X)$ and let $\Delta=\langle W_k \rangle ^{D_n}$ be an orbit of $D_n$ on $P_k(X)$. Let $S=(D_n)_{W_k}$. Then there is a bijective correspondence between orb$(B_{n-1},\Delta)$ and the orbits of $S$ on non-singular $1$-spaces of $V_{2n}$. More specifically if $g\in D_n$ and $\alpha=g\langle v \rangle $ is a non-singular $1$-space in $V_{2n}$, then the orbit $\alpha ^{S}$ corresponds to the orbit $\langle g^{-1}\cdot W_k \rangle^{B_{n-1}}$.
\end{lemma}

\begin{proof}
Let $$\phi_1: \Delta \rightarrow [D_n:S]$$ be the bijection defined by $\phi_1(g\cdot W_k)=gS$ for all $g\in D_n$. Let $$\phi_2: \mathrm{orb}(B_{n-1},\Delta)\rightarrow B_{n-1}\backslash D_n /S$$  be the bijection defined by $\phi_2((g\cdot W_k)^{B_{n-1}})=B_{n-1}\phi_1(g\cdot W_k)$. Let $$\phi_3: B_{n-1}\backslash D_n /S\rightarrow S\backslash D_n /B_{n-1}$$ be the canonical bijection of double cosets such that $\phi_3(B_{n-1}gS)=Sg^{-1}B_{n-1}$. Finally let $$\phi_4: S\backslash D_n /B_{n-1}\rightarrow \mathrm{orb}(S,\langle v\rangle^{D_n})$$ be the bijection defined by $\phi_4(Sg^{-1}B_{n-1})=\alpha^S$ for $\alpha = g^{-1}\langle v\rangle $.

The composition $$\phi_4 \circ \phi_3\circ \phi_2: \mathrm{orb}(B_{n-1},\Delta) \rightarrow \mathrm{orb}(S,\langle v\rangle^{D_n})$$ is the required bijection.

\end{proof}

We now list results related to totally singular subspaces.
We start with the number of totally singular subspaces of a certain dimension.

\begin{proposition}\cite[Ex.~11.3]{T}\label{number of totally singular subspaces symplectic}
Let $V$ be a symplectic geometry of dimension $2m$ over the finite field $\mathbb{F}_q$. Then, for $k\leq m$, we have $$|P_k^{TS}(V)|=\prod_{i=0}^{k-1}(q^{2m-2i}-1)/(q^{i+1}-1).$$
\end{proposition}

\begin{proposition}\cite[Ex.~11.3]{T}\label{number of totally singular subspaces symplectic orthogonal}
Let $V$ be an orthogonal geometry of dimension $n$ over $\mathbb{F}_q$ and let $m>0$ be the Witt index of $V$. Then $$|P_k^{TS}(V)|={\genfrac[]{0pt}{1}{m}{k}}_q\prod_{i=0}^{k-1}(q^{n-m-i-1}+1) $$
where
$$\genfrac[]{0pt}{1}{m}{k}_q=\prod_{i=0}^{k-1}(q^{m}-q^i)/(q^k-q^i). $$
\end{proposition}

The following proposition serves as a reference for the dimension of the varieties of totally singular $k$-spaces.

\begin{proposition}\label{dimension totally singular subspaces}
Let $V$ be either a symplectic or orthogonal geometry of dimension $n$ over an algebraically closed field. Then $\dim P_k^{TS}(V)= kn+\frac{k-3k^2}{2}$ in the symplectic case and $ kn-\frac{k+3k^2}{2}$ in the orthogonal case.
\end{proposition}

\begin{proof}
The dimensions are simply given by $\dim G / P_k=\dim G-\dim P_k$ for $G=B_n,C_n,D_n$. Note that they agree with the degree of the polynomials giving the number of totally singular subspaces in Proposition~\ref{number of totally singular subspaces symplectic} and Proposition~\ref{number of totally singular subspaces symplectic orthogonal}. 
\end{proof}
If $G$ is an algebraic group with a dense orbit on $P_k^{TS}(V)$ for a $G$-module $V$, then we have the following bounds on $\dim V$.
\begin{lemma}\label{1-spaces bound}
Suppose that $G<SO(V)$ has a dense orbit on singular $1$-spaces of $V$. Then $$\dim V\leq \dim G +2.$$
\end{lemma}
\begin{proof}
The dimension of the variety of singular $1$-spaces in $V$ is $\dim V -2$ by Proposition~\ref{dimension totally singular subspaces}. Hence $G$ must be at least of dimension $\dim V -2$.
\end{proof}
\begin{lemma}\label{k-spaces bound}
Suppose that $G<SO(V)$ or $G<Sp(V)$ has a dense orbit on $P_k^{TS}(V)$, where $1\leq k \leq \dim V/2$. Then
$$\dim V\leq \frac{\dim G}{k}+\frac{3k+1}{2} \leq \frac{4}{k}\dim G+2.$$
\end{lemma}
\begin{proof}
By assumption $1\leq k \leq \dim V/2$. By Proposition~\ref{dimension totally singular subspaces} we must have $\dim G\geq  k\dim V-\frac{3}{2}k^2-\frac{k}{2}$. Therefore $\dim V \leq \frac{\dim G}{k}+\frac{3}{2}k+\frac{1}{2}\leq \frac{\dim G}{k} +\frac{3}{4}\dim V$ and $\dim V\leq \frac{4}{k}\dim G+2$, as claimed.
\end{proof}

We then have the following lemma.

\begin{corollary}\label{k bound}
Suppose that $G<SO(V)$ or $G<Sp(V)$ has a dense orbit on $P_k^{TS}(V)$, where $1\leq k \leq \dim V/2$.
Then $$k^2-k\leq 2\dim G.$$
\end{corollary}
\begin{proof}
By Lemma~\ref{k-spaces bound} we know that $\dim V\leq \frac{1}{2k}(2\dim G +3k^2+k) $. Since $\dim V \geq 2k$ this means that $4k^2\leq 2\dim G  +3k^2+k$, which gives the required conclusion.
\end{proof}

We conclude with a general bound on $\dim V$ in terms of $\dim G$, independent of $k$.

\begin{lemma}\label{super bound}
Suppose that $G<SO(V)$ or $G<Sp(V)$ has a dense orbit on $P_k^{TS}(V)$, where $1\leq k \leq \dim V/2$. Then if $\dim G \geq 5$ we have $$\dim V\leq \dim G +2.$$
\end{lemma}

\begin{proof}
By Corollary~\ref{k bound} we know that $k^2-k\leq 2\dim G$. Therefore $$k\leq \floor{\frac{1}{2}(\sqrt{8\dim G +1}+1)}.$$ By Lemma~\ref{k-spaces bound} we have $$\dim V\leq \frac{\dim G}{k}+\frac{3k+1}{2}.$$ Now $$ \frac{\dim G}{k}+\frac{3k+1}{2}\leq \dim G +2 \iff \dim G\geq \frac{k}{k-1}\cdot \frac{3k-3}{2}=\frac{3k}{2}.$$ This is equivalent to $k\leq \frac{2}{3}\dim G$ which holds if $\floor{\frac{1}{2}(\sqrt{8\dim G +1}+1)} \leq \frac{2}{3}\dim G$ or equivalently if $\dim G \geq 5$. 
\end{proof}

For any group $G$ and endomorphism $\sigma$ let $H^1(\sigma,G)$ denote the set of equivalence classes of $G$ under the equivalence relation $x\sim y\iff y=z^{\sigma}xz^{-1}$ for some $z\in G$.
Now, let $G$ be a connected algebraic group and $\sigma$ a Frobenius endomorphism of $G$. 

The Lang-Steinberg theorem \cite[Thm.~21.7]{MT} says that if $\sigma:G\rightarrow G$ is a surjective homomorphism of connected algebraic groups, such that $G_{\sigma}$ is finite, then the map $g\rightarrow g^{\sigma}g^{-1}$ from $G\rightarrow G$ is surjective. 

Whenever a connected group $G$ acts on a set $S$ on which $\sigma$ also acts compatibly, the Lang-Steinberg theorem allows us to determine the $G_\sigma$-orbits on the set of fixed points $S_\sigma$. 
The following proposition, which is a direct consequence of Lang-Steinberg, will be used to understand how the orbits split when going to finite fields.

\begin{proposition}\label{lang-steinberg}\cite[I, $2.7$]{LangSteinberg}
Let $\sigma$ be a Frobenius endomorphism of the connected group $G$. Suppose that $G$ acts transitively on a set $S$, and that $\sigma$ also acts on $S$ in such a way that $(hs)^{\sigma}=h^{\sigma}s^{\sigma}$ for all $s\in S,h\in G$. Then the following hold.
\begin{enumerate}[label=(\roman*)]
\item $S$ contains an element fixed by $\sigma$;
\item Fix $s_0\in S_{\sigma}$, and assume that $X=G_{s_0}$ is a closed subgroup of $G$. Then there is a bijective correspondence between the set of $G_{\sigma}$-orbits on $S_{\sigma}$ and the set $H^1(\sigma,X/X^0).$
\end{enumerate}
\end{proposition}

A fruitful angle in understanding the action of simple algebraic groups on Grassmannian varieties comes from the notion of \textit{generic stabilizer}.
If $G$ is an algebraic group acting on a variety $X$, we say that the action has \textit{generic stabilizer} $S$ if there exists an open subset $U\leq X$ such that $G_u$ is conjugate to $S$ for all $u\in U$. 

We conclude by discussing some methods and implications of the solution to the generic stabilizer problem by Guralnick and Lawther in \cite{generic1} and \cite{generic2}. In these two papers the authors prove the existence (with one exception) of generic stabilizers for the action of simple algebraic groups on faithful irreducible rational modules and corresponding Grassmannian varieties. In each such case they explicitly determine the structure of the generic stabilizer.

The most direct application for our problems comes from the following lemma and its corollary. 

\begin{lemma}\label{minimum dimension lemma}
Let $G$ be a simple algebraic group and suppose that $G$ acts rationally on an irreducible variety $X$. Let $Y$ be an open subset of $X$ and let $d=\inf_{y\in Y}\dim G_y$. Then $\dim G_x\geq d$ for all $x\in X$.
\end{lemma}

\begin{proof}
Consider the morphism $$\phi: G\times X\rightarrow X\times X,$$ where $\phi(g,x)=(gx,x)$. Replace the codomain of $\phi$ with $\overline{\phi(G\times X)}$, so that $\phi$ is a dominant morphism. Let $e=\inf_{x\in X}\dim G_x$, so that $\dim \phi(G\times X)=\dim G -e+\dim X$.
Then by the Fiber Dimension Theorem (see remarks after \cite[Cor.~14.6]{eisenbud}), there exists an open subset $U$ of $\overline{\phi(G\times X)}$ such that for all $u\in U$ we have $\dim \phi^{-1}(u) =\dim G\times X -\dim \overline{\phi(G\times X)}=\dim G\times X -\dim \phi(G\times X)=e$.

Let $(y,x)\in \phi(G\times X)$. Then the fiber $\phi^{-1}(y,x)=\{(g,x)|gx=y\}$ has dimension $\dim G_x$. Since $Y\times Y$ is an open subset of $X\times X$, by the previous paragraph there must exist $y\in Y$ such that $\dim G_y=e$. This implies that $e=d$, concluding.  
\end{proof}

\begin{corollary}\label{minimum dimension generic}
Let $G$ be a simple algebraic group and suppose that $G$ acts rationally on an irreducible variety $X$. If there exists a generic stabilizer $S$ for the $G$-action on $X$, then for all $x\in X$ we have $\dim G_x\geq \dim S$.
\end{corollary}

\begin{proof}
By assumption there exists an open set $Y$ such that all elements of $Y$ have stabilizer conjugate to $S$. By Lemma~\ref{minimum dimension lemma} we then find that $\dim G_x\geq \dim S$ for all $x\in X$.
\end{proof}

We will sometimes make use of Corollary~\ref{minimum dimension generic} to prove that a group $H$ has no dense orbit on $P_2^{TS}(V_H(\lambda))$. Suppose that $S$ is the generic stabilizer for the $H$-action on $P_2(V_H(\lambda))$. Then by Corollary~\ref{minimum dimension generic}, if $\dim H-\dim S< \dim  P_2^{TS}(V_H(\lambda))$, there cannot be a dense orbit on $P_2^{TS}(V_H(\lambda))$.

Let us consider some of the methods used in \cite{generic1,generic2}.
In particular we are interested in the \textit{localization to a subvariety} approach \cite[§1.4]{generic1}. Let $X$ be a variety on which a simple algebraic group $G$ acts. Let $Y\subseteq X$ and $x\in X$. The \textit{transporter} in $G$ of $x$ into $Y$ is $$\mathrm{Tran}_G(x,Y)=\{g\in G:g.x\in Y\}.$$

If $Y$ is a subvariety of $X$, a point $y\in Y$ is called \textit{$Y$-exact} if $$\codim \mathrm{Tran_G}(y,Y)=\codim Y.$$

Let $\phi:(G,X)\rightarrow X$ be the orbit map. Then we have the following lemma.

\begin{lemma}\cite[Lemma~1.4.2]{generic1}\label{loc to a subvariety lemma}
Let $Y$ be a subvariety of $X$, and let $\hat{Y}$ be a dense open subset of $Y$. Suppose that all points in $\hat{Y}$ are $Y$-exact. Then $\phi(G\times \hat{Y})$ contains a dense open subset of $X$.
\end{lemma}

\section{List of candidates}\label{list of candidates section}

In this section we prove the following proposition.

\begin{proposition}\label{prop k=2 list of candidates restated}
Let $H<SO(V)$ or $H< Sp(V)$ be a connected simple algebraic group over an algebraically closed field of characteristic $p$, acting irreducibly with finitely many orbits $P_k^{TS}(V)$, for some $2\leq k \leq \dim V/2$.  Suppose that $H$ does not already have finitely many orbits on $P_k(V)$. Then, up to graph and field twists, $H$ and $V$ are in Table~\ref{tab:candidates higher dimensions restated}.

\begin{center}
\begin{longtable}{l l l l l}
\caption{Candidate modules for $k\geq 2$} \label{tab:candidates higher dimensions restated} \\

\hline \multicolumn{1}{c}{\textbf{$H$}} & \multicolumn{1}{c}{\textbf{$V$}} & \multicolumn{1}{c}{\textbf{$\dim V$}}&  \multicolumn{1}{c}{\textbf{$k$}} & \multicolumn{1}{c}{\textbf{$p$}} \\ \hline 
\endhead
$A_1$ & $\lambda_1+p^i\lambda_1$ & $4$ & $k=2$ & \\*
$A_1$ & $3\lambda_1$ & $4$ & $k=2$ & $p\neq 2,3$\\*
$A_1$ & $4\lambda_1$ & $5$ & $k=2$ & $p\neq 2,3$\\*
$A_2$ & $\lambda_1+\lambda_2$ & $7$ & $k=2,3$ & $p=3$\\*
$A_2$ & $\lambda_1+\lambda_2$ & $8$ & $k=4$ & $p\neq 3$\\
$A_5$ & $\lambda_3$ & $20$ & $k=2$ & \\ \hline
$B_2$ & $2\lambda_1$ & $10$ & $k=5$ & \\
$B_3$ & $\lambda_3$ & $8$ & $k=4$ & \\
$B_4$ & $\lambda_4$ & $16$ & $k=2,3,7,8$ & \\ \hline
$C_3$ & $\lambda_2$ & $13$ & $k=2$ & $p=3$ \\*
$C_3$ & $\lambda_2$ & $14$ & $k=2$ & $p\neq 3$ \\
$C_3$ & $\lambda_2$ & $14$ & $k=7$ & $p\neq 3$ \\ \hline
$D_6$ & $\lambda_6$ & $32$ & $k=2$ &  \\ \hline
$G_2$ & $\lambda_1$ & $7$ & $k=3$ & $p \neq 2$ \\
$F_4$ & $\lambda_4$ & $25$ & $k=2$ & $p =3$ \\
$F_4$ & $\lambda_4$ & $26$ & $k=2$ & $p \neq 3$ \\
$E_7$ & $\lambda_7$ & $56$ & $k=2$ &  \\
\hline
\end{longtable}
\end{center}
\end{proposition}

The strategy is simple. In \cite{lubeck} we find complete lists of $p$-restricted modules for simple algebraic groups satisfying the dimension bounds of Lemma~\ref{k-spaces bound}. For every self-dual module we can then determine its Frobenius-Schur indicator as explained in the following paragraphs.

An algebraic group $G$ stabilises a non-degenerate bilinear form on a $G$-module $V$ if and only if $V$ is isomorphic to its dual, denoted by $V^*$. Let $\lambda\in X(T)$ be dominant. Then $V_{G}(\lambda)^*\simeq V_{G}(-w_0(\lambda))$, where $w_0\in W$ is the longest element (\cite[Prop.~16.1]{MT}.
In particular it is known (\cite[Remark~16.2]{MT}) that $w_0=-id$ in type $A_1,B_n,C_n,D_n$ ($n$ even), $E_7,E_8,F_4,G_2$. In the remaining cases, $-w_0$ induces an involutory graph-automorphism of the Dynkin diagram.

A self dual $G$-module $V$ has \textit{Frobenius-Schur indicator} $+1$ if $G\leq SO(V)$, and $-1$ otherwise. If $G$ is a simple algebraic group and $p\neq 2$, we are easily able to determine the Frobenius-Schur indicator of an irreducible highest weight module $V_G(\lambda)$.

\begin{lemma}\cite[§6.3]{lubeck}\label{froebenius schur}
Let $G$ be a connected simple algebraic group and $V=V_G(\lambda)$ a self-dual $G$-module in characteristic $p\neq 2$. Then if $Z(G)$ has no element of order $2$ the Frobenius-Schur indicator of $V$ is $+1$. 

Otherwise let $z$ be the only element of order $2$ in $Z(G)$, except for the case $G=D_l$ with even $l$, where $z$ is the element of $Z(D_l)$ such that $D_l/\langle z \rangle \simeq SO_{2l}(k)$, with $D_l$ simply connected. Then the Frobenius-Schur indicator of $V$ is the sign of $\lambda(z)$. This can be computed by \cite[§6.2]{lubeck}.
\end{lemma}

If $p=2$, many Frobenius-Schur indicators for highest weight modules can be found in \cite{mikko}.
We are ready to prove the main proposition of this section.

\textbf{Proof of Proposition~\ref{prop k=2 list of candidates restated}.}

\begin{proof}
Let $\lambda$ be the highest weight of $V$. Assume that $\lambda$ is $p$-restricted. 

Assume first that the rank of $H$ is at least $12$. By Lemma~\ref{k-spaces bound}, if $k\geq 5$ then $\dim V\leq \frac{4}{5}\dim H +2$. By \cite[Thm.~5.1]{lubeck}, except for the natural modules, there are no non-trivial self-dual modules of dimension at most $\frac{4}{5}\dim H +2$ for simple algebraic groups of rank at least $12$. Therefore $2\leq k\leq 4$, and since by Lemma~\ref{k-spaces bound} we have $\dim V\leq \dim H/k +(3k+1)/2$, also for these values of $k$ we have $\dim V\leq \frac{4}{5}\dim H +2$. There are therefore no candidates when the rank of $H$ is at least $12$. 

If $k\geq 5$ and $\dim H\geq 25$, by Lemma~\ref{k-spaces bound} we have $\dim V\leq \dim H-3$. If $2\leq k \leq 4$, by Lemma~\ref{k-spaces bound} we get $\dim V\leq \dim H/2 + 7$, which is smaller than $\dim H -3$ as long as $\dim H\geq 20$. Therefore, when $\dim H\geq 25$, we only need to be concerned with modules of dimension at most $ \dim H -3$. We now proceed by isogeny type of $H$, in the following manner:
\begin{enumerate}
    \item We deal with the case $\dim H\geq 25$ by looking for modules of dimension at most $\dim H -3$;
    \item For every case where $\dim H\leq 25$ we use Corollary~\ref{k bound} to determine a bound on $k$ and then for every $k$ we use Lemma~\ref{k-spaces bound} to determine a bound on $\dim V$;
    \item We find the self dual modules satisfying the bounds determined in Step $2$, by going through the lists in \cite{lubeck}.
\end{enumerate}

Let $H=A_l$. If $l\geq 5$, then $\dim H\geq 35$. In this case the only self-dual module of dimension at most $\dim H -3$ is $V_{A_5}(\lambda_3)$, of dimension $20$. Here $k\leq 10$ and by Lemma~\ref{k-spaces bound}, we get $\dim V\leq 21,16,15,15,16,16,17,19$  for $k$ between $2$ and $10$ respectively. So the only candidate for $l\geq 5$ is $V_{A_5}(\lambda_3)$ with $k=2$. We now assume that $l<5$ and proceed with Step $3$.
If $l=1$ and $k=3$ then $\dim V=6$. However in this case $A_1\leq Sp(V)$ and the dimension of the variety of totally singular $3$-spaces is $6$, by Proposition~\ref{dimension totally singular subspaces}. Hence the only candidates with $l=1$ are $V_{A_1}(3\lambda_1)$ and $V_{A_1}(4\lambda_1)$ with $k=2$.
If $l=2$ we get candidates $V_{A_2}(\lambda_1+\lambda_2)$ with $k=2,p=3$ or $k=3,p=3$ or $k=4,p\neq 3$.
If $l=3$ the only self-dual irreducible $A_3$-modules of dimension at most $12$ is $V_{A_3}(\lambda_2)$, which we discount for being equivalent to the natural module for $D_3$.
If $l=4$ there are no candidates within the given bounds. This concludes the analysis of $H=A_l$.

Let $H=B_l$. Disregarding the case $\lambda=\lambda_2$ in characteristic $p=2$, which is included in the $C_l$ case, the only non-trivial $B_l$-module of dimension at most $\dim H -3$, which is not the natural module, has highest weight $\lambda_l$ with $l\leq 6$. If $l\geq 7$, since $\dim B_l\geq 25$, we have therefore no candidates. 
If $l=2$, the only candidate is $\lambda=2\lambda_1$ and $k=5$. 
If $l=3$, then $\lambda=\lambda_3$ and $k\leq 4$.
This gives candidates $V_{B_3}(\lambda_3)$ with $k=2,3,4$. However here $B_3$ already has finitely many orbits on $P_2(V)$ and $P_3(V)$, by \cite[Thm.~2]{finite}.
If $l=4$, then $\lambda=\lambda_4$ with $k=2,3,7,8$.
If $l=5$ or $l=6$, there are no candidates. This concludes the analysis of $H=B_l$.

Let $H=C_l$. We disregard the spin module when $p=2$, since this is covered by the $B_l$ case. If $k\geq 5$ and $l\geq 5$, then $\dim V \leq \frac{4}{5}\dim H+2<2l^2-l-2$. Since $2l^2-l-2$ is the smallest dimension of a self-dual $C_l$-module that is neither the natural module nor a spin module, we have no candidates when $k,l\geq 5$. If $l\geq 5$ and $k\leq 4$, since $\dim V\leq \dim H/k +(3k+1)/2$, we have $\dim V\leq \dim H/2+7$, which implies that $\dim V< 2l^2-l-2$, as for the case $k,l\geq 5$. 
If $l=3$, we have $V_{C_3}(\lambda_2)$ with $k=2$ for all $p$ and $k=7$ for $p\neq 3$. If $l=4$, we have no candidates. This concludes the analysis for $H=C_l$.

Let $H=D_l$. We assume that $l\geq 4$, since $D_3\simeq A_3$. If $k\geq 5$ then $\dim V \leq \frac{4}{5}\dim H+2<2l^2-l-2$. Otherwise, if $k\leq 4$, since $\dim V\leq \dim H/k +(3k+1)/2$ we have $\dim V\leq \dim H/2+7$, which implies that $\dim V< 2l^2-l-2$, as for the previous case. Up to graph and field twists, the only self-dual $D_l$-modules of dimension smaller than $2l^2-l-2$ are the natural module and the spin module for $l=6$.
In the latter case, by Lemma~\ref{k-spaces bound} we find that $k=2$.
Therefore the only new $D_l$-candidate is $V_{D_6}(\lambda_6)$ with $k=2$.  

Let us consider the exceptional groups. In all cases except $H=G_2$, the dimension of $H$ is at least $25$, and therefore by our initial considerations in this proof, the only candidates are the minimal modules $V_{E_7}(\lambda_7)$ and $V_{F_4}(\lambda_4)$ or $V_{F_4}(\lambda_1)(p=2)$.
Finally, let $H=G_2$. Then by Corollary~\ref{k bound} we have $k^2-k\leq 28$ and therefore $k\leq 6$. By Lemma~\ref{k-spaces bound}, we get $\dim V\leq 10,9,10,10,11$ for $k=2,3,4,5,6$ respectively. This gives candidates $V_{G_2}(\lambda_1)$ for $k=2,3$. Note that by \cite[Thm.~2]{finite} the group $G_2$ has finitely many orbits on $P_k(V_{G_2}(\lambda_1))$ when $k=2$ and when $k=3,p=2$. Therefore the only new $G_2$-candidate is $V_{G_2}(\lambda_1)$ with $k=3,p\neq 2$.

Now assume that $\lambda$ is not $p$-restricted. Then the module $V_H(\lambda)$ factors as a tensor product of at least two modules that either were shown in the proof of this proposition to satisfy the dimension bound given by Lemma~\ref{k-spaces bound}, or are natural modules for $H$ with $H$ classical. It is easy to see that if one of the factors is not a natural module for $H$, then the dimension of $V$ violates the bound in Lemma~\ref{k-spaces bound}. The same is true for a tensor product of at least $2$ natural modules for $H$ of type $B_l,C_l$ or $D_l$. Finally for type $A_l$, the case $l=1$ gives us the candidate $\lambda_1+p^i\lambda_1$ for $k=2$ and the case $l\geq 2$ gives no candidates since $\lambda_1+p^i\lambda_1$ and $\lambda_1+p^i\lambda_l$ are not self dual. 

\end{proof}

\section{Proof of Theorem~\getrefnumber{main theorem k=2}}\label{proof of main theorem section}
In this section we prove Theorem~\ref{main theorem k=2}. By Proposition~\ref{prop k=2 list of candidates restated} it suffices to consider the modules listed in Table~\ref{tab:candidates higher dimensions restated} with $k=2$, and determine whether there are finitely many orbits on totally singular $2$-spaces. We divide the work into multiple subsections, one for every group $H$ of the $H$-modules in Table~\ref{tab:candidates higher dimensions restated} with $k=2$. In the cases $V_{C_3}(\lambda_2)(p\neq 3)$, $V_{A_5}(\lambda_3)$, $V_{F_4}(\lambda_4)(p\neq 3)$ and $V_{E_7}(\lambda_7)$ we prove the existence of infinitely many orbits on totally singular $2$-spaces, while we show that in all other cases there are finitely many orbits on totally singular $2$-spaces.

\subsection{$H$ of type $A_1$}\label{a1 section}
In this section we consider the cases involving $H=A_1$. Most of the work just follows from previous results on $1$-spaces. 

\begin{proposition}\label{A1 lambda1+p lambda1}
Suppose $p>0$ and let $V=V_H(\lambda_1+p^i\lambda_1)$, a $4$-dimensional orthogonal module. Then $H$ has two orbits on $P_2^{TS}(V)$, with stabilizers isomorphic to $U_1T_1$.
\end{proposition}
\begin{proof}
Let $q=p^i$ and let $\sigma=\sigma_q$ be the standard Frobenius morphism acting on $K$ as $t\rightarrow t^\sigma=t^q$ and consequently on $H$ as $x_{\pm \alpha_1}(t)\rightarrow x_{\pm \alpha_1}(t^q)$. Let $H=SL_2(K)$.
We can view $V$ as the space $M_{2\times 2}(K)$ of $2\times 2$ matrices on which $H$ acts by $g.v=g^Tvg^\sigma$ for $v\in M_{2\times 2}(K)$ and $g\in H$. Since $H$ preserves the determinant of $v$ for all $v\in M_{2\times 2}(K)$, we can take the quadratic form $Q:V\rightarrow K$ as $Q(v)=\det v$.
The totally singular $2$-spaces in $V$ are therefore the $2$-spaces consisting of matrices with determinant $0$. Consider the following totally singular $2$-spaces: $$W_2^{(1)}=\left\{\left(\begin{matrix} 0 & a\\ 0 & b \\ \end{matrix}\right) :a,b\in K\right\},\quad  W_2^{(2)}\left\{\left(\begin{matrix} a & b\\ 0 & 0 \\ \end{matrix}\right) :a,b\in K\right\}.$$ The Borel subgroup $B^+\leq H$ of upper triangular matrices fixes $W_2^{(1)}$, while the Borel subgroup $B^-\leq H$ of lower triangular matrices fixes $W_2^{(2)}$. By maximality of $B^+$ and $B^-$ in $H$, these are the stabilizers of the $2$-spaces. Let $T$ be the standard maximal torus of diagonal matrices. If $W_2^{(1)}$ and $W_2^{(2)}$ are in the same $N_H(T)$-orbit, then there exists an element $w\in N_H(T)$ such that $w.W_2^{(1)}=W_2^{(2)}$. However, it is immediate to see that $w.W_2^{(1)}$ is either $W_2^{(1)}$ or $$\left\{\left(\begin{matrix} a & 0\\ b & 0 \\ \end{matrix}\right) :a,b\in K\right\}.$$

Therefore the two $2$-spaces are in different $H$-orbits. We now pass to finite fields, where we use a counting argument on the sizes of the orbits to show that the two $H$-orbits that we have found are all of the $H$-orbits. 

Let $q'=p^j$ and let $\sigma'=\sigma_{q'}$ be the standard Frobenius morphism acting on $K$ as $t\rightarrow t^\sigma=t^{q'}$, on $H$ as $x_{\pm \alpha_1}(t)\rightarrow x_{\pm \alpha_1}(t^{q'})$ and in a compatible way on $V$. Since $W_2^{(1)}$ and $W_2^{(2)}$ have a basis in $V_{\sigma'}$, their $H$-orbit is $\sigma'$-stable. Since the stabilizers of $W_2^{(1)}$ and $W_2^{(2)}$ are connected, by Lang-Steinberg we get that $H_{\sigma'}$ has two orbits on $V_{\sigma'}$, with stabilizers isomorphic to $[q'].(q'-1)$. The union of the two orbits contains $2|SL_2(q')|/(q'(q'-1))=2(q'+1)$ elements, which by Proposition~\ref{number of totally singular subspaces symplectic orthogonal} is $|P_2^{TS}(V_{\sigma'})|$. There are therefore precisely two $H$-orbits on $P_2^{TS}(V)$.
\end{proof}

\begin{proposition}\label{a1 3 and 4 prop}
Suppose $p\geq 5$. Then the $H$-modules $V_H(3\lambda_1)$ and $V_H(4\lambda_1)$ in characteristic $p\neq 2,3$ have finitely many orbits on totally singular $2$-spaces. The generic stabilizers are $Alt_4$ and $Sym_3$ respectively. 
\end{proposition}

\begin{proof}
Note that in the first case $A_1\leq Sp_4$, and in the second case $A_1\leq SO_5$. An isomorphism between $C_2$ and $B_2$ sends $A_1\rightarrow A_1$ and swaps the $P_1$ and the $P_2$ parabolics. Therefore the statement is equivalent to $V_H(3\lambda_1)$ being a finite orbit module, and $V_H(4\lambda_1)$ being a finite singular orbit module. By \cite[Thm.~1]{rizzoli} the module $V_{H}(4\lambda_1)$ is a finite singular orbit module, while $V_H(3\lambda_1)$ is an internal module by \cite[Table~1]{finite}. 

The generic stabilizer for the action on singular $1$-spaces of $V_{H}(4\lambda_1)$ is $Alt_4$ by \cite[Thm.~4.1]{rizzoli}, while \cite[Prop.~3.1.6]{generic1} gives us the generic stabilizer $Sym_3$ for the action on $1$-spaces of $V_{H}(3\lambda_1)$. Orbits and stabilizers for the first case can be found in the proof of \cite[Thm.~1]{rizzoli}.
\end{proof}

\subsection{$H$ of type $A_2$ and $V=V_{H}(\lambda_1+\lambda_2)$}\label{a2 section}

Let $H=A_2$ in characteristic $p=3$. In this section we prove the following proposition.
\begin{proposition}\label{proposition a2 k=2}
Let $V=V_H(\lambda_1+\lambda_2)$, a $7$-dimensional orthogonal module. Then $H$ has $7$ orbits on $P_2^{TS}(V)$. Stabilizers and representatives can be found in Table~\ref{tab:a2 reps with stabs}
\end{proposition}

The strategy to prove this proposition is to determine a list of orbit-representatives and corresponding stabilizers, and show that it is a complete list by descending to finite fields. This is done by using the information about the stabilizers in the algebraic case to determine a list of orbit sizes over finite fields, and using a counting argument to show that we have indeed found all the orbits.

Let $\alpha_1,\alpha_2$ be the fundamental roots for $A_2$ and let $\alpha_3=\alpha_1+\alpha_2$. The adjoint module $\mathrm{Lie}(H)$ has the Chevalley basis $e_{\pm\alpha_1},e_{\pm \alpha_2},e_{\pm\alpha_3},h_{\alpha_1},h_{\alpha_2}$. We write $v_1v_2$ for the Lie product of vectors $v_1,v_2\in \mathrm{Lie}(H)$.

We order the negative roots so that $-\alpha_1=\alpha_4$,   $-\alpha_2=\alpha_5$ and  $-\alpha_3=\alpha_6$. 

We now give the table $(N_{ij})$ of the structure constants $N_{\alpha_i,\alpha_j}$ that we will be using, which match the ones in Magma. Note that this is uniquely determined once the sign of $N_{1,2}$ is chosen (see \cite[Thm.~4.2.1]{cartersimple}). 

$$N_{ij}=\left(\begin{matrix} 0 & 1 & 0 &0 & 0 & -1 \\ -1 & 0 & 0 &0 & 0 & 1 \\0 & 0 & 0 &-1 & 1 & 0\\  0 & 0 & 1 &0 & -1 & 0 \\ 0 & 0 & -1 &1 & 0 & 0\\ 1 & -1 & 0 &0 & 0 & 0\\ \end{matrix}\right)   .$$ 

We first note the following.

\begin{lemma}
In the $H$-action on $\mathrm{Lie}(H)$, the element $h_{\alpha_1}-h_{\alpha_2}$ is fixed by $H$.
\end{lemma}

\begin{proof}
It suffices to show that for all $t\in K$ all the elements $x_{\pm \alpha_1}(t),x_{\pm \alpha_2}(t)$ fix $h_{\alpha_1}-h_{\alpha_2}$. We only show it for $g=x_{\alpha_1}(t)$, since the other cases are very similar. By \cite[§4.4]{cartersimple}, we know that $g.h_{\alpha_1}=h_{\alpha_1}-2te_{\alpha_1}=h_{\alpha_1}+te_{\alpha_1}$, as $p=3$. Also, since $h_{\alpha_2}=e_{\alpha_2}e_{-\alpha_2}$, we find that $g.h_{\alpha_2}=(g.e_{\alpha_2})(g.e_{-\alpha_2})=(e_{\alpha_2}+te_{\alpha_3})e_{-\alpha_2}=h_{\alpha_2}+te_{\alpha_1}$. Therefore $g=x_{\alpha_1}(t)$ fixes $h_{\alpha_1}-h_{\alpha_2}$. The other cases follow similarly.
\end{proof}

We can now explicitly construct our highest weight irreducible module as:
$$V_{H}(\lambda_1+\lambda_2)= \mathrm{Lie}(H)/\langle h_{\alpha_1}-h_{\alpha_2} \rangle.$$

In a slight abuse of notation we omit writing the quotient, so that $v$ actually stands for $v+\langle h_{\alpha_1}-h_{\alpha_2} \rangle$. We  order the basis for $V_{H}(\lambda_1+\lambda_2 )$ as $e_{\alpha_i},h_{\alpha_1}$, for $1\leq i\leq 6$. With respect to this ordering, using standard formulas found in \cite[§4.4]{cartersimple}, we find the matrices denoting the transformations $x_{\pm \alpha_1}(t),x_{\pm \alpha_2}(t),x_{\pm \alpha_3}(t)$, as well as $h_{\alpha_1}(\kappa)$ and $h_{\alpha_2}(\kappa)$. These are straightforward calculations and we therefore only state the results.

$$x_{\alpha_1}(t):\left(\begin{matrix} 
1& 0& 0& -t^2& 0& 0& t\\
 0& 1& 0& 0& 0& 0& 0\\
 0& t& 1& 0& 0& 0& 0\\
 0& 0& 0& 1& 0& 0& 0\\
 0& 0& 0& 0& 1& -t& 0\\
 0& 0& 0& 0& 0& 1& 0\\
 0& 0& 0& t& 0& 0& 1\\
\end{matrix}\right),\quad   
x_{\alpha_2}(t):\left(\begin{matrix} 
1& 0& 0& 0& 0& 0& 0\\
 0& 1& 0& 0& -t^2& 0& t\\
 -t& 0& 1& 0& 0& 0& 0\\
 0& 0& 0& 1& 0& t& 0\\
 0& 0& 0& 0& 1& 0& 0\\
 0& 0& 0& 0& 0& 1& 0\\
 0& 0& 0& 0& t& 0& 1\\
\end{matrix}\right) $$ 

$$ x_{\alpha_3}(t):\left(\begin{matrix} 
1& 0& 0& 0& t& 0& 0\\
 0& 1& 0& -t& 0& 0& 0\\
 0& 0& 1& 0& 0& -t^2& -t\\
 0& 0& 0& 1& 0& 0& 0\\
 0& 0& 0& 0& 1& 0& 0\\
 0& 0& 0& 0& 0& 1& 0\\
 0& 0& 0& 0& 0& -t& 1\\
\end{matrix}\right),\quad 
x_{-\alpha_1}(t):\left(\begin{matrix} 
1& 0& 0& 0& 0& 0& 0\\
 0& 1& t& 0& 0& 0& 0\\
 0& 0& 1& 0& 0& 0& 0\\
 -t^2& 0& 0& 1& 0& 0& -t\\
 0& 0& 0& 0& 1& 0& 0\\
 0& 0& 0& 0& -t& 1& 0\\
 -t& 0& 0& 0& 0& 0& 1\\
\end{matrix}\right)   
$$

$$
x_{-\alpha_2}(t):\left(\begin{matrix} 
1& 0& -t& 0& 0& 0& 0\\
 0& 1& 0& 0& 0& 0& 0\\
 0& 0& 1& 0& 0& 0& 0\\
 0& 0& 0& 1& 0& 0& 0\\
 0& -t^2& 0& 0& 1& 0& -t\\
 0& 0& 0& t& 0& 1& 0\\
 0& -t& 0& 0& 0& 0& 1\\
\end{matrix}\right) ,\quad 
x_{-\alpha_3}(t):\left(\begin{matrix} 
1& 0& 0& 0& 0& 0& 0\\
 0& 1& 0& 0& 0& 0& 0\\
 0& 0& 1& 0& 0& 0& 0\\
 0& -t& 0& 1& 0& 0& 0\\
 t& 0& 0& 0& 1& 0& 0\\
 0& 0& -t^2& 0& 0& 1& t\\
 0& 0& t& 0& 0& 0& 1\\
\end{matrix}\right),
$$ 

$h_{\alpha_1}(\kappa)= \mathrm{diag}(\kappa^2,\kappa^{-1},\kappa,\kappa^{-2},\kappa,\kappa^{-1},1)$ and finally $h_{\alpha_2}(\kappa)= \mathrm{diag}(\kappa^{-1},\kappa^2,\kappa,\kappa,\kappa^{-2},\kappa^{-1},1)$.

Let $(\cdot,\cdot):V\times V \rightarrow K$ be the non-degenerate symmetric bilinear form given by $(e_{\alpha_i},e_{\alpha_j})=1$ if and only if $|i-j|=3$, as well as $(h_{\alpha_1},e_{\alpha_i})=0$ and $(h_{\alpha_1},h_{\alpha_1})=-1$. Then $H$ fixes this form, as can be seen by just checking the action of the generators. 

Before analysing the action of $H$ on $P_2^{TS}(V)$, we need some information about the action of $H$ on singular $1$-spaces. Let $T$ be the standard maximal torus $T=\langle h_{\alpha_1}(\kappa),h_{\alpha_2}(\kappa):\kappa\in K^*\rangle$ and $B$ the Borel subgroup $B=\langle T,x_{\alpha_1}(t),x_{\alpha_2}(t):t\in K\rangle$.

\begin{lemma}
The group $H$ has $2$ orbits on singular vectors in $V=V_H(\lambda_1+\lambda_2)$, with representatives $x=e_{\alpha_3}$ and $y=e_{\alpha_1}+e_{\alpha_2}$. Furthermore $H_x=U_3T_1$, $H_y=U_2$, $H_{\langle x \rangle}=U_3T_2=B$ and $H_{\langle y \rangle}=U_2T_1\leq B$.
\end{lemma}

\begin{proof}
Let $\mathfrak{sl}_3(K)$ be the Lie algebra of trace zero $3\times 3$ matrices over $K$. Then $$\mathfrak{sl}_3(K)/\langle I \rangle$$ is isomorphic to $V$ as an $H$-module, with the action given by conjugation by $h\in H=SL_3(K)$. In this different characterization of $V$, the basis elements $e_{\alpha_1},e_{\alpha_2},e_{\alpha_3}$ correspond respectively to elements $e_{1,2},e_{2,3},e_{1,3}\in \mathfrak{sl}_3(K)$, where $e_{i,j}$ denotes a $3\times 3$ matrix with a $1$ in position $i,j$ and $0$ everywhere else. The elements $x_{\alpha_1}(t),x_{\alpha_2}(t)\in SL_3(K)$ are respectively the matrices $I+te_{1,2},I+te_{2,3}$. Finding the stated stabilizers of vectors and $1$-spaces in $V$ is then just a simple linear algebra exercise.
More precisely $$H_x=\{x_{\alpha_1}(r)x_{\alpha_2}(s)x_{\alpha_3}(t)h_{\alpha_1}(\kappa)h_{\alpha_2}(\kappa^{-1}):r,s,t\in K, \kappa\in K^*\}$$ $$H_y=\{x_{\alpha_1}(t)x_{\alpha_2}(t)x_{\alpha_3}(s):t,s\in K\}$$ and a $T_1$ in $H_{\langle y\rangle}$ is given by $ \{h_{\alpha_1}(\kappa)h_{\alpha_2}(\kappa):\kappa\in K^*\}$. By passing to a finite field $\mathbb{F}_q$, the union of the two orbits on singular vectors contains $|SL_3(q)|/(q^3(q-1))+|SL_3(q)|/(q^2)=q^6-1$ vectors. Proposition~\ref{number of totally singular subspaces symplectic orthogonal} tells us that there are $(q^6-1)/(q-1)$ singular $1$-spaces, and therefore $q^6-1$ singular vectors. Hence we do indeed have a complete list of orbit representatives. 
\end{proof}

We now define a list of totally singular $2$-spaces (Table~\ref{tab:a2 reps with stabs}), which we will prove forms a complete set of representatives for the action of $H$ on $P_2^{TS}(V)$. Note that they are all easily seen to be totally singular, using the described quadratic form. 

\begin{center}
\begin{longtable}{l l l}
\caption{List of totally singular $2$-spaces of $V$} \label{tab:a2 reps with stabs} \\

\hline \multicolumn{1}{c}{Index $i$} & \multicolumn{1}{c}{ $ {W_2^{(i)}}$ }& \multicolumn{1}{c}{ $ H_{W_2^{(i)}}$ } \\ \hline 
\endhead
1& $ \langle y, e_{\alpha_1} -e_{-\alpha_1}+e_{-\alpha_2}+h_{\alpha_1}\rangle $ & $U_1$ \\*

  2 &$ \langle y, e_{-\alpha_1}-e_{-\alpha_2}\rangle$ & $A_1$ \\*

  3& $ \langle y, e_{-\alpha_3}\rangle$ & $U_1T_1$\\*

   4& $ \langle e_{\alpha_1},e_{\alpha_2}\rangle$  & $U_1T_2$  \\*

   5&$\langle x,y\rangle $& $U_3T_1$ \\*
 
   6&$\langle x, e_{\alpha_2}\rangle $ & $U_2A_1T_1$  \\*

   7&$\langle x, e_{\alpha_1}\rangle $ & $U_2A_1T_1$ \\
\hline
\end{longtable}
\end{center}

We also list the stabilizers, which we will now prove are as claimed.
We divide the work for the calculations up into the following lemmas.

\begin{lemma}
The stabilizer of $W_2^{(1)}$ is a connected unipotent group of dimension 1.
\end{lemma}

\begin{proof}
Let $v=e_{\alpha_1} -e_{-\alpha_1}+e_{-\alpha_2}+h_{\alpha_1}$, so that $W_2^{(1)}=\langle y,v\rangle$. We first show that the stabilizer of $\langle v\rangle $ in $ H_{\langle y \rangle}$ is trivial. 

Let $g=h_{\alpha_1}(\kappa)h_{\alpha_2}(\kappa)x_{\alpha_1}(t)x_{\alpha_2}(t)x_{\alpha_3}(s)\in H_{\langle y \rangle}\cap H_{\langle v \rangle}$. Using the matrix description of the action of the generators of $H$ on $V$, while setting $A=s+t-t^2$, we determine that 
$$ x_{\alpha_1}(t)x_{\alpha_2}(t)x_{\alpha_3}(s).  (e_{\alpha_1} -e_{-\alpha_1}+e_{-\alpha_2}+h_{\alpha_1})= (1 + A)e_{\alpha_1}+  Ae_{\alpha_2}+(-A-t^3)e_{\alpha_3}  -e_{-\alpha_1}+e_{-\alpha_2}+h_{\alpha_1}.$$
Even considering the scaling induced by $h_{\alpha_1}(\kappa)h_{\alpha_2}(\kappa)$, we require $A=0$ and $t^3=0$, which means that $s=t=0$. Therefore $\kappa=1$, and we conclude that $g=1$.

We just saw that if $g=x_{\alpha_1}(t)x_{\alpha_2}(t)x_{\alpha_3}(s)\in H_y$ and $A=s+t-t^2$, then $g.v= v+ Ay-(A+t^3)e_{\alpha_3}$. Therefore if $t^3=-s-t+t^2$, we get $g.v=v+Ay$. Since $s$ is arbitrary so is $A$, and we have a $U_1$ action on $W_2^{(1)}$. If the stabilizer of $W_2^{(1)}$ induced any further action, then there would be an element in the stabilizer acting diagonally on $y$ and $v$. However this is impossible since $H_{\langle y \rangle}\cap H_{\langle v \rangle }=1$. Therefore $H_{W_2^{(1)}}=U_1$.
\end{proof}

For later use the following lemma finds a particular $2$-space in the same orbit as $W_2^{(1)}$.

\begin{lemma}\label{equivalent 2-space with u1 stab a2 k=2}
The $2$-space $W_2^{(*)}=\langle e_{\alpha_1}+e_{\alpha_2}+e_{\alpha_3},e_{-\alpha_2}-e_{-\alpha_1}\rangle$ is in the same $H$-orbit as $W_2^{(1)}$.
\end{lemma}

\begin{proof}
Let $g=x_{\alpha_1}(1)$ and let $v=e_{\alpha_1} -e_{-\alpha_1}+e_{-\alpha_2}+h_{\alpha_1}$, so that $W_2^{(1)}=\langle y,v\rangle$.
Then $g.y=g.(e_{\alpha_1}+e_{\alpha_2})=e_{\alpha_1}+e_{\alpha_2}+e_{\alpha_3}$, and $g.v=e_{-\alpha_2}-e_{-\alpha_1}$. Therefore $W_2^{(*)}$ is in the same $H$-orbit as $W_2^{(1)}$.
\end{proof}

\begin{lemma}
The stabilizer of $W_2^{(2)}$ is isomorphic to a group of type $A_1$.
\end{lemma}

\begin{proof}
Let $v=e_{-\alpha_1}-e_{-\alpha_2}$, so that $W_2^{(2)}=\langle y,v\rangle$. We first show that the pointwise stabilizer $ H_{\langle y\rangle}\cap H_{\langle v\rangle} $ is trivial. Let $g=h x_{\alpha_1}(t)x_{\alpha_2}(t)x_{\alpha_3}(s)\in H_{\langle y\rangle}\cap H_{\langle v\rangle} $, for $h=h_{\alpha_1}(\kappa)h_{\alpha_2}(\kappa)$. Using the matrix description of the action of the generators of $H$ on $V$, while setting $A=-s+t^2$, we determine that 
$$ g.  (e_{-\alpha_1}-e_{-\alpha_2})= \kappa^2 Ae_{\alpha_1}+  \kappa^{-1} Ae_{\alpha_2}+\kappa t^3e_{\alpha_3}  +\kappa^{-2}e_{-\alpha_1}-\kappa e_{-\alpha_2}.$$
We therefore require $A=0$ and $t^3=0$, which means that $s=t=0$, and $\kappa=1$. Hence $g=1$ and $H_{\langle y\rangle}\cap H_{\langle v\rangle} = 1$.

The $A_1$ subgroup $\langle X_{\pm \alpha_3} \rangle=\langle x_{\pm \alpha_3}(t):t\in K \rangle$ fixes $W_2^{(2)}$. This is therefore the full stabilizer, i.e. $H_{W_2^{(2)}}=A_1$.
\end{proof}

\begin{lemma}
The stabilizer of $W_2^{(3)}$ is isomorphic to $U_1T_1$.
\end{lemma}

\begin{proof}
Let $v=e_{-\alpha_3}$, so that $W_2^{(3)}=\langle y,v\rangle$. The element $x_{-\alpha_1}(t)x_{-\alpha_2}(-t)x_{-\alpha_3}(t^2)$ sends $y$ to $y+t^3v$. Therefore $ W_2^{(3)}$ contains a unique point in the $H$-orbit with representative $x$, i.e. the point $\langle v \rangle =\langle e_{-\alpha_3} \rangle$. This means that the stabilizer of $W_2^{(3)}$ is contained in the stabilizer of $\langle v \rangle$, which is the Borel subgroup $B^-$, opposite to $B$. Let $$g= h_{\alpha_1}(\kappa)h_{\alpha_2}(\kappa')x_{-\alpha_1}(r)x_{-\alpha_2}(s)x_{-\alpha_3}(t) $$ so that $g\in H_{\langle v \rangle}=B^-$. Then, $$x_{-\alpha_1}(r)x_{-\alpha_2}(s)x_{-\alpha_3}(t).y= y+(-r^2 + r s - t)e_{-\alpha_1}+(-s^2 + t)e_{-\alpha_2}+(r s^2 - r t - s t)e_{-\alpha_3}+(-r-s)h_{\alpha_1} .$$ 
If $g$ fixes $W_2^{(3)}$, then $r=-s$ and $t=s^2$. Therefore $g.y=h_{\alpha_1}(\kappa)h_{\alpha_2}(\kappa').y+\lambda v$, for some $\lambda\in K$, and since $h_{\alpha_1}(\kappa)h_{\alpha_2}(\kappa')$ scales the summands $e_{\alpha_1},e_{\alpha_2}$ of $y$, we require $ \kappa=\kappa'$. This allows us to conclude that $H_{W_2^{(3)}}=U_1T_1$.

\end{proof}

\begin{lemma}
The stabilizer of $W_2^{(4)}$ is isomorphic to $U_1T_2$.
\end{lemma}

\begin{proof}
By definition $W_2^{(4)}=\langle e_{\alpha_1},e_{\alpha_2} \rangle$. Let $\lambda\in K^*$. Then $\langle e_{\alpha_1}+\lambda e_{\alpha_2} \rangle$ is in the $H$-orbit with representative $\langle y \rangle=\langle e_{\alpha_1}+e_{\alpha_2} \rangle $, and $W_2^{(4)}$ contains precisely two points in the orbit with representative $\langle x \rangle$. Therefore the stabilizer of $W_2^{(4)}$ either fixes or swaps both of these points. Let $g\in H_{\langle e_{\alpha_1}\rangle}\cap H_{\langle e_{\alpha_2}\rangle}$. Then there exists $h\in T$ such that $hg\in H_{\langle y\rangle}$, since $y=e_{\alpha_1}+e_{\alpha_2}$. The same holds if $g$ swaps $ \langle e_{\alpha_1}\rangle$ and $\langle e_{\alpha_2}\rangle $. Therefore $H_{W_2^{(4)}}\leq  B$. Since $T\leq H_{W_2^{(4)}}$ we only need to find which elements of the form $x_{\alpha_1}(r)x_{\alpha_2}(s)x_{\alpha_3}(t) $ fix $W_2^{(4)}$. We find that $$x_{\alpha_1}(r)x_{\alpha_2}(s)x_{\alpha_3}(t).e_{\alpha_1}=e_{\alpha_1}-se_{\alpha_3};$$
$$x_{\alpha_1}(r)x_{\alpha_2}(s)x_{\alpha_3}(t).e_{\alpha_2}=e_{\alpha_1}+re_{\alpha_3}.$$
Therefore $r=s=0$ and $t$ is arbitrary. This shows that $H_{W_2^{(4)}}=U_1T_2$.
\end{proof}

\begin{lemma}
The stabilizer of $W_2^{(5)}$ is isomorphic to $U_3T_1$.
\end{lemma}

\begin{proof}
The element $x_{\alpha_1}(t)$ sends $y$ to $y+tx$.
The point $\langle x \rangle$ is therefore  the unique point in its $H$-orbit in $ \langle x,y\rangle$. Therefore $ H_{W_2^{(5)}}\leq B$. Since all the positive root subgroups fix $W_2^{(5)}$, and $T_{W_2^{(5)}}=\langle h_{\alpha_1}(\kappa)h_{\alpha_2}(\kappa):\kappa\in K^*\rangle $, we conclude that $H_{W_2^{(5)}}=U_3T_1$.
\end{proof}

\begin{lemma}
The stabilizer of $W_2^{(6)}$ is the standard maximal parabolic $P_2=U_2A_1T_1$.
\end{lemma}

\begin{proof}
It suffices to observe that the generators $T,x_{\pm \alpha_1}(t),x_{\alpha_2}(t),x_{\alpha_3}(t)$ of $P_2$ fix $ W_2^{(6)}$.
\end{proof}

\begin{lemma}
The stabilizer of $W_2^{(7)}$ is the standard maximal parabolic $P_1=U_2A_1T_1$.
\end{lemma}

\begin{proof}
It suffices to observe that the generators $T,x_{\alpha_1}(t),x_{\pm\alpha_2}(t),x_{\alpha_3}(t)$ of $P_1$ fix $ W_2^{(7)}$.
\end{proof}

We have therefore shown that the stabilizers are as in Table~\ref{tab:a2 reps with stabs}. We are now ready to prove Proposition~\ref{proposition a2 k=2}.

\renewcommand*{\proofname}{Proof of Proposition~\ref{proposition a2 k=2}.}

\begin{proof}
We use a counting argument. Let $q=3^e$ and let $\sigma=\sigma_{q}$ be the standard Frobenius morphism acting on $K$ as $t\rightarrow t^\sigma=t^{q}$, on $H$ as $x_{\pm \alpha_i}(t)\rightarrow x_{\pm \alpha_i}(t^{q})$ and in a compatible way on $V$. Since for all $1\leq i\leq 7$ the $2$-space $W_2^{(i)}$ has a basis in $V_{\sigma}$, its $H$-orbit is $\sigma$-stable. Since the stabilizers of the $W_2^{(i)}$ are connected, by Lang-Steinberg we get seven $H_{\sigma}$-orbits on $V_{\sigma}$, with stabilizers as in Table~\ref{tab:a2 orbits finite field}.

\begin{center}
\begin{xltabular}[h]{\textwidth}{l l l }
\caption{$A_2(q)$-orbits on totally singular $2$-spaces in $V(q)$} \label{tab:a2 orbits finite field} \\
\hline \multicolumn{1}{c}{Orbit number } & \multicolumn{1}{c}{$|A_2(q)_{W_2(q)}|$ } & \multicolumn{1}{c}{$|W_2(q)^{A_2(q)}|$}\\
\hline 
\endhead
1& $q$ & $q^7-q^5-q^4+q^2 $\\

 2&$q(q-1)(q+1)$  & $q^2 (q^3-1)$\\
 
  3 &$q(q-1)$   &$q^2 (q^4+q^3-q-1)$  \\

  4& $q(q-1)^2$ & $q^2 (1 + q) (1 + q + q^2)$ \\

   5& $q^3(q-1)$  &  $q^4+q^3-q-1$\\

   6&$q^3(q-1)^2(q+1)$  &   $1 + q + q^2$ \\

    7&$q^3(q-1)^2(q+1)$  &  $1 + q + q^2$\\
\hline 
\end{xltabular}
\end{center}
\vspace{-1cm}
The sizes of the orbits add up to $q(1 + q) (1 + q^2) (1 - q + q^2) (1 + q + q^2)$, which is the number of totally singular $2$-spaces by Proposition~\ref{number of totally singular subspaces symplectic orthogonal}. We have therefore found a complete list of orbit representatives.
\end{proof}
\renewcommand*{\proofname}{Proof.}

 \subsection{$H$ of type $C_3$ and $V=V_{H}(\lambda_2)$}\label{c3 section}
Let $H=C_3$ and $V=V_H(\lambda_2)$, an orthogonal module of dimension $14-\delta_{p,3}$. In this section we prove the following proposition.

\begin{proposition}\label{c3 k=2 main prop}
The group $H$ has $12$-orbits on $P_2^{TS}(V)$ if $p=3$ and no dense orbit otherwise. Representatives and stabilizers for the $p=3$ case can be found in Table~\ref{tab:c3 orbits}.
\end{proposition}

By \cite[Thm.~2]{generic2} the action of $H$ on $P_2(V)$ has a generic stabilizer with connected component $T_1$. By Corollary~\ref{minimum dimension generic} the stabilizer of any $2$-space of $V$ is at least $1$-dimensional. If $p\neq 3$ we have $\dim P_2^{TS}(V)=21=\dim H$, which implies that there is no dense orbit on $P_2^{TS}(V)$. To prove Proposition~\ref{c3 k=2 main prop} we let $p=3$, and show that the group $H$ has $12$ orbits on $P_2^{TS}(V)$.

Let $V_6$ be the natural module for $C_3$, with an ordered basis $(e_1,e_2,e_3,f_3,f_2,f_1)$ where $(e_i,f_i)$ is a hyperbolic pair. Let $$V_{15} = \Lambda^2 V_6.$$ Then $V_{15}$ has a $C_3$-submodule $V_{14}$ given by
$$V_{14}=\langle e_i\wedge e_j,f_i\wedge f_j,e_i\wedge f_j, \sum \alpha_{i}e_i\wedge f_i : i\neq j, \sum \alpha_i=0 \rangle.$$ 
Since $p=3$, the vector $e_1\wedge f_1+ e_2\wedge f_2 + e_3 \wedge f_3 $ is in $V_{14}$ and is fixed by $C_3$. Then by \cite[Table~4.5]{cline} the irreducible module $V$ is obtained as $$V=V_{14}/\langle e_1\wedge f_1+ e_2\wedge f_2 + e_3 \wedge f_3\rangle .$$  Let $(v_1,\dots ,v_{13})$ be the ordered basis for $V$ given by 
\begin{flalign*}
 v_1&= e_1\wedge e_2, & v_5&=e_2\wedge f_3, & v_9&=e_3\wedge f_2, &\\
 v_2&=e_1\wedge e_3, & v_6&=e_1 \wedge f_2, &v_{10}&=e_3\wedge f_1, &\\
 v_3&=e_2\wedge e_3, &v_7&=e_1\wedge f_1 -e_2\wedge f_2, & v_{11}&=f_2\wedge f_3,& \\
  v_4&=e_1\wedge f_3, & v_8&=e_2\wedge f_1, & v_{12}&=f_1\wedge f_3,& \\
  & & & & v_{13}&=f_1\wedge f_2, \\
 \end{flalign*}

where by $v_i$ we actually mean $v_i+\langle e_1\wedge f_1+ e_2\wedge f_2 + e_3 \wedge f_3\rangle$. Then $C_3$ fixes a non-degenerate quadratic form on $V$, given by $$Q\left(\sum_1^{13}\alpha_i v_i\right)= \sum_1^7 \alpha_i \alpha_{14-i}.$$ In the following lemma we provide a characterization of the $C_3$-orbits on singular vectors:

\begin{lemma}\label{1-spaces c3}
The following is a complete list of representatives for the action of $C_3$ on singular vectors and singular $1$-spaces in $V$.
\begin{table}[h]
\begin{center}
 \begin{tabular}{c c c c } 
 \hline
 Orbit number & Vector $x$&  $H_{x}$ & $H_{\langle x \rangle} $\\ [0.5ex] 
 \hline
 1& $ v_1 $ & $U_7A_1A_1$  & $U_7A_1A_1T_1=H_{\langle e_1,e_2 \rangle }$\\

   2& $v_2+v_5 $ & $U_6A_1$  & $U_6A_1T_1 \leq H_{\langle e_1,e_2 \rangle }$ \\
 \hline
 
\end{tabular}
\end{center} 
\caption{$C_3$-orbits on singular $1$-spaces}
\label{tab:c3 orbits 1 spaces}
\end{table}
\end{lemma}

\begin{proof}
Let $x=v_1=e_1\wedge e_2$. Then $g\in H$ fixes $\langle x \rangle$ if and only if it stabilises $\langle e_1, e_2 \rangle$. Note that the stabilizer of $\langle e_1, e_2 \rangle$ is the maximal parabolic $P_2\leq H$, which is isomorphic to $U_7A_1A_1T_1$.
To fix $x$ the induced action on $\langle e_1, e_2 \rangle$ has to have determinant $1$, giving $H_x=U_7A_1A_1\leq P_2$.

Let $x=v_2+v_5=e_1\wedge e_3+e_2\wedge f_3$ and $g\in H_{\langle x \rangle}$. Then $g$ needs to stabilise the support $\langle e_1,e_2,e_3,f_3 \rangle$, as well as $\langle e_1,e_2,e_3,f_3 \rangle^\perp = \langle e_1,e_2 \rangle$. So in particular we have $H_{\langle x\rangle}\leq P_2 $. Let $g.e_1=\alpha e_1+\beta e_2$ and $g.e_2=\gamma e_1+\delta e_2$. Then $$g.x=e_1\wedge (\alpha g.e_3+\gamma g.f_3)+e_2\wedge (\beta g.e_3+\delta g.f_3)=x.$$ This implies that $g.(\alpha e_3+\gamma f_3)=e_3+\lambda e_1$ and $g.(\beta e_3+\delta f_3)=f_3+\lambda' e_2$ for some $\lambda,\lambda'\in K$. This is enough to show that $g \in U_6 A_1T_1 \leq P_2 $, where $A_1$ is generated by short root subgroups. The stabilizer of $x$ is similarly obtained.

To conclude we pass to finite fields $\mathbb{F}_q$, where $q=3^a$ for some $a\in\mathbb{N},$ and add up the sizes of the two orbits. We get $|Sp_6(q)|/(q^7(q^3-q)^2(q-1))+|Sp_6(q)|/(q^6(q^3-q)(q-1))=(q^{12}-1)/(q-1)$, which is the number of singular $1$-spaces. Therefore we only have two $C_3$-orbits on singular $1$-spaces of $V$.
\end{proof}

In Table~\ref{tab:c3 orbits first list} we give a list of $12$ totally singular $2$-spaces $W_2^{(i)}$ . We will prove that it is a complete list of representatives for the $H$-action on $P_2^{TS}(V)$. 

\begin{table}[h]
\begin{center}
 \begin{tabular}{c c  } 
 \hline
 Identifier $i$& $W_2^{(i)}$ \\ [0.5ex] 
 \hline
 1& $ \langle v_1,v_2 \rangle $ \\
 
   2& $\langle v_1,v_8 \rangle$  \\
 
 3 &$\langle v_1,v_2+v_5 \rangle$    \\ 

   4& $\langle v_1,v_2 +v_8\rangle$  \\

   5&$\langle v_1,v_{12} \rangle$  \\

   6&$\langle v_1,v_2+v_{11} \rangle$  \\

  7&$\langle v_1,v_8+v_9 \rangle$  \\

    8&$\langle v_1,v_9 +v_{12}\rangle$   \\

   9&$\langle v_2+v_5,v_{3}+v_6 \rangle$  \\

   10&$\langle v_2+v_5,v_{12}-v_{9} \rangle$  \\

  11&$\langle v_2+v_5,v_4+v_{13} \rangle$  \\

  12&$\langle v_2+v_5,v_5+v_{13} \rangle$  \\
 \hline
 
\end{tabular}
\end{center} 
\caption{List of totally singular $2$-spaces}
\label{tab:c3 orbits first list}
\end{table}

For each given $2$-space $x=W_2^{(i)}=\langle a,b \rangle$ we are first going to determine the centralizer $H_a\cap H_b$, and we will then find $H_x/(H_a\cap H_b)\leq GL_2=A_1T_1$. We take $H=C_3=Sp_6$, keeping in mind that $Z(Sp_6)=-I$ acts trivially on $V$. Let $T$ be the maximal torus of diagonal matrices in $H$.

To better understand the computations in the upcoming lemmas, let us look at the general form of an element $z\in H_{e_1\wedge e_2}$. This is given by $$z=\left(\begin{matrix}a & b & * & * & * & * \\ c & d & * & * & * & * \\ 0 & 0 & g & h & * & * \\ 0 & 0 & i & l & * & *\\ 0 & 0 & 0 & 0 & a & -b \\ 0 & 0 & 0 & 0 & -c & d \\ \end{matrix}\right),$$ such that $z$ preserves the alternating form on $V_6$ (i.e. $z\in H$) and $ad-bc=gl-hi=1$. We are not going to provide explicit calculations for most of the upcoming lemmas, except for the most difficult ones (for example Lemma~\ref{more difficult c3 lemma}). 

\begin{lemma}
The stabilizer of $W_2^{(1)}$ is isomorphic to  $U_8A_1T_2$.
\end{lemma}

\begin{proof}
Let $x= W_2^{(1)}=\langle v_1,v_2 \rangle$. By Lemma~\ref{1-spaces c3}, we have $H_{v_1} = H_{\langle e_1, e_2 \rangle }/T_1$ and $H_{v_2} = H_{\langle e_1, e_3 \rangle }/T_1$, giving $$H_{v_1}\cap H_{v_2} = \left\{\left(\begin{matrix}a & b & * & * & * & * \\ 0 & a^{-1} & 0 & * & * & * \\ 0 & 0 & a^{-1} & h & * & * \\ 0 & 0 & 0 & a & 0 & *\\ 0 & 0 & 0 & 0 & a & -b \\ 0 & 0 & 0 & 0 & 0 & a^{-1} \\ \end{matrix}\right) \in C_3\right\}.$$ The group of matrices of such form in $SL_6(K)$ is isomorphic to $U_{12}T_1$. In this case, taking the intersection with $Sp_6(K)$ is equivalent to imposing $4$ linearly independent conditions on the matrix entries given by requiring $0=(\im e_3,\im f_1)=(\im f_3,\im f_2)=(\im f_3,\im f_1)=(\im f_2,\im f_1)$. This shows that $H_{v_1}\cap H_{v_2}=U_8T_1$.

Since $x$ has a basis of weight vectors with respect to the maximal torus $T$, we have that $T\leq H_x$. Furthermore there exists an element $g_1\in H$ such that $g_1.e_2=e_2+e_3$ and $g_1.e_i=e_i$ for all $i\neq 2$, as well as an element $g_2\in H$ such that $g_2.e_3=e_3+e_2$ and $g_2.e_i=e_i$ for all $i\neq 3$. Then $g_1.v_1=v_1+v_2, g_1.v_2=v_2$ and $g_2.v_2=v_2+v_1,g_2.v_1=v_1$. This shows that $H_x/H_{v_1}\cap H_{v_2}=A_1T_1$, concluding the proof.
\end{proof}

\begin{lemma}
The stabilizer of $W_2^{(2)}$ is isomorphic to $U_5A_1A_1T_1$.
\end{lemma}

\begin{proof}
Let $x= W_2^{(2)}=\langle v_1,v_8 \rangle$. Since $v_1=e_1\wedge e_2$ and $v_8=e_2\wedge f_1$, any linear combination $\alpha v_1+\beta v_8=e_2 \wedge (-\alpha e_1 +\beta f_1)$ is a vector in the orbit of $v_1$.
Here $H_{v_1} = H_{\langle e_1, e_2 \rangle }/T_1$ and $H_{v_8} = H_{\langle e_2, f_1 \rangle }/T_1$. We get $$H_{v_1}\cap H_{v_8} = \left\{\pm\left(\begin{matrix}1 & 0 & 0 & * & * & 0 \\ c & 1 & * & * & * & * \\ 0 & 0 & g & h & * & 0\\ 0 & 0 & i & l & * & 0\\ 0 & 0 & 0 & 0 & 1 & 0 \\ 0 & 0 & 0 & 0 & -c & 1 \\ \end{matrix}\right)\in C_3\right\},$$ which is a $U_5A_1\times \langle -I\rangle $.

Since $x$ has a basis of weight vectors, we have $T\leq H_x$. Furthermore there exist elements $g_1,g_2\in H$ such that $g_1.e_1=e_1+f_1$, $g_2.f_1=e_1+f_1$ and $g_1,g_2$ fix all other basis vectors. Then $g_1.v_1=v_1-v_8,g_1.v_8=v_8$ and $g_2.v_8=v_8-v_1,g_2.v_1=v_1$. This means that $H_x/H_{v_1}\cap H_{v_8}=A_1T_1$, and therefore $H_x=U_5A_1A_1T_1$.
\end{proof}

\begin{lemma}
The stabilizer of $W_2^{(3)}$ is isomorphic to $U_7A_1T_1$.
\end{lemma}

\begin{proof}
Let $x= \langle v_1,v_2+v_5 \rangle$. Since $H_{\langle v_2+v_5 \rangle }\leq H_{\langle v_1 \rangle}$, we get $H_{\langle v_2+v_5 \rangle}\cap H_{\langle v_1 \rangle}= U_6A_1T_1$. 
Here, the $2$-space $x$ contains only a single point in the orbit of $\langle v_1 \rangle$, so $H_x\leq H_{\langle v_1 \rangle}$. Let $g\in H$ such that $g.e_3=e_3+\alpha e_2$, $g.f_2=f_2-\alpha f_3$ and fixing the other basis vectors. Then $g.(v_2+v_5)=v_2+v_5+\alpha v_1$ and $g.v_1=v_1$, and since $H_{\langle v_2+v_5 \rangle}\cap H_{\langle v_1 \rangle}= U_6A_1T_1$, we conclude that $H_x=U_7A_1T_1$.
\end{proof}

\begin{lemma}
The stabilizer of $W_2^{(4)}$ is isomorphic to $U_7T_2$.
\end{lemma}

\begin{proof}
Let $x=W_2^{(4)}=\langle v_1,v_2 +v_8\rangle$. Let $g\in H_{v_1}\cap H_{v_2+v_8}$. Since $v_2+v_8=e_1\wedge e_3+e_2\wedge f_1$, we find that $g$ needs to stabilise $\langle e_1,e_2,e_3,f_1 \rangle$, as well as its perpendicular space $\langle e_2,e_3 \rangle $. This shows that $g$ is of the form 
$$g=\pm \left(\begin{matrix}1 & 0 & 0 & * & * & * \\ b & 1 & a & * & * & * \\ 0 & 0 & 1 & h & * & * \\ 0 & 0 & 0& 1& -a & *\\ 0 & 0 & 0 & 0 & 1 & 0 \\ 0 & 0 & 0 & 0 & -b & 1 \\ \end{matrix}\right).$$
Let $g.f_1=\alpha e_1+\beta e_2+\gamma e_3+\delta f_3+f_1$. Then $$g.(v_2+v_8)= (e_1+be_2)\wedge (e_3+a e_2)+e_2\wedge (\alpha e_1+\beta e_2+\gamma e_3+\delta f_3+f_1)=v_2+v_8$$ implies that $\delta = 0$, $\alpha = a$ and $b=-\gamma $. This means that $g$ is of the form 
$$g=\pm \left(\begin{matrix}1 & 0 & 0 & * & * & a \\ b & 1 & a & * & * & \beta \\ 0 & 0 & 1 & h & * & -b \\ 0 & 0 & 0& 1& -a & 0\\ 0 & 0 & 0 & 0 & 1 & 0 \\ 0 & 0 & 0 & 0 & -b & 1 \\ \end{matrix}\right).$$

It is then easy to see that such elements in $H$ form the centralizer $H_{v_1}\cap H_{v_2+v_8}= U_6\times \langle -I\rangle$. This is again a $2$-space with only one point from the first orbit, therefore the most that could be induced on it is a $U_1T_2$, and this indeed easily seen to be the case. Since $\langle -I \rangle\leq T_2$, we conclude that $H_x=U_7T_2$.
\end{proof}

\begin{lemma}
The stabilizer of $W_2^{(5)}$ is isomorphic to $U_5T_3.2$.
\end{lemma}

\begin{proof}
Let $x=W_2^{(5)}=\langle v_1,v_{12} \rangle$. By intersecting the two parabolics $H_{\langle v_1 \rangle}=H_{\langle e_1,e_2 \rangle}$ and $H_{\langle v_{12} \rangle}=H_{\langle f_1,f_3 \rangle}$ we get $H_{\langle v_1 \rangle}\cap H_{\langle v_{12} \rangle}=U_5T_3$. 
Since the $2$-space $x$ contains precisely two points belonging to the orbit of $\langle v_1\rangle$, we have $H_x\leq U_5T_3.2$. Indeed, both $v_1$ and $v_{12}$ are weight vectors, and they can be swapped with the map sending $e_1\rightarrow f_1$, $f_1\rightarrow -e_1$ and $e_2\rightarrow f_3$, $f_3 \rightarrow -e_2$, $e_3\rightarrow f_2$, $f_2 \rightarrow -e_3$. Note that this map is swapping two $T_1$'s in $T_3=T$. 
\end{proof}

\begin{lemma}
The stabilizer of $W_2^{(6)}$ is isomorphic to $U_5T_2$.
\end{lemma}

\begin{proof}
Let $x=W_2^{(6)}=\langle v_1,v_2+v_{11} \rangle$. Let $g\in H_{v_1}\cap H_{v_2+v_{11}}$. Since $v_2+v_{11}=e_1\wedge e_3+ f_2\wedge f_3$, we find that $g$ stabilises $\langle e_1,e_3,f_3,f_2\rangle$, and its perpendicular space $\langle e_1,f_2 \rangle$. Then, since $g.(v_2+v_{11})=v_2+v_{11}$, we get that $\langle e_1 \rangle$ is fixed by $g$, and so is $\langle f_3 \rangle$.
Therefore $g$ is of the form $$g=\pm \left(\begin{matrix}1 & b & a & 0 & * & * \\ 0 & 1 & 0 & 0 & * & * \\ 0 & 0 & 1 & 0 & 0 & * \\ 0 & 0 & i& 1& 0 & -a\\ 0 & 0 & 0 & 0 & 1 & -b \\ 0 & 0 & 0 & 0 & 0 & 1 \\ \end{matrix}\right).$$ At this point a quick calculation shows that $H_{v_1}\cap H_{v_2+v_{11}}=U_4\times \langle -I \rangle$. Here we can have at most an induced $U_1T_2$ action, which we now exhibit. The $T_2$ is simply found as $T\cap H_{\langle v_2+v_{11}\rangle}$, while given $g\in H$ such that $g.e_3=e_3+e_2$, $g.f_2=f_2-f_3$ and $g$ fixes all other basis vectors, we get $g.v_1=v_1$ and $g.(v_2+v_{11})=v_2+v_{11}+v_1$. Since $-I\leq T_2$, we get the claimed stabilizer $H_x=U_5T_2$.
\end{proof}

\begin{lemma}
The stabilizer of $W_2^{(7)}$ is isomorphic to $U_4T_2$.
\end{lemma}

\begin{proof}
Let $x=W_2^{(7)}=\langle v_1,v_8+v_9 \rangle$. Let $g\in H_{v_1}\cap H_{v_8+v_{9}}$. Since $v_8+v_{9}=e_3\wedge f_2+ e_2\wedge f_1$, we find that $g$ stabilises $\langle e_2,e_3,f_2,f_1\rangle$, and its perpendicular space $\langle e_3,f_1 \rangle$. Requiring $g.(v_8+v_9)=v_8+v_9$ we find that $g$ is of the form $$g=\pm \left(\begin{matrix}1 & 0 & 0 & * & * & * \\ 0 & 1 & 0 & * & * & * \\ 0 & 0 & 1 & * & * & * \\ 0 & 0 & 0& 1& 0 & 0\\ 0 & 0 & 0 & 0 & 1 & 0 \\ 0 & 0 & 0 & 0 & 0 & 1 \\ \end{matrix}\right),$$ which can be reduced to $$g=\pm \left(\begin{matrix}1 & 0 & 0 & -c & 0 & 0 \\ 0 & 1 & 0 & d & c & 0 \\ 0 & 0 & 1 & b & -d & c \\ 0 & 0 & 0& 1& 0 & 0\\ 0 & 0 & 0 & 0 & 1 & 0 \\ 0 & 0 & 0 & 0 & 0 & 1 \\ \end{matrix}\right).$$ This shows that $H_{v_1}\cap H_{v_8+v_{9}}=U_3\times \langle -I\rangle$. Again, we have at most a $U_1T_2$ induced action, which we now exhibit. The $T_2$ is simply given by $T\cap H_{\langle v_8+v_9 \rangle}$. Note that it contains $-I$ as usual. Then consider $g\in H$ such that $g$ fixes all basis vectors apart from sending $f_1\rightarrow e_1+f_1$. We have $g.v_1=v_1$ and $g.(v_8+v_9)=v_8+v_9-v_1$. This proves that $H_x=U_4T_2$.
\end{proof}

\begin{lemma}
The stabilizer of $W_2^{(8)}$ is isomorphic to $U_1A_1T_1$.
\end{lemma}

\begin{proof}
Let $x=W_2^{(8)}=\langle v_1,v_9 +v_{12}\rangle$. Let $g\in H_{v_1}\cap H_{v_9+v_{12}}$. Since $v_9+v_{12}=f_1\wedge f_3+ e_3\wedge f_2$, we find that $g$ stabilises $\langle e_3,f_1,f_2,f_3\rangle$, and its perpendicular subspace $\langle f_1,f_2 \rangle$. Therefore $g$ stabilises $\langle e_1,e_2 \rangle$, $\langle e_3,f_3 \rangle$ and $\langle f_1,f_2 \rangle$. This means that $g$ is contained in the Levi subgroup $A_1A_1T_1$. Since $g$ fixes $ f_1\wedge f_3+ e_3\wedge f_2$, we know that $g.f_1$ and $g.f_2$ completely determine $g.e_3$ and $g.f_3$. 
Therefore $H_{v_1}\cap H_{v_9+v_{12}}=A_1$ and similarly $H_{\langle v_1 \rangle}\cap H_{\langle v_9+v_{12}\rangle}=A_1T_1$. Since $x$ contains precisely one point in the orbit of $\langle v_1\rangle$, the stabilizer $H_x$ is a subgroup of $H_{\langle v_1 \rangle}$. The map $$\left(\begin{matrix}-1 & 1 & 1 & -1 & 1 & 1 \\ -1 & 0 & 0 & 1 & 1 & 0 \\ 0 & 0 & 0 & -1 & -1 & 0 \\ 0 & 0 & -1& 1& 1 & 1 \\ 0 & 0 & 0 & 0 & -1 & -1 \\ 0 & 0 & 0 & 0 & 1 & 0 \\ \end{matrix}\right)$$  fixes $\langle v_1 \rangle$ and sends $v_9+v_{12}$ to $v_9+v_{12}-v_1$. Hence $H_x=U_1A_1T_1$.
\end{proof}

\begin{lemma}
The stabilizer of $W_2^{(9)}$ is isomorphic to $U_4T_1$.
\end{lemma}

\begin{proof}
Let $x=W_2^{(9)}=\langle v_2+v_5,v_{3}+v_6 \rangle$. Let $g\in H_{v_2+v_5}\cap H_{v_{3}+v_6}$. Since $v_2+v_5=e_1\wedge e_3+e_2\wedge f_3$ and $v_3+v_6=e_2\wedge e_3+e_2\wedge f_2$ only involve vectors in $\langle e_1,e_2,e_3,f_2,f_3\rangle $, this subspace is fixed by $g$.
Then $g$ also fixes the perpendicular space $\langle e_1 \rangle$. Requiring $g.(v_{2}+v_5)=v_{2}+v_5$ we find that $\langle e_2 \rangle$ is also fixed. This leads to $g$ being of the form  $$\pm \left(\begin{matrix}
1& 0& a & 0 & b & c \\
 0& 1& 0 & 0 & -a & b \\
  0 & 0 & 1 & 0 & 0 & 0 \\ 
  0 & 0 & 0& 1& 0 & -a \\ 
  0 & 0 & 0 &0 & 1 & 0 \\ 
  0 & 0 & 0 & 0 & 0& 1 \\ \end{matrix}\right),$$ showing that $H_{v_2+v_5}\cap H_{v_{3}+v_6}=U_3\times \langle -I \rangle$. We find a $T_1\leq H_x$ inside the maximal torus $T$. We then consider the element $$g= \left(\begin{matrix}
1& -1& 0 & 0 & 0 & 0 \\
 0& 1& 0 & -1 & 0 & 0 \\
  0 & 0 & 1 & 0 & -1 & 0 \\ 
  0 & 0 & 1& 1& 1 & 0 \\ 
  0 & 0 & 0 &0 & 1 & 0 \\ 
  0 & 0 & 0 & 0 & 0& 1 \\ \end{matrix}\right),$$ which fixes $v_2+v_5$ and sends $v_3+v_6$ to $v_3+v_6+v_2+v_5$.
  This implies that there exists a $U_1T_1\leq H_x/(H_{v_2+v_5}\cap H_{v_{3}+v_6})$.
  Finally suppose that there exists $g\in H_x$ such that $g(\langle v_2+v_5 \rangle )\neq \langle v_2+v_5 \rangle$. Since we already have an induced $U_1T_1$ action on $x$, we can assume that $g$ is fixing $\langle v_3+v_6 \rangle $. This means that $g$ fixes $\langle e_1, e_3 \rangle $, i.e. $g.\langle e_1\wedge e_3 \rangle= \langle e_1\wedge e_3 \rangle$. Let us consider $g.e_2\wedge f_3$. Since $g\in H_{\langle v_3+v_6 \rangle}$, we have $g.e_2\in \langle e_1,e_2,e_3,f_2\rangle$. However the fact that $g.f_3=\alpha e_1+\beta e_2+\gamma e_3+\delta f_2+\epsilon f_3$ with $\epsilon\neq 0$, combined with $g\in H_x$, means that $g.\langle e_2 \rangle =\langle e_2 \rangle$. Then $g\not\in H_{x}$, a contradiction.
  
  We have therefore shown that $H_x\leq H_{\langle v_2+v_5 \rangle}$. The same method used to find the centralizer of $x$ can be used to get $H_{\langle v_2+v_5 \rangle}\cap H_{\langle v_3+v_6 \rangle}=U_3T_1$. Then we indeed have $H_x=U_4T_1$.
\end{proof}

\begin{lemma}\label{more difficult c3 lemma}
The stabilizer of $W_2^{(10)}$ is isomorphic to $A_1T_1.2$.
\end{lemma}

\begin{proof}
 Let $x=W_2^{(10)}=\langle v_2+v_5,v_{12}-v_{9} \rangle$. Let $g\in H_{v_2+v_5}\cap H_{v_{12}-v_9}$. Then $g$ stabilises $\langle e_1,e_2 \rangle \oplus \langle e_3,f_3 \rangle \oplus \langle f_1,f_2 \rangle$, i.e. $g$ is in the Levi subgroup $A_1A_1T_1$. Let $g.e_3=d e_3-b f_3$ and $g.f_3=-c e_3+a f_3$. Since $g\in Sp_6$, we have $ad-bc=1$. From the fact that $g\in H_{v_2+v_5}$ we then get that $g$ is of the form
 $$g=\left(\begin{matrix}a & b & 0 & 0 & 0 & 0 \\ c & d & 0 & 0 & 0 & 0 \\ 0 & 0 & d & -c & 0 & 0 \\ 0 & 0 & -b & a & 0 & 0\\ 0 & 0 & 0 & 0 & a & -b \\ 0 & 0 & 0 & 0 & -c & d \\ \end{matrix}\right).$$
Now,
\begin{flalign*}
 g.(v_{12}-v_{9})&= (-b f_2+d f_1)\wedge (-ce_3 +a f_3)+(af_2-cf_1)\wedge (de_3-bf_3)=&\\
 & =f_1\wedge (cd e_3 +(ad+bc) f_3)+f_2\wedge (ab f_3+(ad+bc) e_3).
\end{flalign*}
This implies that $cd=ab=0$ and $ ad+bc=1$. Combined with $ad-bc=1$ this implies $bc=0$ and therefore $b=c=0$ and $d=a^{-1}$. Therefore $g\in T$ and $g$ is an element of the form $\mathrm{diag}(a,a^{-1},a^{-1},a,a,a^{-1})$, which tells us that $H_{v_2+v_5}\cap H_{v_{12}-v_9} = T_1$. Note that as usual $\langle -I \rangle \in T_1$.

  We exhibit an $A_1.2$ action by $H_x$ on $x$. Consider the unipotent element sending $e_1\rightarrow e_1+\alpha f_2$ and $e_2 \rightarrow e_2+\alpha f_1$. It sends $e_1\wedge e_3 \rightarrow e_1\wedge e_3 +\alpha e_1\wedge f_2$ and $e_2\wedge f_3 \rightarrow e_2\wedge f_3 +\alpha f_1\wedge f_3$. Therefore it sends $v_2+v_5$ to $v_2+v_5+\alpha (v_{12}-v_{9})$, while fixing $v_{12}-v_9$. At the same time the unipotent element sending $f_1\rightarrow f_1+\alpha e_2$ and $f_2\rightarrow f_2+\alpha e_1$ sends $v_{12}-v_{9}$ to $v_{12}-v_9+\alpha(v_2+v_5)$. This shows that we do indeed have a faithful $A_1$ action on $x$.
  
  Then consider the element $g$ such that $g(e_1)=e_2$, $g(e_2)=-e_1$, $g(e_3)=f_3$, $g(f_3)=-e_3$, $g(f_1)=f_2$ and $g(f_2)=-f_1$.
It fixes $v_2+v_5$ while scaling $v_{12}-v_9$ by $-1$. Therefore we do indeed have $A_1T_1.2\leq H_x$. 

Assume that $A_1T_1.2 \neq H_x$. Then we can find $g\in H_x$ such that $g\in H_{v_2+v_5}$ and $g(v_{12}-v_9)=\alpha (v_{12}-v_9)$ for some $\alpha \neq \pm 1$. We again find that $g$ is in the Levi $A_1A_1T_1$. As before, $g$ is of the form $$g=\left(\begin{matrix}a & b & 0 & 0 & 0 & 0 \\ c & d & 0 & 0 & 0 & 0 \\ 0 & 0 & d & -c & 0 & 0 \\ 0 & 0 & -b & a & 0 & 0\\ 0 & 0 & 0 & 0 & a & -b \\ 0 & 0 & 0 & 0 & -c & d \\ \end{matrix}\right),$$ with $ad-bc=1$. Looking at $g(v_{12}-v_{9})$ it is then easy to find that either $a=d=1,b=c=0$, or $a=d=0,b=-1,c=1$, giving $\alpha=\pm 1$, a contradiction. Hence $H_x=A_1T_1.2$.
\end{proof}

\begin{lemma}
The stabilizer of $W_2^{(11)}$ is isomorphic to $U_2T_1$.
\end{lemma}

\begin{proof}
Let $x=W_2^{(11)}=\langle v_2+v_5,v_4+v_{13} \rangle$. Let $g\in H_{v_2+v_5}\cap H_{v_4+v_{13}}$. Since $v_2+v_5=e_1\wedge e_3+ e_2\wedge f_3$ and $v_4+v_{13}=f_3\wedge e_1+f_1\wedge f_2$, the element $g$ stabilises $\langle e_1,e_2,e_3,f_3 \rangle $, $\langle e_1,e_2 \rangle$, $\langle e_1,f_1,f_2,f_3 \rangle $ and $\langle f_2,f_3 \rangle$. This means that $\langle e_1 \rangle$ and $\langle f_3 \rangle$ are fixed. Hence $g(v_4+v_{13})=\alpha f_3\wedge e_1 +g(f_1)\wedge g(f_2)$, which implies that $\langle f_2 \rangle $ and $\langle f_1,f_2 \rangle$ are fixed. This quickly leads to $g$ being of the form $$g=\pm \left(\begin{matrix}1 & b & 0 & 0 & 0 & 0 \\ 0 & 1 & 0 & 0& 0 & 0 \\ 0 & 0 & 1 & 0 & 0 & 0 \\ 0 & 0 & -b& 1& 0 & 0\\ 0 & 0 & 0 & 0 & 1 & -b \\ 0 & 0 & 0 & 0 & 0 & 1 \\ \end{matrix}\right).$$ The centralizer of $x$ is therefore $U_1\times \langle -I \rangle$. If we set $T_1=T_x$, then a similar reasoning actually shows that $H_{\langle  v_2+v_5 \rangle}\cap H_{\langle  v_4+v_{13} \rangle}=U_1T_1.$ The element $$\left(\begin{matrix}
-1& 1& 0 & 0 & 0 & 1 \\
 0& -1& 0 & 0 & 0 & 0 \\
  0 & -1 & -1 & 0 & 0 & 0 \\ 
  1 & 0 & 0& -1& 0 & 1 \\ 
  1 & -1 & 1 & 1 & -1 & -1 \\ 
  0 & -1 & 1 & 0 & 0& -1 \\ \end{matrix}\right)$$  fixes $\langle f_3\wedge e_1+f_1\wedge f_2 \rangle$ and sends $v_2+v_5$ to $v_2+v_5+v_4+v_{13}$. Together with the $T_1$ this implies that we have a $U_2T_1\leq H_x$. Now  let $g\in H_x\cap H_{\langle v_2+v_5 \rangle}$.
Let $$g(v_4+v_{13})=g(v_4)+(a_1e_1+a_2e_2+a_3e_3+b_3f_3+b_2f_2+b_1f_1)\wedge (c_1e_1+c_2e_2+c_3e_3+d_3f_3+d_2f_2+d_1f_1).$$

Since $g$ fixes $\langle v_2+v_5 \rangle$, the subspace $\langle e_1,e_2,e_3,f_3 \rangle$ is fixed by $g$ and $g.v_4=g.(e_1\wedge f_3)$ is in the kernel of the projection map from $V$ onto $V_6\wedge f_1 +V_6\wedge f_2$.
Therefore, the coefficients of basis vectors involving $f_1$ or $f_2$ in $g.(v_4+v_{13})$, are all obtained from $g(v_{13})$. The coefficient of $f_1\wedge f_2$ in $g(v_4+v_{13})$, which is $b_1d_2-b_2d_1$, must be non-zero. We want to show that $g$ actually fixes $\langle f_1,f_2 \rangle$, which we do by proving that all other coefficients involving $f_1$ or $f_2$ are zero. This is a simple matter of linear algebra. Consider for example the coefficient of $e_3\wedge f_2$ in $ g.(v_4+v_{13})$, which is $\det \left(\begin{matrix}
a_3 &c_3 \\ b_2 & d_2 \\
\end{matrix}\right)$, as well as the coefficient of $e_3\wedge f_1$ which is $\det \left(\begin{matrix}
a_3 &c_3 \\ b_1 & d_1 \\
\end{matrix}\right) $. Since both of these determinants are zero and  $b_1d_2-b_2d_1=\det \left(\begin{matrix}
b_2 &d_2 \\ b_1 & d_1 \\
\end{matrix}\right)\neq 0$, we have that $a_3=c_3=0$. This same reasoning can be applied to deal with the other coefficients, showing that $g\in H_{\langle f_1,f_2\rangle}$.

Now consider $$g(v_4)=(a_1e_1+a_2e_2+a_3e_3+b_3f_3)\wedge (c_1 e_1+c_2e_2)$$ for some new arbitrary scalars $a_1,a_2,a_3,b_3,c_1,c_2$. If $a_3\neq 0$ then $c_2=0$ and $g(v_4)=\lambda v_4$, and the same must be true if $a_3=0$. Therefore $g\in H_{\langle v_4+v_{13}\rangle}$. This allows us to conclude that $H_x=U_2T_1$, as claimed.
\end{proof}

\begin{lemma}
The stabilizer of $W_2^{(12)}$ is isomorphic to $U_1T_1.2$.
\end{lemma}

\begin{proof}
Although this can be done with the same method as for the other $2$-spaces, we deal with this case starting from \cite[$3.2.18$]{generic2}. Here the authors fix a $7$-dimensional subspace $V_7$ of $V$ and then
construct a set $Y$ of $2$-spaces of $V_7$, and an open dense subset $\hat{Y}_1\subseteq Y$ with the property that for all $y\in \hat{Y}_1$ $$\mathrm{Tran}_H(y,Y)=\tilde{A}_2T_1\langle n^*\rangle.$$
Here $\tilde{A}_2$ acts on $V_7$ as $V_{\tilde{A}_2}(\lambda_1+\lambda_2)$, the $T_1$ centralizes $V_7$, and $n^*$ is an involution.

It is easy to see (for a similar proof with more details see Proposition~\ref{f4 orbits 13 and 15}) that both $Y$ and $\hat{Y}_1$ have a non-empty intersection with $P_2^{TS}(V)$. This means that we can define a set $Y^{TS}=Y\cap P_2^{TS}(V)$ with an open dense subset $\hat{Y}_1^{TS}:=\hat{Y}_1\cap P_2^{TS}(V)\subseteq Y^{TS}$ with the property that for all $y\in \hat{Y}_1^{TS}$ $$\mathrm{Tran}_H(y,Y^{TS})=\tilde{A}_2T_1\langle n^*\rangle.$$ The stabilizer of any $2$-space $y\in \hat{Y}_1^{TS}$ is therefore simply given by its stabilizer in $\tilde{A}_2\langle n^* \rangle$.
In Proposition~\ref{proposition a2 k=2} we determined that $U_1$ is the generic stabilizer for the action on totally singular $2$-spaces of $V_{\tilde{A}_2}(\lambda_1+\lambda_2)$. In Proposition~\ref{f4 orbits 13 and 15} we explicitly show that there is a member of $\hat{Y}_1^{TS}$ with stabilizer $U_1\langle n^*\rangle$ in $\tilde{A}_2\langle n^* \rangle$.
This proves that there is a totally singular $2$-space with stabilizer $U_1T_1.2$. 
\end{proof}

We prove Proposition~\ref{c3 k=2 main prop} by proving the following equivalent proposition. 

\begin{proposition}
Table~\ref{tab:c3 orbits} has a complete list of representatives for the action of $C_3$ on $P_2^{TS}(V)$.

\begin{center}
 \begin{longtable}{l l l l } \caption{$C_3$-orbits on totally singular $2$-spaces}
\label{tab:c3 orbits}\\
 \hline
 Orbit number & $H$-orbit representative $x$&  $H_{x}$ & $\dim x^H$ \\ 
 \hline
 \endhead
 1& $ \langle v_1,v_2 \rangle $ & $U_8A_1T_2$  & $8$\\*
 
   2& $\langle v_1,v_8 \rangle$ & $U_5A_1A_1T_1 $  & $9$ \\*
 
 3 &$\langle v_1,v_2+v_5 \rangle$   &$ U_7A_1T_1$   & $10$ \\*

   4& $\langle v_1,v_2 +v_8\rangle$  &   $U_7T_2 $  & $12$\\*

   5&$\langle v_1,v_{12} \rangle$  &   $U_5T_3.2$   & $13$\\*

   6&$\langle v_1,v_2+v_{11} \rangle$  &   $U_5T_2$   & $ 14$\\*

  7&$\langle v_1,v_8+v_9 \rangle$  &   $U_4T_2$   & $15$\\*

    8&$\langle v_1,v_9 +v_{12}\rangle$  &   $U_1A_1T_1$ & $16$  \\*

   9&$\langle v_2+v_5,v_{3}+v_6 \rangle$  &   $U_4T_1$   & $16$\\*

   10&$\langle v_2+v_5,v_{12}-v_{9} \rangle$  &   $A_1T_1.2 $   & $17$\\*

  11&$\langle v_2+v_5,v_4+v_{13} \rangle$  &   $U_2T_1$   & $18$\\*

  12&$\langle v_2+v_5,v_5+v_{13} \rangle$  &   $U_1T_1.2$   & $19$\\*
 \hline
 
\end{longtable}
\end{center} 

\end{proposition}

\begin{proof}
All the stabilizers follow from the above lemmas.
 
 Let $q=p^e=3^e$ for an arbitrary positive integer $e$. Let $\sigma_q$ be the standard Frobenius morphism sending $x_i(t)$ to $x_i(t^q)$ and acting in a compatible way on $V$. Then the induced action of $\sigma$ on $P_2^{TS}(V)$ stabilises the orbits in Table~\ref{tab:c3 orbits}, since for each orbit we have a representative given in terms of the basis with coefficients in $\mathbb{F}_3$.  
The only orbits in Table~\ref{tab:c3 orbits} with a disconnected stabilizer are numbers $5$, $10$ and $12$. 

Let $\Gamma_5$ be the $H$-orbit with representative $W_2=\langle v_1,v_{12} \rangle$ and stabilizer $H_{W_2}=U_5T_3.2$. The component group of $H_{W_2}$, which swaps $\langle v_1 \rangle$ and $\langle v_{12} \rangle$, centralises a $ U_5T_2\leq U_5T_3$ and inverts a $T_1$. Therefore by Lang-Steinberg, the fixed points of $\Gamma_5$ under $\sigma_q$, split into two $Sp_6(q)$-orbits with stabilizers $[q^5].(q-1)^3.2$ and $[q^5].(q-1)^2.(q+1).2$.

In both orbits $10$ and $12$, the component group of the stabilizer centralizes the $2$-space and inverts a $T_1$. In the case of orbit $10$, when passing to finite fields, this produces two orbits with stabilizers  $ SL_2(q).(q-1).2$ and $ SL_2(q).(q+1).2$. Finally, in the case of orbit $12$, when passing to finite fields, we get two orbits with stabilizers $ [q].(q-1).2$ and $[q].(q+1).2$. The stabilizers for the orbits in the finite case are therefore as in Table~\ref{tab:c3 orbits finite fields}.

\begin{center}
 \begin{longtable}{l l l  } \caption{$Sp_6(q)$-orbits on totally singular $2$-spaces of $V_\sigma$}
\label{tab:c3 orbits finite fields}\\
 \hline
 Orbit number & $H_{x}$ & $(H_\sigma )_{x_\sigma}$ \\ 
 \hline
 \endhead
 1&  $U_8A_1T_2$  & $[q^8].SL_2(q).(q-1)^2$\\*
 
   2&  $U_5A_1A_1T_1 $  & $[q^5].SL_2(q)^2.(q-1)$ \\*
 
 3 &$ U_7A_1T_1$   & $[q^7].SL_2(q).(q-1)$ \\*

   4&   $U_7T_2 $  & $[q^7].(q-1)^2$\\*

   5&  $U_5T_3.2$   & $[q^5].(q-1)^3.2$\\*
   &     & $[q^5].(q-1)^2.(q+1).2$\\*

   6&  $U_5T_2$   & $ [q^5].(q-1)^2$\\*

  7&  $U_4T_2$   & $[q^4].(q-1)^2$\\*

    8&   $U_1A_1T_1$ & $[q].SL_2(q).(q-1)$  \\*

   9&   $U_4T_1$   & $[q^4].(q-1)$\\*

   10&  $A_1T_1.2 $   & $SL_2(q).(q-1).2$\\*
   &     & $SL_2(q).(q+1).2$\\*

  11&   $U_2T_1$   & $[q^2].(q-1)$\\*

  12&   $U_1T_1.2$   & $[q].(q-1).2$\\*
    &      & $[q].(q+1).2$\\*
 \hline
 
\end{longtable}
\end{center} 

We can then compute the indices of the stabilizers over $\mathbb{F}_q$ to get the sizes of the orbits. Adding up the orbit sizes gives precisely the number of totally singular subspaces of dimension $2$ in $V_\sigma$. We have therefore found a complete list of orbit representatives. 
\end{proof}

\subsection{$H$ of type $F_4$ and $V=V_{H}(\lambda_4)$}\label{f4 section k=2}

In this section we prove the following:

\begin{proposition}\label{F4 main prop}
Let $V=V_{F_4}(\lambda_4)$, which is an orthogonal module of dimension $26-\delta_{p,3}$. Then $F_4$ has $15$ orbits on $P_2^{TS}(V)$ if $p=3$, and no dense orbit otherwise. Orbit representatives and stabilizers for the case $p=3$ can be found in Table~\ref{tab:f4 orbits}.
\end{proposition}

By \cite[Thm.~3]{generic2} the action of $F_4$ on $P_2(V)$ has a generic stabilizer with connected component $A_2$. By Corollary~\ref{minimum dimension generic} the stabilizer of any $2$-space of $V$ is at least $8$-dimensional. If $p\neq 3$ we have $\dim P_2^{TS}(V)=45=\dim F_4-7$, which implies that there is no dense $F_4$-orbit on $P_2^{TS}(V)$. To prove Proposition~\ref{F4 main prop} we therefore only need to show that when $p=3$, the group $F_4$ has $15$ orbits on $P_2^{TS}(V)$. Let $p=3$.

We begin by describing the setup that we will be using. Let $H$ be the simply connected group of type $E_7$ in characteristic $p=3$, with simple positive roots $\beta_1\dots, \beta_7$. Let $G^+=E_6$ have simple positive roots $\gamma_i=\beta_i$ for $i\leq 6$, so that $G^+\leq H$. Let $$V^+=\langle e_\beta : \beta = \sum m_i\beta_i, m_7=1 \rangle < \mathrm{Lie}(H).$$ We order the positive roots in both root systems first according to the height, and for roots of the same height by lexicographical ordering. We denote a root in position $i$ by $\gamma_i$ with $1\leq i\leq 36$ for $E_6$, and $\beta_i$ with $1\leq i \leq 63$ for $E_7$. We also adopt the convention $\gamma_{i+36}=-\gamma_i$ and $\beta_{i+63}=-\beta_{i}$.

We take $G=F_4$ to be the subgroup of $E_6$ having long simple roots $\gamma_2,\gamma_4$ and short simple root groups $\{x_{\gamma_3}(t)x_{\gamma_5}(-t):t\in K\}$ and $\{x_{\gamma_1}(t)x_{\gamma_6}(-t):t\in K\}$.
These correspond respectively to the fundamental simple roots $\alpha_1,\alpha_2,\alpha_3,\alpha_4$ in $F_4$. Let $T$ be the standard maximal torus of $F_4$. Again, let $\alpha_i$ denote the $i$-th positive root in $F_4$ according to the height and lexicographical ordering, for $1\leq i \leq 24$, and $\alpha_{i+24}=-\alpha_i$. By \cite[$3.2.12$]{generic1}, this construction of $F_4$ is equivalent to setting $G=(E_6)_{v_0}$, where $$v_0=e_{\beta_{47}}+e_{\beta_{48}}+e_{\beta_{49}}\in \mathrm{Lie}(E_7).$$ Note that $\beta_{47}= 1122111$, $\beta_{48}=1112211$, $\beta_{49}=0112221$.

We let $W=N_G(T)/T$ denote the Weyl group of $G=F_4$. Given an element $w\in W$ we write $\dot{w}$ for a pre-image in $N_G(T)$. We denote by $X_i$ the root subgroup $X_{\alpha_i}$ in $F_4$, with $X_{-i}$ denoting $X_{-\alpha_i}$. We use $n_i$ for the standard pre-image $n_i(1)=x_i(1)x_{-i}(1)x_i(t)\in \langle X_i,X_{-i}\rangle$ in $F_4$, while $n_i'$ denotes the ones in $E_7$. 

By \cite[$3.2.18$]{generic2}, in the $27$-dimensional $E_6$-module $V^+$, the subgroup $G=F_4$ stabilises the $26$-dimensional subspace $$ V_{26}= \langle \sum a_{i}e_{\beta_i} \in V^+: a_{47}+a_{48}+a_{49}=0\rangle.$$ Since $p=3$ we have $v_0\in V_{26}$ and $F_4$ stabilises the $25$-dimensional space $$V=V_{26}/\langle v_0 \rangle.$$ This is the $25$-dimensional irreducible module of highest weight $\lambda_4=\alpha_{21}=1232$, the highest short root in $F_4$. The zero weight space for $T$ in $V_{26}$ is the $2$-dimensional subspace $$V_0=\langle a_{47}e_{\beta_{47}}+a_{48}e_{\beta_{48}}+a_{49}e_{\beta_{49}} : a_{47}+a_{48}+a_{49}=0\rangle,$$ while the $0$-space in $V$ is the $1$-dimensional $\overline{V_0} = V_0/\langle v_0 \rangle$. When talking about the $25$-dimensional module $V$ we will keep the same notation we would use for $V_{26}$, without explicitly writing the quotient. If $\lambda$ is a weight of $V$, we write $V_{\lambda}$ for the weight space, and $v_{\lambda}$ for an arbitrary vector in $V_\lambda$.

The signs of the structure constants for $E_7$ are taken to match the ones found in \cite{gilkey-seitz}, which coincide with the default ones in Magma. Note that in both cases the extra-special pairs of roots are assigned a positive sign. In the following lemma we describe the root subgroups $X_{\alpha_i}$ in terms of root subgroups in $E_6$.

\begin{lemma}\label{f4e6}
Let $1\leq i \leq 24$. Then the root subgroup $X_{\alpha_i}$ in $F_4$ can be written in terms of elements of $E_6$ as described in Table~\ref{tab:f4 in e6}.

\begin{table}[h]
\begin{minipage}{.45\linewidth}
\begin{center}
\begin{xltabular}[h]{\textwidth}{l l  }
\caption{Embedding in $E_6$ of $F_4$-root subgroups} \label{tab:f4 in e6} \\

\hline \multicolumn{1}{c}{ Index $i$} & \multicolumn{1}{c}{$E_6$-embedding}  \\ \hline 
\endhead
1& $X_{\gamma_2} $ \\

  2 &$X_{\gamma_4} $ \\ 

  3& $x_{\gamma_3}(t)x_{\gamma_5}(-t)$  \\

    4&$x_{\gamma_1}(t)x_{\gamma_6}(-t) $ \\

   5& $X_{\gamma_8} $ \\

   6&$x_{\gamma_{9}}(t)x_{\gamma_{10}}(t) $ \\
    7&$x_{\gamma_{7}}(t)x_{\gamma_{11}}(-t) $  \\
    8&$x_{\gamma_{13}}(t)x_{\gamma_{14}}(t) $   \\
   9&$X_{\gamma_{15}} $ \\
   10&$x_{\gamma_{12}}(t)x_{\gamma_{16}}(t)  $ \\
  11&$X_{\gamma_{19}} $   \\
  12&$x_{\gamma_{17}}(t)x_{\gamma_{20}}(t)  $ \\
   13&$x_{\gamma_{18}}(t)x_{\gamma_{21}}(t) $ \\
   14&$X_{\gamma_{24}} $ \\
   15&$x_{\gamma_{22}}(t)x_{\gamma_{25}}(t)  $  \\
    16&$X_{\gamma_{23}} $  \\
    17&$x_{\gamma_{26}}(t)x_{\gamma_{28}}(t)  $  \\
    18&$X_{\gamma_{27}} $  \\
    19&$x_{\gamma_{29}}(t)x_{\gamma_{31}}(-t)  $  \\
    20&$X_{\gamma_{30}} $  \\
    21&$x_{\gamma_{32}}(t)x_{\gamma_{33}}(-t)  $  \\
    22&$X_{\gamma_{34}} $  \\
    23&$X_{\gamma_{35}} $  \\
    24&$X_{\gamma_{36}} $  \\
\hline
\end{xltabular}
\end{center}

\end{minipage}
\hfill
\begin{minipage}{.45\linewidth}
\begin{center}
\begin{xltabular}[h]{\textwidth}{l r  }
\caption{Correspondence between $E_7$-roots and $F_4$-weights} \label{tab:e7 to f4} \\

\hline \multicolumn{1}{c}{$\beta$} & \multicolumn{1}{c}{$\lambda$}  \\ \hline 
\endhead
$2 2 3 4 3 2 1$& $1 2 3 2$ \\

  $1 2 3 4 3 2 1$ &$1 2 3 1$ \\ 
 
  $1 2 2 4 3 2 1$& $1 2 2 1$  \\

     $1 2 2 3 3 2 1$&$1 1 2 1$ \\

   $1 1 2 3 3 2 1$& $0 1 2 1$ \\
    $1 2 2 3 2 2 1$&$1 1 1 1$ \\
     $1 1 2 3 2 2 1$&$0 1 1 1$  \\
     $1 2 2 3 2 1 1$&$1 1 1 0$   \\
    $1 1 2 2 2 2 1$&$0 0 1 1$ \\
    $1 1 2 3 2 1 1$&$0 1 1 0$ \\
   $1 1 1 2 2 2 1$&$0 0 0 1$   \\
  $1 1 2 2 2 1 1$&$0 0 1 0$ \\
    $0 1 1 2 2 2 1,1 1 1 2 2 1 1,1 1 1 2 1 1 1$&$0000$ \\
    $0 1 1 2 2 1 1$&$-0 0 0 1$  \\
    $1 1 1 2 1 1 1$&$-0 0 1 0$  \\
     $0 1 1 2 1 1 1$&$-0 0 1 1 $ \\
     $1 1 1 1 1 1 1$&$-0 1 1 0$  \\
     $0 1 1 1 1 1 1$&$-0 1 1 1$  \\
     $1 0 1 1 1 1 1$&$-1 1 1 0$  \\
     $0 0 1 1 1 1 1$&$-1 1 1 1$  \\
     $0 1 0 1 1 1 1$&$-0 1 2 1$  \\
     $0 0 0 1 1 1 1$&$-1 1 2 1$  \\
     $0 0 0 0 1 1 1$&$-1 2 2 1$  \\
    $0 0 0 0 0 1 1$&$-1 2 3 1$  \\
    $0 0 0 0 0 0 1$&$-1 2 3 2$\\
 \hline
\end{xltabular}
\end{center}
\end{minipage}
\end{table}

\end{lemma}

\begin{proof}
Note that our definition of $F_4$ is equivalent to taking the fixed points of the automorphism of $E_6$ induced by $-w_0$, where $w_0$ is the longest element of the Weyl group for $E_6$. This is in fact the standard folding of the Dynkin diagram for $E_6$, and the equivalence can be seen by just looking at the generators for $F_4$. Since, with respect to our ordering of the roots,
\begin{flalign*}
& w_0=(1, 42)(2, 38)(3, 41)(4, 40)(5, 39)(6, 37)(7, 47)(8, 44)(9, 46)(10, 45)(11, 43)(12, 52)(13, 50)\\
& (14, 49)(15, 51)(16, 48)(17, 56)(18, 57)(19, 55)(20, 53)(21, 54)(22, 61)(23, 59)(24, 60)(25, 58)\\
&(26, 64)(27, 63)(28, 62)(29, 67)(30, 66)(31, 65)(32, 69)(33, 68)(34, 70)(35, 71)(36, 72),
\end{flalign*}
it is easy to find the pairs of $E_6$-roots that are swapped and the roots that are fixed by $-w_0$. These correspond to the ones listed in the second column in Table~\ref{tab:f4 in e6}. 
We still need to determine the signs appearing in the products. A way to do this is to check that the elements listed in Table~\ref{tab:f4 in e6} do indeed fix $v_0=e_{\beta_{47}}+e_{\beta_{48}}+e_{\beta_{49}}$. This can be done with the standard formulas found in \cite[Lemma~6.2.1]{cartersimple}. Let us show it for number $6$. Let $g=x_{\gamma_{9}}(t)x_{\gamma_{10}}(t)$. Then $g=x_{\beta_{10}}(t)x_{\beta_{11}}(t)=x_{0 0 1 1 0 0 0}(t)x_{0 0 0 1 1 0 0}(t)$. To determine the action of $g$ on $v_0=e_{1122111}+e_{1112211}+e_{0112221}$, we note that the only non-trivial ($\beta_{10}$ or $\beta_{11}$)-root strings through either $1122111$, $1112211$ or $0112221$ are $(1112211,1 1 2 3 2 1 1)=(\beta_{48},\beta_{48}+\beta_{10})$ and $(1122111,1 1 2 3 2 1 1)=(\beta_{47},\beta_{47}+\beta_{11})$. By \cite[Lemma~6.2.1]{cartersimple} we know that $x_{\beta_{10}}(t).v_0=v_0\pm t e_{1 1 2 3 2 1 1}$ and $x_{\beta_{11}}(t).v_0=v_0\pm t e_{1 1 2 3 2 1 1}$. The sign in both images depends on the sign of the structure constants $N_{0011000,1112211}$ and $N_{0001100,1122111}$, which are $-$ and $+$ respectively, as can be seen in \cite{gilkey-seitz}. Hence $g$ fixes $v_0$, which means that $g\in F_4$, as claimed. This allows us to conclude that the second column of Table~\ref{tab:f4 in e6} does indeed contain a complete list of positive root subgroups of $F_4$.

What remains to be determined is that the way they are ordered corresponds to the given ordering of the $F_4$-root subgroups $X_{\alpha_i}$.
A quick way to check this is by taking an arbitrary element $g=h_{\alpha_1}(\kappa_1)h_{\alpha_2}(\kappa_2)h_{\alpha_3}(\kappa_3)h_{\alpha_4}(\kappa_4)$ in the standard maximal torus for $F_4$, and find the character with which it acts on each root subgroup. We show this for number $6$. By construction $g$ is the element $ h_{\beta_2}(\kappa_1)h_{\beta_4}(\kappa_2)h_{\beta_3}(\kappa_3)h_{\beta_5}(\kappa_3)h_{\beta_1}(\kappa_4)h_{\beta_6}(\kappa_4)$ in $E_7$. The root subgroup number $6$ consists of elements of the form $x_{\gamma_{9}}(t)x_{\gamma_{10}}(t)$, so we just need to determine $x_{\gamma_{9}}(t)^g$. We get $x_{\gamma_{9}}(t)^g=x_{\beta_{10}}(t)^g=x_{\beta_{10}}(\kappa_1^{-1}\kappa_2\kappa_4^{-1}t)$. At the same time $x_{\alpha_6}(t)^g=x_{\alpha_6}(\kappa_1^{-1}\kappa_2\kappa_4^{-1}t)$, which shows that the two root subgroups are the same.
\end{proof}

This means that if we have a root element in $F_4$, we can use Lemma~\ref{f4e6} to express it as a product of at most two root elements in $E_6$, which are naturally embedded in $E_7$. Then \cite[Lemma~6.2.1]{cartersimple} gives us a way to explicitly compute the action on any non-zero weight vector in $V^+$. 

An element $e_\beta\in V^+$ is a weight vector of weight $\lambda$ with respect to the $F_4$ action, as described in the following lemma.

\begin{lemma}\label{lemma e7 f4 weights}
Let $\beta=  \sum m_i\beta_i$ be a root in $E_7$ with $m_7=1$. Then $e_\beta\in V^+$ and $e_\beta\in V^+_\lambda$ for the $F_4$-weight $\lambda$ described in Table~\ref{tab:e7 to f4}.

\end{lemma}

\begin{proof}
Let $g=h_{\alpha_1}(\kappa_1)h_{\alpha_2}(\kappa_2)h_{\alpha_3}(\kappa_3)h_{\alpha_4}(\kappa_4)$ be an arbitrary element in the standard maximal torus of $G=F_4$. By construction this is the element $ h_{\beta_2}(\kappa_1)h_{\beta_4}(\kappa_2)h_{\beta_3}(\kappa_3)h_{\beta_5}(\kappa_3)h_{\beta_1}(\kappa_4)h_{\beta_6}(\kappa_4)$ in $E_7$. We give one example of how to conclude. Consider $\beta=1123321$. An element $h_{r}(\kappa)$ operates on $e_{\beta}$ as $h_r(\kappa).e_{\beta}=\kappa^{A_{r\beta}}$, where $A_{r\beta}=2(r,\beta)/(r,r)$ (see \cite[§3.3]{cartersimple}). This implies that $g.e_{\beta}=\kappa_1^{-1}\kappa_3 e_{\beta}$. On the other hand if we set $\lambda=0121$ we get $\lambda(g)=\kappa_1^{-1}\kappa_3$, which implies that $g.v_\lambda =\kappa_1^{-1}\kappa_3 v_\lambda$. This implies that the weight vector $v_{0121}$ corresponds to a scalar multiple of $e_{1123321}$, as claimed in Table~\ref{tab:e7 to f4}. The other cases follow similarly.
\end{proof}

Let us recall some results about the $F_4$-orbits on $V$. First, we want to understand the quadratic form on $V$ fixed by $F_4$. This can be done with the following lemma.

\begin{lemma}\label{quadratic form f4 k=2}
In $V=V_{F_4}(\lambda_4)$, a set of hyperbolic pairs of non-zero weights is given by pairs of weight vectors in opposite weight spaces.
\end{lemma}

\begin{proof}
This is simply a matter of understanding the isomorphism between $V$ and $V^*$. If $v^+$ is a highest weight vector in $V=\langle v^+\rangle \oplus V'$, with $V'$ being the sum of the other weight spaces,  and $f\in V^*$ is the map such that $f(v^+)=1$ and $f(V')=0$, then $f$ is a highest weight vector of $V^*$ of highest weight $-\lambda_4$ with respect to the opposite Borel subgroup $B^-$. Therefore $\dot{w_0} f$ is a highest weight vector of $V^*$ with respect to $B$, of highest weight $-w_0\lambda=\lambda$. If $^*:V\rightarrow V^*$ is an isomorphism, then a bilinear form fixed by $F_4$ is given by $(u,v)=v^*(u)$. This means that $(v_\gamma,v^+)=0$ whenever $\gamma\neq -\lambda_4$ and is non-zero when $\gamma = -\lambda_4$. Hence $v^+$ and $v^-$ form an hyperbolic pair for the form fixed by $F_4$, for an appropriate lowest weight vector $v^-$. Since the Weyl group is transitive on pairs of opposite weights in this module, we are done. 
\end{proof}

The $F_4$ orbits on $1$-spaces have already been determined, as can be seen from the following proposition.

\begin{proposition}\cite[$4.13$]{double1}
The group $G=F_4$ has $2$ orbits on singular $1$-spaces of $V$, with stabilizers the standard parabolic subgroup $P=U_{15}B_3T_1$ and a subgroup $U_{14}G_2T_1$ of $P$. A pair of representatives is given by $\langle x \rangle$ and $\langle y \rangle$, where $x=e_{ 2 2 3 4 3 2 1}$ and $y=e_{ 1 2 2 3 2 2 1}+e_{1 1 2 3 3 2 1}$. The stabilizers $G_x$ and $G_y$ are respectively $U_{15}B_3$ and $U_{14}G_2$.
\end{proposition}
For the purpose of our analysis we explicitly construct $U_{15}B_3$ and $U_{14}G_2$. We use the ordering of the $F_4$-roots, with $X_i$ denoting $X_{\alpha_i}$, and $X_{-i}$ denoting $X_{-\alpha_i}$. Removing the last node of the Dynkin diagram for $F_4$ get the maximal parabolic fixing $\langle x \rangle$, so that $B_3=\langle X_{\pm 1},X_{\pm 2},X_{\pm 3}\rangle $. 
The subgroup $G_2\leq B_3$ can then be constructed  as $G_2=\langle X_{\pm 2},x_1(t)x_3(t),x_{-1}(t)x_{-3}(t) :t\in K\rangle$. We denote these specific copies of $B_3$ and $G_2$ as $S^{(x)}$ and $S^{(y)}$ respectively.

We call the points in the orbit with representative $\langle x \rangle$ white points, and all points in the orbit $\langle y \rangle ^G$ grey points. To give some more detail about the stabilizers of $x$ and $y$ we have the following lemma.

\begin{lemma}
Let $x=e_{ 2 2 3 4 3 2 1}$ and $y=e_{ 1 2 2 3 2 2 1}+e_{1 1 2 3 3 2 1}$ and let $v\in \{x,y\}$.
Then the root subgroups (with respect to $T\cap G_v$) of $G_v$ are as in Table~\ref{tab:root subgroups of stabilizers}, where by $x_i(t_1)x_j(t_2)$ we mean the subgroup $\langle x_i(t_1)x_j(t_2):t_1,t_2\in K \rangle$. 
\newpage
\begin{center}
\begin{xltabular}{\textwidth}{l l X X }
\caption{Root subgroups of $G_x$ and $G_y$} \label{tab:root subgroups of stabilizers} \\

\hline
\multicolumn{1}{c}{$v$} & \multicolumn{1}{c}{$G_v$} & \multicolumn{1}{c}{Root subgroups of $S^{(v)}$} & \multicolumn{1}{c}{Root subgroups of $R_u(G_v)$}  \\ \hline 
\endhead
$x$& $U_{15}B_3$ & $X_i:i\in \pm\{1,2,3,5,6,8,9,11,14\}$& $X_i:i\in \{4,7,10,12,13,15,16,17,18,19,20,\allowbreak21,22,23,24\}$ \\*
 \hline
 $y$ &$U_{14}G_2$& $X_{\pm 2},X_{\pm 11},X_{\pm 14},x_1(t)x_3(t),\allowbreak x_5(-t)x_6(t),x_8(t)x_9(t), x_{-1}(t)x_{-3}(t),\allowbreak x_{5}(-t)x_{-6}(t), x_{-8}(t)x_{-9}(t)$& $X_i, x_{12}(t)x_{13}(t):i \in \{4,7,10,15,16,17,18,\allowbreak 19,20,21,22,23,24\}$\\*
 \hline
\end{xltabular}
\end{center}
\end{lemma}

\begin{proof}
For $U_{15}B_3$ it is clear what the root subgroups are. For $U_{14}G_2$ it suffices to check that the listed ones do indeed fix $y$. 
\end{proof}

We begin by finding the first two orbits for the $G$-action on totally singular $2$-spaces. This can be done as a consequence of work in \cite[$3.2.18$]{generic2}.

\begin{proposition}\label{f4 orbits 13 and 15}
There are totally singular $2$-spaces in $V_{F_4}(\lambda_4)$ with stabilizers $U_1A_2.2$ and $A_1A_2.2$. In both cases the $A_2$ factor is a long root $A_2$ acting trivially on the fixed $2$-space. 
\end{proposition}

\begin{proof}
As in \cite[$3.2.18$]{generic2}, let $A$ be a long root $A_2\leq F_4$ having simple roots $\alpha_1$ and $\alpha_{23}=1342$, and $\tilde{A_2}$ the short root $A_2$ having simple roots $\alpha_3$ and $\alpha_4$.  Also let $\gamma_{ij}$ denote the $E_7$ root in position $i,j$ in the matrix 
$$\left(\begin{matrix} 0112111 & 0112211 & 0112221 \\ 1112111 & 1112211  & 1112221 \\ 1122111 & 1122211 & 1122221 \\ \end{matrix}\right) .$$ Let $\overline{V}$ be the $7$-dimensional space $\overline{V}=\{\sum a_{ij} e_{\gamma_{ij}} + \langle v_0 \rangle : a_{22}+a_{31}+a_{13}=0 \}$, on which $\tilde{A_2}$ acts as $V_{\tilde{A}_2}(\lambda_1+\lambda_2)$, again as described in \cite[$3.2.18$]{generic2}. Let $Y$ be the set of $2$-spaces of $\overline{V}$.
In \cite[$3.2.17$]{generic2} the authors construct a set $\hat{Y_1}\subset Y$, with the key property that for any $y\in \hat{Y_1}$, the transporter $\mathrm{Tran}_G(y,Y)$ is $A\tilde{A_2}\langle n^* \rangle$, where $n^*$ is an involution which sends $e_{\gamma_{ij}}$ to $e_{\gamma_{ji}}$ for all $1\leq i,j \leq 9$. The subgroup $A$ fixes all vectors in $\overline{V}$.
Therefore, given a totally singular $2$-space $W_2\in  \hat{Y_1}$, we have that $G_{W_2}=A(\tilde{A}_2\langle n^*\rangle)_{W_2}$.

We have already classified the orbits on totally singular $2$-spaces in $V_{\tilde{A}_2}(\lambda_1+\lambda_2)$ in Proposition~\ref{proposition a2 k=2}. In Proposition~\ref{proposition a2 k=2} the $2$-spaces are given in terms of vectors in $\mathrm{Lie}(\tilde{A}_2)$. We can understand the isomorphism between the two modules by restricting the $F_4$ weights to $\tilde{A}_2$. Then Lemma~\ref{lemma e7 f4 weights} shows that $e_{\gamma_{11}},e_{\gamma_{12}},e_{\gamma_{21}}$ correspond respectively to $e_{\alpha_3},e_{\alpha_2},e_{\alpha_1}$, while $e_{\gamma_{33}},e_{\gamma_{23}},e_{\gamma_{32}}$ correspond respectively to $e_{-\alpha_3},e_{-\alpha_2},e_{-\alpha_1}$, where all the $e_{\alpha_i}$ are as in Proposition~\ref{proposition a2 k=2}. 

Without giving the full description of $\hat{Y}_1$, it suffices to note that 
if \begin{flalign*}
v^{(1)}&=a_{33}e_{\gamma_{33}} + a_{12}e_{\gamma_{12}} + a_{21}e_{\gamma_{21}},&\\
v^{(2)}&=b_{11}e_{\gamma_{11}} + b_{23}e_{\gamma_{23}} + b_{32}e_{\gamma_{32}},&
\end{flalign*} then by the proof of \cite[$3.2.17$]{generic2} the $2$-space $\langle v^{(1)},v^{(2)}\rangle\in  \hat{Y}_1$ if and only if $$(a_{12}b_{23}-a_{33}b_{11})(a_{21}b_{32}- a_{12}b_{23})(a_{33}b_{11}-a_{21}b_{32})\neq 0.$$ 

First consider the totally singular $2$-space given by $\langle v^{(1)},v^{(2)}\rangle$ where $v^{(1)}=e_{\gamma_{21}}+e_{\gamma_{12}}$ and $ v^{(2)}=e_{\gamma_{23}}-e_{\gamma_{32}}$. This corresponds precisely to $W_2^{(2)}$ in our classification of $A_2$-orbits on totally singular $2$-spaces in Proposition~\ref{proposition a2 k=2}. It is an element of $\hat{Y_1}$, has stabilizer $A_1$ in $\tilde{A_2}$, and is at the same time fixed by $n^*$. Therefore it has stabilizer $A_1A_2.2$ in $G$.

Finally consider the totally singular $2$-space given by $\langle v^{(1)},v^{(2)}\rangle$ where $v^{(1)}=e_{\gamma_{33}} + e_{\gamma_{12}} + e_{\gamma_{21}}$ and $ v^{(2)}= e_{\gamma_{23}} - e_{\gamma_{32}}$.
This $2$-space corresponds to $W_2^{(*)}$ in Lemma~\ref{equivalent 2-space with u1 stab a2 k=2}, where it is shown that it is in the same $\tilde{A}_2$-orbit as $W_2^{(1)}$ from Proposition~\ref{proposition a2 k=2}. Therefore it has stabilizer $U_1$ in $\tilde{A}_2$, and we also see that it is a member of $\hat{Y}_1$. Since it is fixed by $n^*$, the $2$-space  $\langle v^{(1)},v^{(2)}\rangle$ has stabilizer $U_1A_1.2$ in $G$.
\end{proof}

In Table~\ref{tab:f4 orbits list} we define a list of $2$-spaces of $V$, in terms of $x,y$ and specific elements of the Weyl group of $F_4$. We write the Weyl group elements both as permutations and as the corresponding products of fundamental $n_i$'s $\in N_G(T)$. The first column of Table~\ref{tab:f4 orbits list} fixes a numbering, indexed by $i$, that will be used throughout this section.
Orbits number $13$ and $15$ are missing, since they are dealt with in Proposition~\ref{f4 orbits 13 and 15}. 
Since the $2$-spaces are defined in terms of specific elements of $N_G(T)$, we need to justify the description in terms of explicit vectors, which is done in Lemma~\ref{table f4 justification of the explicit vectors}.
\begin{center}
\begin{xltabular}[h]{\textwidth}{l l X }
\caption{List of $2$-spaces} \label{tab:f4 orbits list} \\
\hline \multicolumn{1}{c}{ $i$} & \multicolumn{1}{c}{$W_2$ } & \multicolumn{1}{c}{$w_i$, $\dot{w}_i$}\\
& \multicolumn{1}{c}{$\dot{w}_i.x$ or $\dot{w}_i.y$} &  \\
\hline 
\endhead
 \hline
 1& $\langle x,\dot{w}_1.x\rangle $ & $(3, 7)(4, 28)(6, 10)(8, 12)(9, 16)(11, 18)(14, 20)(19, 21)(27, 31)$\\
 & $ e_{1 2 3 4 3 2 1}$& $(30, 34)(32,
    36)(33, 40)(35, 42)(38, 44)(43, 45),$ $n_4$\\
 \hline
  2 &$\langle x,\dot{w}_2.y\rangle $ & $id$\\ 
  &$e_{1223221}+e_{1123321}$&\\
 \hline

  3& $\langle x,\dot{w}_3.x\rangle $ & $(1, 18)(4, 37)(5, 20)(7, 34)(8, 21)(10, 31)(11, 22)(13, 28)$ \\
  & $e_{1 2 2 3 2 1 1}$& $(14, 23)(16, 40)(25,
    42)(29, 44)(32, 45)(35, 46)(38, 47),$ $n_4n_3n_2n_3n_4$\\
 \hline
    4&$\langle x,\dot{w}_4.y\rangle $ &$(3, 7)(4, 28)(6, 10)(8, 12)(9, 16)(11, 18)(14, 20)$\\
    & $-e_{1 2 2 3 2 1 1}+e_{1 1 2 3 3 2 1}$& $(19, 21)(27, 31)(30, 34)(32,
    36)(33, 40)(35, 42)(38, 44)(43, 45),$ $ n_4$\\
 \hline
  
   5& \begin{tabular}{@{}c@{}}$\langle x,\dot{w}_5.x\rangle $ \\ $e_{0 1 1 2 2 1 1}$\end{tabular}  &$(1, 47, 33, 46, 35, 2, 25, 23, 9, 22, 11, 26)(3, 15, 34, 30, 37, 43, 27, 39, 10,
    6,\allowbreak 13, 19) (4, 31, 32, 17, 36, 21, 28, 7, 8, 41, 12, 45)(5, 48, 20, 42, 16,
    14, 29,\allowbreak 24, 44, 18, 40, 38),n_1 n_2 n_1 n_3 n_2 n_1 n_4 n_3 n_2 n_1 n_3 n_2 n_3 n_4  $\\
  
\hline
   6&$\langle x,\dot{w}_6.y\rangle $ &$(1, 18)(4, 37)(5, 20)(7, 34)(8, 21)(10, 31)(11, 22)(13, 28)$ \\
  & $-e_{1 2 2 3 2 2 1}+e_{0 1 1 2 2 1 1}$& $(14, 23)(16, 40)(25,
    42)(29, 44)(32, 45)(35, 46)(38, 47),$ $n_4n_3n_2n_3n_4$\\
 \hline
    7& \begin{tabular}{@{}l@{}}$\langle y,\dot{w}_7.y\rangle $ \\ $-e_{1 2 2 3 2 1 1}+e_{1 1 2 3 2 2 1}$\end{tabular} & $(2, 9, 16)(3, 4, 31)(5, 11, 18)(6, 13, 10)(7, 27, 28)(8, 15, 12)(14, 22, 20)\allowbreak (17,
    19, 21) (26, 33, 40)(29, 35, 42)(30, 37, 34)(32, 39, 36)(38, 46, 44)\allowbreak (41, 43,45),$ $n_4n_3$\\
 \hline
    8& \begin{tabular}{@{}l@{}}$\langle x,\dot{w}_8.y\rangle $ \\ $e_{1 2 2 3 2 1 1}-e_{0 0 1 1 1 1 1}$\end{tabular}  & $(2, 23, 9, 29, 48, 35)(3, 41, 43, 27, 17, 19)(4, 31, 21, 28, 7, 45)(5, 24, 11,
    26,\allowbreak 47, 33)(6, 13, 36, 32, 39, 10)(8, 15, 34, 30, 37, 12)(14, 22, 44, 38,
    46, 20)\allowbreak(16, 42)(18, 40),$ 
 $n_3 n_2 n_1 n_4 n_3 n_2 n_1 n_3 n_2 n_3 n_4 $
 \\
 \hline
   9&$\langle x,\dot{w}_9.y\rangle $ & $(4, 43)(7, 41)(10, 39)(12, 37)(13, 36)(15, 34)$\\
   &$-e_{0 1 0 1 1 1 1}+e_{0 0 1 1 1 1 1}$ & $(16, 48)(17, 31)(18, 47)(19,
    28)(20, 46)(21, 45)(22, 44)(23, 42)(24, 40),$\\
    && $n_4 n_3 n_2 n_1 n_3 n_2 n_3 n_4 n_3 n_2 n_1 n_3 n_2 n_3 n_4 $
    \\
\hline
   10& \begin{tabular}{@{}l@{}}$\langle y,\dot{w}_{10}.y\rangle $ \\ $-e_{1 1 2 3 2 1 1}-e_{1 1 2 2 2 2 1}$\end{tabular}  & $(1, 2, 11, 20, 9, 18, 25, 26, 35, 44, 33, 42)(3, 4, 36, 30, 39, 34, 27, 28, 12,
    6,\allowbreak 15, 10)(5, 14, 24, 23, 22, 16, 29, 38, 48, 47, 46, 40)(7, 32, 41, 43, 45,
    37, 31,\allowbreak 8, 17, 19, 21, 13),$ $n_4n_3n_2n_1$\\
 \hline
  11&  \begin{tabular}{@{}l@{}}$\langle y,w_{11}.y\rangle $ \\ $-e_{1 1 2 3 2 2 1}+e_{0 1 1 2 2 1 1}$\end{tabular}  & $(1, 16, 42)(2, 5, 20)(4, 39, 37)(6, 8, 21)(7, 36, 34)(9, 11, 22)(10, 31, 12)\allowbreak (13,
    28, 15)(14, 24, 23)(18, 25, 40)(26, 29, 44)(30, 32, 45)(33, 35, 46)(38,\allowbreak 48,
    47),$ $n_4n_3n_2n_1n_3n_4$
\\
 \hline
  12&  \begin{tabular}{@{}l@{}}$\langle y,\dot{w}_{12}.y\rangle $ \\ $-e_{0 1 1 1 1 1 1}+e_{1 0 1 1 1 1 1}$\end{tabular} &$ (2, 9, 16, 48)(3, 4, 45, 17)(5, 11, 18, 47)(6, 13, 32, 39)(7, 43, 31, 19)(8, 15,\allowbreak
    30, 37)(10, 36)(12, 34)(14, 22, 38, 46)(20, 44)(21, 41, 27, 28)(23, 29, 35,\allowbreak
    42)(24, 26, 33, 40),$
    $n_4 n_3 n_2 n_1 n_3 n_2 n_3 n_4 n_3 n_2 n_1 n_3 n_2 $
    \\
 \hline
   14&  \begin{tabular}{@{}l@{}}$\langle y,w_{14}.y\rangle $ \\ $e_{1 1 1 1 1 1 1}-e_{0 1 1 2 1 1 1}$\end{tabular} & $(1, 38, 44, 25, 14, 20)(2, 24, 33, 26, 48, 9)(3, 39, 19, 27, 15, 43)(4, 34, 32,\allowbreak
    28, 10, 8)(5, 18, 47, 29, 42, 23)(6, 17, 12, 30, 41, 36)(7, 45, 13, 31, 21,
    37)\allowbreak (11, 35)(16, 40)(22, 46),$
   $n_1 n_3 n_2 n_4 n_3 n_2 n_1 n_3 n_2 n_3 n_4 n_3 n_2 n_3 $
\\
 \hline
 
\end{xltabular}
\end{center} 

We will prove that together with the orbits found in Proposition~\ref{f4 orbits 13 and 15} this is a complete list of orbit representatives for the action of $G$ on $P_2^{TS}(V)$. Before diving into the proof, we want to understand how we can write $\dot{w}_i.x$ and $\dot{w}_i.y$ as elements in $\mathrm{Lie}(E_7)$.  Of course we can  easily see that $\dot{w}_i.x=\pm e_\beta$ and $\dot{w}_i.y=\pm e_{\beta'}\pm e_{\beta''}$, for some known $E_7$-roots $\beta,\beta',\beta''$. The difficulty lies in understanding the signs. We will deal with this by computing in the Lie group $E_7(3)$ using Magma. 

\begin{lemma}\label{table f4 justification of the explicit vectors}
Let $i\in \{1,\dots,12,14\}$ be an arbitrary orbit number. Then $\dot{w}_i.x$ or $\dot{w}_i.y$ (as appropriate), is as in Table~\ref{tab:f4 orbits list}.
\end{lemma}
\begin{proof}
Recall that $n_1,n_2,n_3,n_4$ are the standard elements in $N_{F_4}(T)$. By definition these are $x_{\alpha_i}(1)x_{-\alpha_i}(-1)x_{\alpha_i}(1)$ for $1\leq i \leq 4$. We can then express these as elements in $E_7$, by substituting $x_{\pm \alpha_i}(1)$ with the appropriate product  in $E_7$. This gives us elements $n_2',n_4',n_3'{n_5'}^{-1},n_1'{n_6'}^{-1}$ respectively, where $n_1',\dots n_6' \in E_7$. Given $\dot{w}_i\in N_G(T)$ as in Table~\ref{tab:f4 orbits list}, we can now determine $\dot{w}_i.x$ or $\dot{w}_i.y$. As an example let us consider $\dot{w}_3=n_4n_3n_2n_3n_4$. This is the same as $ n_1'{n_6'}^{-1}n_3'{n_5'}^{-1}n_4'n_3'{n_5'}^{-1}n_1'{n_6'}^{-1}$, which is naturally an element in $E_7(3)$. This allows us to compute that indeed $\dot{w}_3.x=e_{1 2 2 3 2 1 1}$, rather than $-e_{1 2 2 3 2 1 1}$.
\end{proof}

We can then prove the following proposition.

\begin{lemma}
The $2$-spaces listed in Table~\ref{tab:f4 orbits list} are totally singular.
\end{lemma}

\begin{proof}
By Lemma~\ref{quadratic form f4 k=2}, the hyperbolic pairs for the quadratic form fixed by $G$ are given by pairs of opposite weight vectors. Recall that $x=e_{ 2 2 3 4 3 2 1}$ and $y=e_{ 1 2 2 3 2 2 1}+e_{1 1 2 3 3 2 1}$, and that $x$ and $y$ are singular vectors. By Table~\ref{tab:e7 to f4}, the vectors $ e_{0000001}, e_{0011111}, e_{0101111}$ are opposite weight vectors to $e_{ 2 2 3 4 3 2 1}, e_{ 1 2 2 3 2 2 1},e_{1 1 2 3 3 2 1}$ respectively. Looking at $\dot{w}_i.x$ or $\dot{w}_i.y$ as appropriate, we find that they are perpendicular to $x$ or $y$ as appropriate. This allows us to conclude that the $2$-spaces listed are totally singular.
\end{proof}

Before stating the main proposition let us prove the following easy lemmas.

If a $2$-space contains at least $2$ white points, then its centralizer is obtained by intersecting two conjugates of $ U_{15}B_3$. Since these intersections are easily derived from the intersections of conjugates of the parabolic $P=U_{15}B_3T_1$, we can find all possible structures occurring.

\begin{lemma}\label{parabolic intersections}
The possible intersections of two conjugates of $P=U_{15}B_3T_1$ are $P$, $U_{20}A_2T_2$, $U_{15}C_2T_2$, $U_{13}A_2T_2$ and $B_3T_1$. 
\end{lemma} 

\begin{proof}
If $W_P$ is the subgroup of the Weyl group $W=W(F_4)$ corresponding to $P$, then there are five $(W_P,W_P)$-double cosets in $W$, of sizes $48, 384, 288, 384, 48$. By a well known correspondence (see \cite[$2.8.1$]{carterfinite}) this means that there are five $(P,P)$-double cosets in $F_4$, with representatives being pre-images in the five $(W_P,W_P)$-double cosets. A set of such representatives is given by $ \{id, n_4,n_4 n_3 n_2 n_3 n_4 ,n_2 n_3 n_2 n_4 n_3 n_2 n_1 n_3 n_2 n_3 n_4 n_3 n_2 n_1 n_3 , \dot{w}_0\}$, where $w_0$ is the longest element of the Weyl group.
The stated structures can then be found using \cite[$2.8.7$]{carterfinite}.
\end{proof}

\begin{lemma}\label{two white}
If $W_2$ is a $2$-space containing at least $2$ white points then $C_{F_4}(W_2)$ has dimension at least $21$. If $W_2$ is a totally singular $2$-space containing at least $1$ white point then $C_{F_4}(W_2)$ has dimension at least $12$.
\end{lemma}

\begin{proof}
By Proposition~\ref{parabolic intersections} the smallest intersections of conjugates of $U_{15}B_3$ are $B_3$ and a subgroup of $U_{13}A_2T_2$ containing $U_{13}A_2$. In the second case this can viewed as the centralizer of $W_2^{(5)}$, which makes it easy to check that it contains no non-trivial element of $T_2$. The first statement follows. For the second statement, dimensional considerations tell us that $\dim C_{F_4}(W_2)$ is at least $\dim U_{14}G_2+\dim U_{15}B_3 -\dim F_4= 28+36-52=12$.
\end{proof}

Let us state the main proposition for this section. 

\begin{proposition}\label{main proposition f4 k=2}
Let $W_2$ be a $2$-space in Table~\ref{tab:f4 orbits list}. Then the stabilizer $(F_4)_{W_2}$, the centralizer $C_{F_4}(W_2)$, and the number of white points in $W_2$ are as in Table~\ref{tab:f4 orbits}.
\begin{center}
\begin{xltabular}[h]{\textwidth}{l l l l l l l }
\caption{$F_4$-orbits on totally singular $2$-spaces} \label{tab:f4 orbits} \\
\hline \multicolumn{1}{c}{$i$} & \multicolumn{1}{c}{$W_2^{(i)}$ } & \multicolumn{1}{c}{$(F_4)_{W_2}$}& \multicolumn{1}{c}{Centralizer}& \multicolumn{1}{c}{$P\cap P^{w_i}$}& \multicolumn{1}{c}{\# white points}& \multicolumn{1}{c}{Orbit dimension}\\
\hline 
\endhead
 $1$& $\langle x,\dot{w}_1.x\rangle $ & $ U_{20}A_2A_1T_1$ & $U_{20}A_2$& $U_{20}A_2T_2$&  all  & $20$\\

  $2$ &$\langle x,\dot{w}_2.y\rangle $ &$U_{15}G_2T_1$   &$U_{14}G_2$ & $U_{15}B_3T_1$&  $ 1$   & $22$ \\

  $3$& $\langle x,\dot{w}_3.x\rangle $ &$U_{15}C_2A_1T_1$ & $U_{15}C_2$& $U_{15}C_2T_2$&  all  & $23$ \\

    $4$&$\langle x,\dot{w}_4.y\rangle $ &$U_{19}A_1T_2$  & $U_{18}A_1$& $U_{20}A_2T_2$ &     $1$   & $28$\\

   $5$& $\langle x,\dot{w}_5.x\rangle $ &$U_{13}A_2T_2.2$  &  $U_{13}A_2$& $U_{13}A_2T_2$&   $2$  & $29$\\

   $6$&$\langle x,\dot{w}_6.y\rangle $ &$U_{15}A_1T_2$  &   $U_{14}A_1$& $U_{15}C_2T_2$ &  $1$   & $ 32$\\

    $7$&$\langle y,\dot{w}_7.y\rangle $ &$U_{14}A_1T_1$  &  $U_{14}$& $U_{20}A_2T_2$&   $0$   & $34$\\

    $8$&$\langle x,\dot{w}_8.y\rangle $ &$U_{12}A_1T_2$  &   $U_{11}A_1$& $U_{13}A_2T_2$&  $1$   & $35$\\

   $9$&$\langle x,\dot{w}_9.y\rangle $ &$U_1G_2T_1$  &   $G_2$& $B_3T_1$ & $1$ & $36$\\

   $10$&$\langle y,\dot{w}_{10}.y\rangle $ &$U_{12}A_1T_1$  & $U_{11}A_1$&$U_{20}A_2T_2$ &   $0$   & $36$\\

  $11$&$\langle y,\dot{w}_{11}.y\rangle $ &$U_{10}A_1T_1$  &   $U_{10}$& $U_{15}C_2T_2$ & $0$   & $38$\\

  $12$&$\langle y,\dot{w}_{12}.y\rangle $ &$U_8A_1T_1$  &   $U_8$&$U_{13}A_2T_2$ & $0$   & $40$\\

   $13$& &$A_2A_1.2$  &   $A_2$& & $0$   & $41$\\

   $14$&$\langle y,\dot{w}_{14}.y\rangle $ &$U_6A_1T_1$  &   $U_5A_1$& $U_{13}A_2T_2$ & $0$   & $42$\\

   $15$& &$U_1A_2.2$  &  $A_2$& &  $0$   & $43$\\
 \hline
 
\end{xltabular}
\end{center} 
 
\end{proposition}
Again, note that orbits number $13$ and $15$ are dealt with in Proposition~\ref{f4 orbits 13 and 15}, where the representatives are given explicitly and not in terms of Weyl group elements.

We prove Proposition~\ref{main proposition f4 k=2} by splitting the work into lemmas, one for each orbit, and then conclude by passing to finite fields. Before we begin let us discuss the strategy. First note that in Table~\ref{tab:f4 orbits} we give information about the number of white points contained in a given $2$-space. This will be justified throughout the proofs, depending on the individual case. We have already seen that the listed $2$-spaces are totally singular. This means that they only contain white or grey points.

We begin with the $2$-spaces that are defined in terms of a basis of grey vectors, i.e. numbers $7,10,11,12,14$. For each $\dot{w}=\dot{w}_i$ and $W_2=W_2^{(i)}$ with $i\in {7,10,11,12,14}$, we first determine
\begin{flalign*}
H_1 &=U_{15}B_3 \cap (U_{15}B_3)^{\dot{w}};\\
H_2 &=U_{15}B_3 \cap (U_{14}G_2)^{\dot{w}};\\
H_3 &=U_{14}G_2 \cap (U_{15}B_3)^{\dot{w}};\\
H_4 &=U_{14}G_2 \cap (U_{14}G_2)^{\dot{w}}.
\end{flalign*}
We then determine $(F_4)_{W_2}/H_4$, which gives us the full stabilizer of $W_2$. The intersection $H_1$ is easily derived from intersecting conjugates of $P=U_{15}B_3T_1$. These intersections are listed in Proposition~\ref{parabolic intersections}.

\begin{lemma}\label{f4 orbit 14}
The stabilizer of $W_2=\langle y,\dot{w}_{14}.y\rangle$ is isomorphic to $U_6A_1T_1$.
\end{lemma}

\begin{proof}
The first step is to get  
\begin{center}
 \begin{tabular}{c c c c} 

 $H_1$ & $H_2$& $H_3$& $H_4$\\ [0.5ex] 
 \hline
  $U_{13}A_2$ & $U_{11}A_1$& $U_6A_2$& $U_5A_1$\\

\end{tabular}.
\end{center}
We have $H_1=U_{13}A_2$. Here $U_{13}=\langle X_i \rangle_i$ for $i\in \{1,4,5,8,12,15,17,18,20,24,-3,-6,-9\}$, and $A_2=\langle X_j\rangle_j$ for $j\in \{\pm 2,\pm 11,\pm 14\}$. By \cite[Thm.~17.6]{MT}, determining the action of $A_2$ on $U_{13}$ is a matter of root levels and shapes in $U_{13}$. Conjugating by an appropriate element  $n\in N_G(T)$, so that $U_{13}^n$ is contained in the standard Borel subgroup $B$, we find that there are $5$ different levels of roots in $U_{13}$. They each have a unique root shape, and from the lowest level to the highest they contain roots $\{\alpha_i: i\in {4,15,17}\}$, $\{\alpha_i: i\in {8,-3,-6}\}$, $\{\alpha_i: i\in {18,20,24}\}$, $\alpha_{12}$ and $\{\alpha_i: i\in {1,5,-9}\}$. As described in Section~\ref{notation section} this gives quotients with the structure of irreducible $A_2$-modules of dimension $3,3,3,1,3$ respectively. This means that  $\langle X_1,X_5,X_{-9} \rangle$ is stabilised by $A_2$ and has the structure of an irreducible $3$-dimensional $A_2$-module. We can then explicitly check that in fact $A_2$ normalises both $ X_{12}$ as well as $\langle X_{18},X_{20},X_{24} \rangle $, which means that also $\langle X_{18},X_{20},X_{24} \rangle $ has the structure of an irreducible $3$-dimensional $A_2$-module. Finally, we find that $A_2$ normalises $U_6^{(1)}:=\langle X_{1},X_{5},X_8,X_{-3},X_{-6},X_{-9}\rangle$ as well as $U_6^{(2)}:= \langle X_{4},X_{15},X_{17},X_{18},X_{20}, X_{24}\rangle$.

By looking at the permutation of the roots induced by $w=w_{14}$, we can find the generators of $(U_{14}G_2)^{\dot{w}}$ and get  a subgroup $H_2^\star= U_{11}A_1$ of $H_2$, with $A_1=\langle X_{\pm 11}\rangle$ and
$$R_u(H_2^\star)=U_{11}=\langle X_{i}, x_{18}(t)x_{17}(t),x_{2}(t)x_{4}(t),x_{14}(t)x_{15}(-t) : i \in \{1,5,8,12,20,24,-3,-9\},t\in K\rangle.$$

We note that $U_{11} \cap U_{13}=U_9=\langle X_{i}, x_{18}(t)x_{17}(t):t\in K\rangle_i $ for $i \in \{1,5,8,12,20,24,-3,-9\}$ and that the projection of $H_2^\star$ on $A_2$ is $U_2A_1=\langle X_i\rangle_i$ for $i\in \{\pm 11,2,14\}$. If the projection of $H_2$ on $A_2$ was larger than $U_2A_1$ and not contained in $U_2A_1T_1$, it would be the full $A_2$, since $U_2A_1T_1$ is the only maximal connected subgroup of $A_2$ containing $U_2A_1$. It is then easy to see that the $A_2$ action on $U_{13}$ combined with the fact that $U_9\leq R_u(H_2)$, would imply that $ U_{13}\leq R_u(H_2)$ and therefore $H_2=H_1$. However $X_2\leq A_2\leq H_1$ does not fix $\dot{w}.y$, which implies that $H_1\ne H_2$, a contradiction. Therefore the projection of $H_2$ on $A_2$ is contained in the parabolic subgroup $U_2A_1T_1$.

We now consider the action of $U_2A_1$ on $$U_{13}=\langle X_i \rangle_i, i\in \{1,4,5,8,12,15,17,18,20,24,-3,-6,-9\},$$ aiming to show that the minimal overgroups of $U_9\leq U_{13}$ that are stabilised by $U_2A_1$ are not contained in $H_2$. In fact a quick check shows that $U_9\lhd U_{13}$. Therefore we simply need to understand the $U_2A_1$ action on $U_{13}/U_9=\langle X_4,X_{15},X_{17},X_{-6},U_9\rangle/U_9= \langle \overline{X_4},\overline{ X_{15}},\overline{X_{17}},\overline{X_{-6}}\rangle$. In $U_{13}/U_9$, the subgroups $\overline{X_{-6}}$ and $\langle \overline{X_{4}},\overline{X_{15}},\overline{X_{17}} \rangle$ commute and are fixed by the $U_2A_1$-action. At the same time $U_2A_1$ stabilises $\overline{X_{17}}$ and acts irreducibly on  $\langle \overline{X_{4}},\overline{X_{15}},\overline{X_{17}} \rangle/\overline{X_{17}}$. Hence any non-trivial $U_2A_1$-invariant subgroup of $\langle \overline{X_{4}},\overline{X_{15}},\overline{X_{17}} \rangle/\overline{X_{17}}$ contains a non-trivial element in $\overline{X_{17}}$.
Therefore any non-trivial $U_2A_1$-invariant subgroup of $U_{13}/U_9$ contains a non-identity element $\overline{x_{-6}(\alpha)x_{17}(\beta)}$. 

Suppose that $ H_2\neq H_2^\star$, i.e. that $H_2$ is a proper overgroup of $H_2^\star$ in $H_1$. If $R_u(H_2)=R_u(H_2^\star)$, since the projection of $ H_2$ on $A_2$ is contained in $ U_2A_1T_1$, we get that 
$H_2\leq H_2^\star T_1^\star$ for some $T_1^\star\leq N_{H_1}(H_2^\star)$. We now restrict the possibilities for $T_1^\star$, by determining $N_{H_1}(H_2^\star)$. We find a $U_{13}A_1T_1\leq N_{H_1}(H_2^\star)$, where $U_{13}=\langle R_u(H_2), X_{17},X_{-6}\rangle$ and $T_1=\{h_{\alpha_2}(\kappa):\kappa\in K^*\}$. The same reasoning as in the previous paragraph shows that this is the full $N_{H_1}(H_2^\star)$.
Since $N_{H_1}(H_2^\star)/H_2^\star=\langle T_1,X_{17},X_{-6},H_2^\star\rangle/H_2^\star$, we can assume that $T_1^\star=T_1^g$ for some $g\in X_{-6}X_{17}$. Therefore, since $H_2\neq H_2^\star$, we have that $H_2$ contains a non-trivial element of the form $h_{\alpha_2}(\kappa) x_{-6}(\alpha)x_{17}(\beta)$. 
On the other hand suppose that $ R_u(H_2)\neq R_u(H_2^\star)$. Let $h\in R_u(H_2)\setminus R_u(H_2^\star)$. Since $h\in U_{13}A_2$, we can write $h$ as $ ul$ for $u\in U_{13},l\in A_2$. In fact, since $h$ is unipotent and the projection  of $ H_2$ on $A_2$ is contained in $ U_2A_1T_1$, we must have $l\in U_2A_1$. Multiplying $h$ by an element of $H_2$ which projects onto $A_2$ as $l^{-1}$, we can assume that $h$ projects trivially onto $A_2$. Then by the previous paragraph $H_2$ contains an element $x_{-6}(\alpha)x_{17}(\beta)$, with either $\alpha$ or $\beta$ non-zero. 

We will now directly show that non-trivial elements of the form $g=h_{\alpha_2}(\kappa) x_{-6}(\alpha)x_{17}(\beta)$ are not in $(U_{14}G_2)^{\dot{w}_{14}}$.
This is equivalent to $g.(\dot{w}_{14}.y)\neq \dot{w}_{14}.y$. It is easy to see that $\{h_{\alpha_2}(\kappa):\kappa\in K^*\}\cap (U_{14}G_2)^{\dot{w}_{14}}=1$ and therefore $g=h_{\alpha_2}(\kappa) x_{-6}(\alpha)x_{17}(\beta)$ with one of $\alpha$ and $\beta$ being non-zero.
Since $\dot{w}_{14}.y=e_{1 1 1 1 1 1 1}-e_{0 1 1 2 1 1 1}$, we just need to compute the action of $g$ on $e_{1 1 1 1 1 1 1}$ and $e_{0 1 1 2 1 1 1}$. We can do this directly in $F_4$, by noting that $ e_{1 1 1 1 1 1 1}=v_{-0 1 1 0}$ and  $e_{0 1 1 2 1 1 1}=v_{-0 0 1 1}$, where $v_{\lambda}$ denotes a vector in the weight space $V_{\lambda}$. The $-\alpha_{6}$ and $\alpha_{17}$ root strings through $-0110$ and $-0011$ are $(-0011,-0121)$ and $(-0110,1111)$. Furthermore 
the $-\alpha_{6}$ root string through $1111$ is trivial. This means that $ g.(\dot{w}_{14}.y)-\dot{w}_{14}.y=v_{-0121}+v_{1111}$, where at least one of the weight vectors appearing in the sum is non-zero. Hence $g\not\in H_2$, a contradiction. Therefore $H_2=H_2^\star=U_{11}A_1$.
 
 We then find a $U_{6}A_2 \leq H_3$, with $U_6=\langle X_{4},X_{15},X_{17},X_{18},X_{20},X_{24}\rangle$ and $A_2=\langle X_j\rangle_j$ for $j\in \{\pm 2,\pm 11,\pm 14\}$ (as for $H_1$). Using the action of $A_2$ on $U_{13}$, we see that the minimal overgroups of $U_6A_2$ in $H_1=U_{13}A_2$ either contain a non-trivial element $x_{12}(t)$ or a non-trivial root element in $\langle X_{1},X_{5},X_8,X_{-3},X_{-6},X_{-9}\rangle$, neither of which are contained in $U_{14}G_2$, showing that $U_{6}A_2 = H_3$.

We now find a $U_5A_1\leq H_2 \cap H_3 = H_4$, given by $$U_5A_1=\langle X_{\pm 11},x_4(t)x_2(t),x_{14}(-t)x_{15}(t),x_{18}(t)x_{17}(t),X_{20},X_{24}:t\in K \rangle.$$ Since $H_2=U_{11}A_1$, we know that $H_4/R_u(H_4)$ is $A_1$. The projection of $U_5A_1$ on $ A_2$ is $U_2A_1$ with $U_2A_1=\langle X_i\rangle_i$ for $i\in \{\pm11,2,14\}$. Therefore this is also the projection of $H_4$ on $A_2$, and $U_5A_1\leq H_4 \leq U_6U_2A_1=U_8A_1$. Suppose that there is an element in $U_8A_1\cap H_4$ but not in $ U_5A_1$. Then $H_4$ contains a non-trivial element of the form $ x_{4}(t_1)x_{14}(t_2)x_{17}(t_3)$. A quick check using root strings shows that this does not fix $y$, a contradiction. Therefore $U_5A_1= H_4$. Then, by Lemma~\ref{two white}, $W_2$ is a totally grey $2$-space.

We now exhibit a $U_1T_1$ in $N_G(H_4)/H_4$, and argue that it corresponds to the full $(G)_{W_2}/H_4$. The element $$g=x_7(1) x_{18}(-1) x_{12}(1) x_{21}(1) x_{13}(1) x_{23}(-1) x_{24}(1) x_{17}(-1) x_6(1) x_5(-1)$$ fixes $y$ and sends $\dot{w}_{14}.y$ to $\dot{w}_{14}.y-y$. This can be checked by hand using the explicit $E_7$-action, however due to the number of terms involved we refrain from explicitly writing the calculations. We also have $T_1=\{h_{\alpha_1}(\kappa)h_{\alpha_3}(\kappa):\kappa\in K^*\}\in G_{W_2}$, fixing $y$ and scaling $\dot{w}_{14}.y$ by $\kappa$. Quotienting by $H_4$, this gives us the required $U_1T_1$ in $N_G(H_4)/H_4$. Suppose that $x\in G_{W_2}/H_4$ and that $x\not\in U_1T_1$. If $x$ is not contained in the Borel subgroup $U_1T_2\leq A_1T_1$, then $\langle U_1T_1,x \rangle =A_1T_1$ and we can actually assume that $x\in U_1T_2$. Therefore suppose that $x$ acts by scalars on $y,\dot{w}_{14}.y$. Let $g$ be a pre-image of $x$ in $(F_4)_{W_2}$. Since we already exhibited $h_{\alpha_1}(\kappa)h_{\alpha_3}(\kappa)$, which fixes $y$ and scales $\dot{w}_{14}.y$ by $\kappa$, we can assume that $g\in G_{\langle y \rangle}\cap G_{\dot{w}_{14}.y}=U_{14}G_2T_1\cap (U_{14}G_2)^{\dot{w}_{14}}$. Recall that we already determined $ H_4= U_{14}G_2\cap (U_{14}G_2)^{\dot{w}_{14}}$. To reach a contradiction it would suffice to prove that $U_{14}G_2T_1\cap (U_{14}G_2)^{\dot{w}_{14}}=H_4$. If we trace back the work we did to find $H_4$, we see that we simply need to prove that $H_1=U_{15}B_3\cap (U_{15}B_3)^w=U_{15}B_3T_1\cap (U_{15}B_3)^w$ and that $H_3=U_{14}G_2\cap H_1= U_{14}G_2T_1 \cap H_1$.
For the first part we know that $$U_{13}A_2 =U_{15}B_3\cap (U_{15}B_3)^w\leq U_{15}B_3T_1\cap (U_{15}B_3)^w\leq U_{15}B_3T_1\cap (U_{15}B_3T_1)^w =U_{13}A_2T_1.$$ Therefore $H_1\leq U_{15}B_3T_1\cap (U_{15}B_3)^{\dot{w}_{14}} \leq H_1T_1$, and since no non-trivial element of $T_1$ is in  $(U_{15}B_3)^{\dot{w}_{14}}$, we find that the first inequality is in fact an equality. Since $H_1=U_{13}A_2$ and $H_3$ already contains the full $A_2$,  we get the required $H_3=U_{14}G_2\cap H_1= U_{14}G_2T_1 \cap H_1$. We have therefore finally shown that $G_{W_2}=U_5A_1.(U_1T_1)$.
\end{proof}

\begin{lemma}\label{f4 orbit 12}
The stabilizer of $W_2=\langle y, \dot{w}_{12}.y\rangle$ is isomorphic to $U_8A_1T_1$.
\end{lemma}

\begin{proof}
We aim to get  
\begin{center}
\begin{tabular}{c c c c} 

 $H_1$ & $H_2$& $H_3$& $H_4$\\ [0.5ex] 
 \hline
  $U_{13}A_2$ & $U_{11}A_1$& $U_{11}A_1$& $U_8$\\

\end{tabular}.
\end{center}

As in Lemma~\ref{f4 orbit 14}, we again have $H_1=U_{13}A_2$, this time with $$U_{13}=\langle X_i \rangle_i, i \in \{3, 4, 7, 13, 15, 16, 18, 22, -2, -5, -6, -8, -14\},$$ and $A_2=\langle X_j\rangle_j$ for $j\in \{\pm 1,\pm 9,\pm 11\}$. We have also seen how to get the $U_{11}A_1$ intersection, so that $H_2$ and $H_3$ both have the structure of a $U_{11}A_1$. More precisely $$H_2=\langle X_i, x_1(t)x_4(t),x_{11}(t)x_{13}(-t),x_{15}(t)x_{16}(t)\rangle_i \langle X_{\pm 9}\rangle$$ for $i\in \{3,7,18,22,-2,-5,-6,-14\}$, and $$H_3=\langle X_i, x_1(t)x_3(t),x_{-5}(t)x_{-6}(-t),x_{-8}(t)x_{-9}(t)\rangle_i \langle X_{\pm 11}\rangle$$ for $i\in \{4,7,15,16,18,22,-2,-14\}$. 

We find a $U_8\leq  H_2\cap H_3$ generated by $X_i, x_{1}(t)x_{3}(t)x_4(t),x_{15}(t)x_{16}(t),x_{-5}(t)x_{-6}(-t)$ for $i\in \{7,18,22,-2,-14\}$ and $t\in K$. We now show that $U_8=R_u(H_2)\cap R_u(H_3)$. Of course $U_8\leq R_u(H_2)\cap R_u(H_3)$. Let $ g \in R_u(H_2)$. Write $g$ as $u_8 x_3(t_1)x_{-5}(t_2)x_{11}(t_3)x_{13}(-t_3)$ for some $u_8\in U_8$ and assume that not all $t_i$'s are $0$. Now if $g\in H_3$, then $g.y=y$. Since $y=e_{ 1 2 2 3 2 2 1}+e_{1 1 2 3 3 2 1}= v_{1111}+v_{0121}$, we first find the $\alpha_3$-root strings through $1111$ and $0121$, which are $(1111,1121)$ and $(0121)$. Then we find the $-\alpha_5$-root strings through $1111$, $0121$ and $1121$, which are $(1111,0011)$, $(0121)$ and $(0021)$. Then we find the $ \alpha_{11}$-roots strings through $1111$, $0121$, $1121$ and $0011$, which are all trivial. Finally we find the $ \alpha_{13}$-roots strings through $1111$, $0121$, $1121$ and $0011$, which are $(1111,1232)$, $(0121)$, $(1121,1242)$ and $(0011)$. Since no cancellation can happen between the different weight vectors, we find that $g.y\neq y$. Hence $g\not\in H_3$ and $U_8= R_u(H_2) \cap R_u(H_3)$.

Hence $U_8 \lhd H_2\cap H_3$, since $U_8=R_u(H_2)\cap R_u(H_3)$ and therefore $H_4\leq N_{H_2}(U_8)\cap N_{H_3}(U_8)$. We argue that both $ N_{H_2}(U_8)$ and $N_{H_3}(U_8)$ have the structure of a $U_{11}T_1$. We deal with the $H_2$ case only, since the other is derived similarly. The standard maximal torus in $A_1=X_{\pm 9}\leq H_2$ is $ \{h_{\alpha_9}(\kappa):\kappa\in K^*\}= \{h_{\alpha_1}(\kappa)h_{\alpha_3}(\kappa)h_{\alpha_4}(\kappa):\kappa\in K^*\}$, and it is easily seen that it normalizes $U_8$. Therefore we have a $U_{11}T_1\leq N_{H_2}(U_8)\leq U_{11}A_1$. All that is left to show is that no other element of $A_1$ normalizes $ U_8$. Using a Bruhat parametrisation for $A_1$, we simply need to look at elements of the form $x_{9}(t)$ and $x_{9}(t_1)n_{\alpha_9}x_{9}(t_2)$. Suppose that there is a non-trivial element $x_9(t)\in N_{H_2}(U_8)$. Taking the commutator of $x_9(t)$ and $x_{-14}(1)\leq U_8$ gives a non-trivial element in $ X_{-5}$, which is contradiction since no element of $X_{-5}$ is in $U_8$. Finally let $g=x_{9}(t_1)n_{\alpha_9}x_{9}(t_2)$ and $u=x_{1}(1)x_{3}(1)x_4(1)\in U_8$. We show that $ u^g\not\in U_8$ by showing that it is not in $ H_3$. To prove this it is enough to show that $u^g$ does not fix $y$. This is a simple albeit lengthy calculation, where we find that $u^g.y=y+v_{0011}+v'$, for a non-zero weight vector $v_{0011}$ and a (possibly $0$) vector $v'$ in the sum of the other weight spaces. This completes the proof that $N_{H_2}(U_8)=R_u(H_2)\langle h_{\alpha_9}(\kappa):\kappa\in K^*\rangle=U_{11}T_1$. Similarly $N_{H_3}(U_8)=R_u(H_3)\langle h_{\alpha_{11}}(\kappa):\kappa\in K^*\rangle=U_{11}T_1$. 

As noted in the previous paragraph, $H_4\leq N_{H_2}(U_8)\cap N_{H_3}(U_8)$. Let $h\in N_{H_2}(U_8)\cap N_{H_3}(U_8)$ and write $h=u_1l_1=u_2l_2$ for $u_1\in R_u(H_2),l_1\in \langle h_{\alpha_9}(\kappa)H\kappa\in K^*\rangle$ and $u_2\in R_u(H_3),l_2\in \langle h_{\alpha_{11}}(\kappa)H\kappa\in K^*\rangle$. 
Since the unipotent radicals $R_u(H_2)$ and $R_u(H_3)$ generate a $14$-dimensional unipotent subgroup, we have $l_1=l_2=1$ and $u_1=u_2$.
Hence the intersection of $N_{H_2}(U_8)$ and  $N_{H_3}(U_8)$ is just the intersection of the unipotent radicals, which we have shown is $U_8$. This concludes the proof that $H_4=U_8$. This means that the $2$-space is purely grey.

We will now exhibit a full $A_1T_1$ acting faithfully on $W_2$. Consider  $$A_1=\langle x_{-19}(t)x_{-20}(-t), x_{19}(t)x_{20}(-t):t\in K\rangle.$$ Let $g=x_{19}(t)x_{20}(-t)\in A_1$. Then $g\in U_{14}G_2$ and therefore $g.y=y$. In terms of root subgroups of $E_6$ we have $ g=x_{\gamma_{29}}(t)x_{\gamma_{31}}(-t)x_{\gamma_{30}}(-t)$. Translated to $E_7$-roots we have $ g=x_{1 1 2 2 1 00}(t)x_{0 1 1 2 2 10}(-t)x_{1 1 1 2 1 10}(-t)$. We compute the action of $g$ on $ \dot{w}_{12}.y=-e_{0111111} + e_{1011111}$. This is just a matter of finding the appropriate root strings and structure constants. In the process we get:
\begin{flalign*}
& x_{1 1 2 2 1 00}(t).(-e_{0111111} + e_{1011111})=-e_{0111111} + e_{1011111};\\
& x_{0 1 1 2 2 10}(-t). (-e_{0111111} + e_{1011111})=-e_{0111111} + e_{1011111} +t e_{1123321}; \\
& x_{1 1 1 2 1 10}(-t).(-e_{0111111} + e_{1011111}+t e_{1123321}) =-e_{0111111} + e_{1011111}+t e_{1123321}+te_{ 1223221}.
\end{flalign*}

Therefore $g.(\dot{w}_{12}.y)=\dot{w}_{12}.y+ty$. In a similar fashion we find that $x_{-19}(t)x_{-20}(-t)$ fixes $\dot{w}_{12}.y$ and sends $y$ to $y+t\dot{w}_{12}.y$. This shows that the $A_1$ we defined acts faithfully on $W_2$. The intersection of $A_1$ with the standard maximal torus of $F_4$ is $\{h_{\alpha_2}(\kappa):\kappa\in K^*\}$. We also find that $h_{\alpha_1}(\kappa)h_{\alpha_2}(\kappa)h_{\alpha_3}(\kappa)h_{\alpha_4}(\kappa)$ acts faithfully on $W_2$, scaling every vector by $\kappa$. Hence we do indeed have a full $A_1T_1$ on top of the centralizer $U_8$, acting faithfully on $W_2$. Therefore $(F_4)_{W_2}=(F_4)_{\langle y,\dot{w}_{12}.y\rangle}=U_8A_1T_1$.
\end{proof}

\begin{lemma}\label{f4 orbit 11}
The stabilizer of $W_2=\langle y,\dot{w}_{11}.y\rangle $ is isomorphic to $U_{10}A_1T_1$.
\end{lemma}

\begin{proof}
We aim to get 
\begin{center}
\begin{tabular}{c c c c} 

 $H_1$ & $H_2$& $H_3$& $H_4$\\ [0.5ex] 
 \hline
  $U_{15}C_2$ & $U_{14}A_1$& $U_{14}A_1$& $U_{10}$\\
 
\end{tabular}.
\end{center}
The group $H_1=U_{15}B_3\cap (U_{15}B_3)^{\dot{w}_{11}}=U_{15}C_2$ intersects $R_u(U_{15}B_3)$ in $$U_{10}=\langle X_{i}\rangle_i, i\in\{10,13,16,17,19,20,21,22,23,24\} $$ and $B_3$ in a $U_5C_2$, where $$U_5=\langle X_i\rangle_i, i\in \{2,6,9,14,-1\};C_2=\langle X_{\pm 3},X_{\pm 5}, X_{\pm 8},X_{\pm 11} \rangle.$$ We find a subgroup $H_2^\star=U_{14}A_1\leq H_2$, given by
 $$U_{14}=\langle X_i, x_{3}(t)x_{16}(t),x_{8}(t)x_{20}(-t),x_{11}(t)x_{21}(-t) : i\in \{2, 6, 9, 13, 14, 17, 19, 22, 23, 24, -1\},t\in K\rangle$$ and $A_1=\langle X_{\pm 5} \rangle$. It projects as a $ U_3A_1$ onto $ C_2$, which analogously to the case $\dot{w}=\dot{w}_{14}$ leads us to the conclusion that the projection of $H_2$  on $C_2$ is contained in $ U_3A_1T_1\leq C_2$.

To prove that $U_{14}A_1\leq H_2\leq U_{14}A_1T_1 $, it then suffices to show that no element in $R_u(H_1)\setminus U_{14}A_1$ is in $H_2$. To do this we just need to show that a non-identity element of the form $$g=x_{10}(t_1)x_{16}(t_2)x_{20}(t_3)x_{21}(t_4)$$ does not fix $\dot{w}_{11}.y= -e_{1 1 2 3 2 2 1}+e_{0 1 1 2 2 1 1}$. As we have previously done, we achieve this by looking at the appropriate root strings. First note that by Table~\ref{tab:e7 to f4}, $\dot{w}_{11}.y=v_{0111}+v_{-0001}$, for two non-zero weight vectors $ v_{0111},v_{-0001}$. The $\alpha_{10}$-root strings through $ 0111$ and $-0001$ are $(0111)$ and $(-0001,0110)$. The $\alpha_{16}$-root strings through $ 0111$, $-0001$ and $0110$ are $(0111)$, $(-0001,0121)$ and $(0110)$. The $\alpha_{20}$-root strings through $ 0111$, $-0001$, $0110$ and $0121$ are all trivial except for $(-0001,1221)$. Finally, the  $\alpha_{21}$-root strings through $ 0111$, $-0001$, $0110$, $0121$ and $1221$ are trivial except for $(-0001,1231)$ and $(0110, 1342)$. This shows that no cancellation can happen between the different weight vectors appearing in $g.(\dot{w}_{11}.y)$. Hence $g.(\dot{w}_{11}.y)\neq \dot{w}_{11}.y$ and we are done. Therefore $U_{14}A_1\leq H_2 \leq U_{14}A_1T_1 $, where $T_1=\{h_{\alpha_3}(\kappa):\kappa\in K^*\}$. Since no non-trivial element in $T_1$ fixes $\dot{w}_{11}.y$, we get the required $H_2=U_{14}A_1$.

Similarly $$H_3= \langle X_i, x_{8}(t)x_{9}(t),x_{-1}(t)x_{-3}(t),x_{5}(t)x_{6}(-t):t\in K\rangle_i\langle X_{\pm 11}\rangle,$$
$$ i\in\{2,10,14,16,17,19,20,21,22,23,24\}. $$

We now find a subgroup $H_4^\star=U_{10}\leq H_2\cap H_3=H_4$ given by 
$$U_{10}=\langle X_i,x_{5}(t)x_{6}(-t),x_{11}(t)x_{21}(-t),x_{8}(t)x_{9}(t)x_{20}(-t): i\in\{2,14,17,19,22,23,24\},t\in K\rangle.$$

We are now going to describe a normal form for elements in $H_2$ and $H_3$ that is going to allow us to find $H_4$. Let
 $$U_a=\langle X_i\rangle_i ,i\in\{2, 6, 9, 13, 14, 16, 17, 19, 20,21, 22, 23, 24, -1\};$$
  $$ U_b=\langle X_i\rangle_i ,i\in\{2, 6,9,10, 14, 16, 17, 19, 20, 21, 22, 23, 24,-1\};$$
$$P_a=\langle X_3,X_8,X_{11},X_{\pm 5}\rangle ; P_b= \langle X_8,X_{-3},X_{5},X_{\pm 11}\rangle.$$

Looking at the generators of $H_2$, we note that with some rearrangements we can write any $g\in H_2$ as $ u_a p_a$, for some $u_a\in U_a$ and $p_a\in P_a$. Similarly we can write any $ g\in H_3$ as $u_bp_b$ for some $u_b\in U_b$ and $p_b\in P_b$. Let $g\in H_2\cap H_3$, and write it as both $u_ap_a$ and $u_bp_b$ for some $u_a\in U_a,p_a\in P_a,u_b\in U_b,p_b\in P_b$. Since $U_a,U_b\leq R_u(H_1)=U_{15}$ and $P_a,P_b\leq C_2$, we must have $u_a=u_b$ and $p_a=p_b$. It is easily seen that $P_a\cap P_b= U_3=\langle X_5,X_8,X_{11} \rangle$. Therefore $H_4$ is unipotent and more precisely $H_4=\langle R_u(H_2),  X_5\rangle \cap \langle R_u(H_3),  X_{21}\rangle$. This is the intersection of two $U_{15}$'s, that we call respectively $V_1$ and $V_2$, both containing $U_{10}=H_4^\star$. Since $V_2\leq B$, we just need to intersect $V_1^\star$ and $V_2$, where 
$$V_1^\star= V_1\cap B = U_{14}=\langle X_i,  x_{5}(t)x_{6}(-t), x_{3}(t)x_{16}(t),x_{8}(t)x_{20}(-t),x_{11}(t)x_{21}(-t) : t\in K\rangle_i,$$  $$i\in \{2, 6, 9, 13, 14, 17, 19, 22, 23, 24\}.$$ Considering the overlap between the generators of $H_4^\star$ and $V_1^\star$, we only need to understand under which conditions an element $$ g=x_5(t_1)x_6(t_2)x_8(t_3)x_9(t_4)x_{20}(-t_3)x_3(t_5)x_{16}(t_5)x_{13}(t_6)$$ belongs to $H_4$. If we are able to show that $t_1=-t_2$, $ t_3=t_4$ and $t_5=t_6=0$, then $g\in H_4^\star$ and $H_4=H_4^\star=U_{10}$. Since $g\in H_2$, we need to understand the action of $g$ on $y$, which is fixed by $H_3$. Let us rewrite $g$ in terms of the explicit roots, i.e. $$g=x_{1 1 0 0}(t_1)x_{0 1 1 0}(t_2)x_{1 1 1 0}(t_3)x_{0 1 2 0}(t_4)x_{1 2 2 2}(-t_3)x_{0010}(t_5)x_{0122}(t_5)x_{0121}(t_6).$$ Recalling that $y=v_{1111}+v_{0121}$, just by finding root strings, we find that $$ g'=x_{1 1 0 0}(t_1)x_{0 1 1 0}(t_2)x_{1 1 1 0}(t_3)x_{0 1 2 0}(t_4)x_{1 2 2 2}(-t_3)$$ sends $y$ to $y+v_{1221}+v_{1231}$, with $v_{1221}$ or $v_{1231}$ possibly equal to $0$. In any case the element $g'' =x_{0010}(t_5)x_{0122}(t_5)x_{0121}(t_6)$ adds a non-zero $v_{1121}$ and $v_{1232}$ to $y+v_{1221}+v_{1231}$, unless $g''=1$. Hence $g''=1$ and we can focus on $g'$. Here we note that $ x_{1 1 0 0}(t_1)x_{0 1 1 0}(t_2).y=y+v_{1221}$ and $$ x_{1 1 1 0}(t_3)x_{0 1 2 0}(t_4)x_{1 2 2 2}(-t_3). (y+v_{1221})=x_{1 1 1 0}(t_3)x_{0 1 2 0}(t_4)x_{1 2 2 2}(-t_3). (y)=y+v_{1231}.$$ Therefore both $ x_{1 1 0 0}(t_1)x_{0 1 1 0}(t_2)$ and $x_{1 1 1 0}(t_3)x_{0 1 2 0}(t_4)x_{1 2 2 2}(-t_3)$ must individually fix $y$, giving the required result. This allows us to conclude that $H_4=H_4^\star$. This implies that the $2$-space only contains grey points.

Having found the centralizer of $W_2=\langle y, \dot{w}_{11}.y\rangle$, we proceed to exhibit a full $A_1T_1$ action on the $2$-space, which lets us conclude that $G_{W_2}=U_{10}.A_1T_1$. The element $x_{16}(-1)x_{18}(1)x_{3}(1)x_1(1)$ fixes $y$ and sends $\dot{w}_{11}.y$ to $\dot{w}_{11}.y+y$. The element $x_{-16}(1)x_{-18}(-1)x_{-3}(1)x_{-1}(-1)$ fixes $\dot{w}_{11}.y$ and sends $y$ to $y+\dot{w}_{11}.y$. Finally there is a $$T_2=\{h_{\alpha_1}(\kappa)h_{\alpha_2}(\kappa)h_{\alpha_3}(\kappa)h_{\alpha_4}(\kappa),h_{\alpha_4}(\kappa):\kappa\in K^*\}$$ which stabilises $W_2$ while acting faithfully on it. 
\end{proof}
\begin{lemma}\label{f4 orbit 10}
The stabilizer of $W_2=\langle y, \dot{w}_{10}.y\rangle$ is isomorphic to $U_{12}A_1T_1$.
\end{lemma}

\begin{proof}
We want to get 
\begin{center}
\begin{tabular}{c c c c} 

 $H_1$ & $H_2$& $H_3$& $H_4$\\ [0.5ex] 
 \hline
  $U_{20}A_2$ & $U_{18}A_1$& $U_{13}A_2$& $U_{11}A_1$\\

\end{tabular}.
\end{center}

The determination of $H_1,H_2,H_3,H_4$ is achieved in a very similar way to orbit number $14$. Let us describe the different groups in more detail. The group $H_1=U_{20}A_2$ has unipotent radical generated by root subgroups numbers $  3, 4, 6, 7, 9, 10, 13, 15, 16, 17, 18, 19, 20, 21, 22, 23, 24, -1,  -5, -8,$ while $A_2$ can be taken as $\langle X_{\pm 2},X_{\pm 11},X_{\pm 14}\rangle$. The group
$H_2=U_{18}A_1$ has unipotent radical $$U_{18}=\langle X_i, x_{2}(t)x_{4}(t),x_{17}(t)x_{18}(t),x_{6}(t)x_{7}(-t),x_{14}(t)x_{15}(-t):t\in K\rangle_i ,$$ $$i\in\{3, 9, 10, 13, 16, 19, 20, 21, 22, 23, 24, -1, -5, -8\},$$ and $A_1=\langle X_{\pm 11} \rangle$. Also, $$H_3=\langle X_i\rangle A_2, i\in \{4, 7, 10, 15, 16, 17, 18, 19, 20, 21, 22, 23, 24\}.$$ Finally $$ H_4=U_{11}A_1= \langle X_i,x_{2}(t)x_{4}(t),x_{17}(t)x_{18}(t),x_{14}(t)x_{15}(-t):t\in K \rangle_i\langle X_{\pm 11}\rangle,$$
$$i\in\{ 10, 16, 19, 20, 21, 22, 23, 24\}.$$ 

We now exhibit a $U_1T_1$ on top of the centralizer $H_4=U_{11}A_1$ and argue it is the full action of $G_{W_2}$. The element $x_7(1)  x_{12}(1) x_{13}(1)  x_{17}(-1) x_{19}(-1) x_6(1) x_5(-1)$ fixes $y$ and sends $\dot{w}_{10}.y$ to $\dot{w}_{10}.y-y$. Furthermore there is a $T_1=\{h_{\alpha_1}(\kappa)h_{\alpha_3}(\kappa):\kappa\in K^*\}$ fixing $y$ and scaling $\dot{w}_{10}.y$. As for orbit number $14$, if there was an element in $G_{W_2}$ that does not fix $\langle y\rangle$, we would have a full $A_1T_1$ action. Therefore, if the action induced by $G_{W_2}$ on $W_2$ is not a $U_1T_1$, then $G_{W_2}$ contains an element fixing $\dot{w}_{10}.y$ and scaling $y$. This would be an element in $U_{14}G_2T_1 \cap (U_{14}G_2)^{\dot{w}_{11}}$. As we did for orbit $ 14$ it is not difficult to find that $U_{14}G_2T_1 \cap (U_{14}G_2)^{\dot{w}_{11}}=H_4$.
Hence $ G_{W_2}=U_{11}A_1.U_1T_1$. Since $y$ and $\dot{w}_{11}.y$ are both grey points the $U_1T_1$ action guarantees that $W_2$ is purely grey $2$-space, as claimed. 
\end{proof}

\begin{lemma}\label{f4 orbit 7}
The stabilizer of $W_2=\langle y, \dot{w}_7.y \rangle$ is isomorphic to $U_{14}A_1T_1$.
\end{lemma}
\begin{proof}
The process is very similar to what we have done for numbers $11$ and $12$ and leads us to 
\begin{center}
\begin{tabular}{c c c c} 

 $H_1$ & $H_2$& $H_3$& $H_4$\\ [0.5ex] 
 \hline
  $U_{20}A_2$ & $U_{18}A_1$& $U_{18}A_1$& $U_{14}$\\

\end{tabular}.
\end{center}

We therefore proceed with the description of the action of $G$ on $W_2=\langle y, \dot{w}_7.y \rangle$. The element $$x_4(-1) x_7(1) x_3(1) x_1(1)$$ fixes $y$ and sends $\dot{w}_7.y$ to $\dot{w}_7.y-y$. The element $x_{-4}(1) x_{-7}(-1) x_{-3}(-1) x_{-1}(1)$ fixes $\dot{w}_7.y$ and sends $y$ to $y-\dot{w}_7.y$. Finally there is a $T_2=\{h_{\alpha_1}(\kappa)h_{\alpha_2}(\kappa)h_{\alpha_3}(\kappa)h_{\alpha_4}(\kappa), h_{\alpha_2}(\kappa):\kappa\in K^*\}$, which stabilises $W_2$ while acting faithfully on it. This means that $G_{W_2}$ induces an $A_1T_1$ action on $W_2$ and we are done. This also means that $W_2$ is purely grey.
\end{proof}

We have concluded the analysis of orbits number $ 7,10,11,12,14$. We proceed with $2$-spaces that we defined with a basis consisting of a white and a grey vector, i.e. orbits number $ 2,4,6,8,9$. Here the centralizer will be the intersection of $U_{15}B_3$ and $(U_{14}G_2)^{\dot{w}}$, for the appropriate $w$. As before, we call the intersections $U_{15}B_3\cap (U_{15}B_3)^{\dot{w}}$ and $U_{15}B_3\cap (U_{14}G_2)^{\dot{w}}$ respectively $H_1$ and $H_2$.

\begin{lemma}\label{f4 orbit 9}
The stabilizer of $W_2=\langle x,\dot{w}_9.y \rangle$ is a $U_1G_2T_1$.
\end{lemma}

\begin{proof}
We find that $H_1= B_3$ is stabilised by $\dot{w}_9$, and therefore $H_2=G_2$.  By Lemma~\ref{two white}, $W_2$ contains precisely one white point, namely $\langle x \rangle$. Therefore the action induced by $G_{W_2}$ on $W_2=\langle x,\dot{w}_9.y\rangle$ is at most $U_1T_2$. In fact it is at most $U_1T_1$, since the intersection of the stabilizer of $\langle x \rangle$ with the stabilizer of $\langle \dot{w}_9.y \rangle$, i.e. $U_{15}B_3T_1\cap (U_{14}G_2T_1)^{\dot{w}_9}$ is of course just $G_2T_1$. This gives us the required $T_1$, and therefore we simply need to exhibit a $U_1$ action to conclude that $G_{W_2}=G_2.U_1T_1$. This is achieved thanks to the element $$x_{12}(-1) x_{21}(-1) x_{13}(1) x_6(-1) x_{11}(-1) x_{14}(-1) x_5(1),$$ which fixes $x$ and sends $y$ to $ y-x$.
\end{proof}

\begin{lemma}\label{f4 orbit 8}
The stabilizer of $W_2=\langle x,\dot{w}_8.y \rangle$ is isomorphic to $U_{12}A_1T_2$.
\end{lemma}

\begin{proof}
We find that $ H_1=U_{13}A_2$ and $ H_2=U_{11}A_1$, in a very similar fashion to orbit number $14$. The $2$-space has therefore a single white point. In fact also $ U_{15}B_3T_1 \cap (U_{14}G_2)^{\dot{w}_8}$ is equal to $H_2$. This means that we have at most a $U_1T_1$ action on the $2$-space. The $T_1$ is given by $ \{h_{\alpha_2}(\kappa):\kappa\in K^*\}$. The $U_1$ is given by $ x_{16}(1) x_{23}(-1)$ together with the $T_1$.
\end{proof}

\begin{lemma}\label{f4 orbit 6}
The stabilizer of $W_2=\langle x,\dot{w}_6.y \rangle$ is isomorphic to $U_{15}A_1T_2$.
\end{lemma}

\begin{proof}
Following the steps of orbit number $11$ leads to $H_1=U_{15}C_2T_2$ and $H_2=U_{14}A_1$. Again this allows to conclude that $W_2=\langle x, \dot{w}_6.x\rangle$ has a single white point. This time we have a full $U_1T_2$ action induced by the stabilizer. This is given by $$h_{\alpha_1}(\kappa)h_{\alpha_2}(\kappa)h_{\alpha_4}(\kappa), h_{\alpha_1}(\kappa)h_{\alpha_3}(\kappa):\kappa\in K^*, x_{16}(-1) x_{18}(1) x_{22}(1) x_{23}(-1) x_{24}(1) .$$ Therefore $G_{W_2}=U_{14}A_1.(U_1T_2) $.
\end{proof}

The next two lemmas follow similarly, and we simply exhibit the induced faithful action on the $2$-spaces.
\begin{lemma}\label{f4 orbit 2}
The stabilizer of $W_2=\langle x,\dot{w}_2.y \rangle$ is isomorphic to $U_{15}G_2T_1$.
\end{lemma}

\begin{proof}
The action of $H_{W_2}$ on $W_2$ is a $U_1T_1$, induced by $ \langle h_{\alpha_2}(\kappa),X_{12}:\kappa\in K^*\rangle$.
\end{proof}

\begin{lemma}\label{f4 orbit 4}
The stabilizer of $W_2=\langle x,\dot{w}_4.y \rangle$ is isomorphic to $U_{19}A_1T_2$.
\end{lemma}

\begin{proof}
The action of $H_{W_2}$ on $W_2$ is a $U_1T_2$ induced by 
$$ \langle h_{\alpha_1}(\kappa)h_{\alpha_2}(\kappa)h_{\alpha_4}(\kappa),h_{\alpha_2}(\kappa),x_{16}(-1) x_{18}(1) x_{22}(1) x_{23}(-1) x_{24}(1):\kappa\in K^*\rangle.$$ 
\end{proof}

It remains to deal with the $2$-spaces numbered $1$, $3$ and $5$ in Table~\ref{tab:f4 orbits list}. Let us start with number $1$.
\begin{lemma}\label{f4 orbit 1}
The stabilizer of $W_2=\langle x,\dot{w}_1.x \rangle$ is isomorphic to $U_{20}A_2A_1T_2$.
\end{lemma}

\begin{proof}
Here the point-wise stabilizer is given by $P \cap P^{\dot{w}_1}=U_{20}A_2T_2$. We also have an $A_1=\langle X_{\pm 4} \rangle$ acting faithfully on the two space, proving that $G_{W_2}=U_{20}A_2A_1T_1$.
\end{proof}

\begin{lemma}\label{f4 orbit 3}
The stabilizer of $W_2=\langle x,\dot{w}_3.x \rangle$ is isomorphic to $U_{15}C_2A_1T_1$.
\end{lemma}

\begin{proof}
We have $P\cap P^{\dot{w}_3}= U_{15}C_2T_2$ and we have an $A_1T_1$ action induced by the $T_2$ together with $$x_{18}(-1)x_{22}(-1)x_{15}(1)x_3(1)x_9(-1)x_{11}(1)x_{14}(1)x_8(-1)x_1(-1)x_{40}(1)x_{30}(-1)$$ and 
$$x_{16}(-1) x_{18}(-1) x_{20}(1) x_{21}(-1) x_{22}(-1) x_{23}(-1) .$$ This shows that indeed $ G_{W_2}=U_{15}C_2A_1T_1$, for $W_2=\langle \dot{w}_3.x,x\rangle$.
\end{proof}

\begin{lemma}\label{f4 orbit 5}
The stabilizer of $W_2=\langle x,\dot{w}_5.x \rangle$ is isomorphic to $U_{13}A_2T_2.2$.
\end{lemma}

\begin{proof}
First note that $P\cap P^{\dot{w}_5}=U_{13}A_2T_2$. The $T_2$ action shows that if $a,b\neq 0$, all vectors of the form $ax+b\dot{w}_5.x$ are in the same $F_4$-orbit, i.e. they are all either white or grey. The element $n_3 n_2 n_3 n_4 $, sends $y$ to $x-\dot{w}_5.x$. Therefore  $x-\dot{w}_5.x$ is a grey vector and if $a,b\neq 0$, so are all the vectors $ax+b\dot{w}_5.x$. This means that the $2$-space $W_2=\langle x, \dot{w}_5.x\rangle $ contains precisely $2$ white points, i.e. $ \langle x \rangle$ and $\dot{w}_5.x$. 

The element $n_4 n_3 n_2 n_3 n_4 $ swaps $\langle x \rangle$ and $\langle \dot{w}_5.x\rangle $ Since there are precisely $2$ white points, this gives the full action of the stabilizer on $W_2$. Hence $G_{W_2}= U_{13}A_2T_2.2$.
\end{proof}

By putting together Lemmas \ref{f4 orbit 14} to \ref{f4 orbit 5} we finally have a complete proof of Proposition~\ref{main proposition f4 k=2}.

We can now use a counting argument to show that the list of $G$-orbits on totally singular $2$-spaces in Table~\ref{tab:f4 orbits} is a complete list. If $q$ is a power of $p$, we denote by $(q-1)$ a torus of size $q-1$, by $(q+1)$ an torus of size $ q+1$, by $q^i$ a unipotent group of size $q^i$, all in the finite group $F_4(q)$.
\begin{corollary}
The group $F_4$ has $15$ orbits on totally singular $2$-spaces of $ V_{F_4}(\lambda_4)$. 
\end{corollary}

\begin{proof}
Let $q=p^e=3^e$ for an arbitrary positive integer $e$. Let $\sigma_q$ be the standard Frobenius morphism sending $x_i(t)$ to $x_i(t^q)$ and acting in a compatible way on $V$. Then the induced action of $\sigma$ on $P_2^{TS}(V)$ stabilises the orbits in Table~\ref{tab:f4 orbits}, since for each orbit we have a representative given in terms of $e_{\beta_i}$'s with coefficients in $\mathbb{F}_3$.  
The only orbits in Table~\ref{tab:f4 orbits} with a disconnected stabilizer are numbers $5$, $13$ and $15$. 

Let $\Gamma_5$ be the $F_4$-orbit with representative $W_2=\langle x, \dot{w}_5.x\rangle$ and stabilizer $G_{W_2}=U_{13}A_2T_2.2$. The element $n_4 n_3 n_2 n_3 n_4 $, which swaps $\langle x \rangle$ and $\langle \dot{w}_5.x \rangle$, centralises an $ A_2T_1\leq A_2T_2$ and inverts $T_1=\{h_{\alpha_2}(\kappa):\kappa\in K^*\}$. Therefore by Lang-Steinberg, the fixed points of $\Gamma_5$ under $\sigma_q$, split into two $F_4(q)$ orbits with stabilizers of size $q^{13}|SL(3,q)|(q-1)^2.2$ and $q^{13}|SL(3,q)|(q-1)(q+1).2$.

In both orbits $13$ and $15$, the component group of the stabilizer centralizes the $2$-space and induces a graph automorphism on $A_2$. In the case of orbit $13$, when passing to finite fields, this produces two orbits with stabilizers of size $ |SL(3,q)||GL(2,q)|.2$ and $ |SU(3,q)||GL_(2,q)|.2$. Finally, in the case of orbit $15$, when passing to finite fields, we get two orbits with stabilizers of size $ |SL(3,q)|q(q-1).2$ and $ |SU(3,q)|q(q-1).2$. The sizes of the stabilizers for the orbits in the finite case are therefore as in Table~\ref{tab:f4 orbits finite field}, where we write $W_2(q)$ to denote the $2$-spaces representatives in $\Gamma_{\sigma_q}$, for the appropriate orbit $\Gamma$.
\begin{center}
\begin{xltabular}[h]{\textwidth}{l l l }
\caption{$F_4(q)$-orbits on totally singular $2$-spaces in $V(q)$} \label{tab:f4 orbits finite field} \\
\hline \multicolumn{1}{c}{Orbit number over $K$ } & \multicolumn{1}{c}{$(F_4)_{W_2}$ } & \multicolumn{1}{c}{$|F_4(q)_{W_2(q)}|$}\\
\hline 
\endhead
1& $ U_{20}A_2A_1T_1$ & $q^{20}|SL(3,q)SL(2,q)|(q-1)$\\*

  2 &$U_{15}G_2T_1$   &$q^{15}|G_2(q)|(q-1)$  \\*

  3& $U_{15}C_2A_1T_1$ & $q^{15}|Sp(4,q)SL(2,q)|(q-1)$ \\*
 
    4&$U_{19}A_1T_2$  & $q^{19}|SL(2,q)|(q-1)^2$\\

   5& $U_{13}A_2T_2.2$  &  $q^{13}|SL(3,q)|(q-1)^2.2$\\
    &   &  $q^{13}|SL(3,q)|(q-1)(q+1).2$  \\

   6&$U_{15}A_1T_2$  &   $q^{15}|SL(2,q)|(q-1)^2$ \\

    7&$U_{14}A_1T_1$  &  $q^{14}|SL(2,q)|(q-1)$\\

    8&$U_{12}A_1T_2$  &   $q^{12}|SL(2,q)|(q-1)^2$\\

   9&$U_1G_2T_1$  &   $q|G_2(q)|(q-1)$ \\

   10&$U_{12}A_1T_1$  & $q^{12}|SL(2,q)|(q-1)$\\

  11&$U_{10}A_1T_1$  &   $q^{10}|SL(2,q)|(q-1)$\\

  12&$U_8A_1T_1$  &   $q^8|SL(2,q)|(q-1)$\\

   13&$A_2A_1.2$  &   $|SL(3,q)SL(2,q)|.2$\\
      &  &   $|SU(3,q)SL(2,q)|.2$\\

   14&$U_6A_1T_1$  &   $q^6|SL(2,q)|(q-1)$\\

   15&$U_1A_2.2$  &  $q|SL(3,q)|.2$\\
     &  &  $q|SU(3,q)|.2$  \\
\hline
 
\end{xltabular}
\end{center}

We find the sizes of the orbits in the finite case by simply computing the index of each stabilizer. Adding up the sizes of the orbits gives the number of totally singular $2$-spaces in an orthogonal vector space of dimension $25$. Therefore the given orbits form a complete list of orbits for the $F_4$ action on totally singular $2$-spaces in $V_{F_4}(\lambda_4)$.
\end{proof}

This completes the proof of Proposition~\ref{F4 main prop}.

\subsection{$H$ of type $B_4$ and $V=V_{H}(\lambda_4)$}\label{b4 section}
In this section we prove the following proposition:

\begin{proposition}\label{intial b4 k=2 prop}
Let $H=B_4$ and $V=V_{B_4}(\lambda_4)$, a $16$-dimensional orthogonal module.
Then $H$ has $7+\delta_{p,2}$ orbits on $P_2^{TS}(V)$. Representatives and stabilizers can be found in Table~\ref{tab:b4 reps}.
\end{proposition}

The strategy consists of first finding the $D_5$-orbits on all $2$-spaces and then using Lemma~\ref{orbit correspondence Dn Bn} to descend to $B_4$.

We now give an explicit construction of the spin module $V$, and refer the reader to \cite{CC} for a more complete treatment of Clifford algebras, spin groups and representations. 

Let $\{e_1,\dots,e_5,e_{6},\dots, e_{10}\}=\{e_1,\dots,e_5,f_{1},\dots, f_{5}\}$ be a standard basis for the $K$-vector space $V_{10}$ with quadratic form $Q$ and bilinear form $(\cdot,\cdot)$, such that $\{e_i,e_{5+i}\}=\{ e_i,f_i\}$ are hyperbolic pairs for $i\leq 5$. 
Let $L,M$ be the totally singular subspaces $\langle e_1,\dots,e_5\rangle $ and $\langle f_1,\dots,f_5\rangle$ respectively.

We denote by $C$ the Clifford algebra of $(V_{10},Q)$. This is an associative algebra over $K$ generated by $V_{10}$, in which $v^2=Q(v)$ for every $v\in V_{10}$. It has the structure of a graded module over $K$. 
Let $\phi':C\rightarrow C $, sending $x\rightarrow x'$, be the involution of $C$ keeping every element of $V_{10}$ invariant, i.e. the anti-automorphism sending a product $\prod_{i=1}^5 v_i\in C$ to $ \prod_{i=1}^5 v_{5-i+1}$. We denote by $C^\pm$ the sums of homogeneous submodules of $C$ of even and odd degrees respectively. Then $C=C^+\oplus C^-$. In particular, $C^+$ is a subalgebra of $C$ invariant under $\phi'$. 

The \textit{Clifford group} $G^*=\{s\in C|s$ is invertible in $C$ and $sV_{10}s^{-1}=V_{10}\}$. The even Clifford group is $(G^*)^+=G^*\cap C^+$. The \textit{spin group} $Spin_{10}$ is $\{s\in (G^*)^+|ss'=1\}$. 

The \textit{vector representation} of the Clifford group $G^*$ is given by $\Theta: G^*\rightarrow Aut(V_{10},Q)$, such that $\Theta(s)\cdot v=svs^{-1}$. The restriction of $\Theta$ to $Spin_{10}$ is the natural representation of $Spin_{10}$.  

Put $e_L=e_1e_2e_3e_4e_5$ and $e_M=e_{6}e_{7}e_8e_9e_{10}$. We denote by $C_W$ the subalgebra of $C$ generated by the elements of a subspace $W\subset V_{10}$. Then $Ce_M$ is a minimal left ideal in $C$, and the correspondence $x\rightarrow xe_M$ generates an isomorphism $C_L\rightarrow Ce_M$ of vector spaces. So for any $s\in C,x\in C_L$ there exists a unique element $y\in C_L$ for which $sxe_M=ye_M$. Setting $\rho(s)\cdot x = s\cdot x=y$ gives us the spinor representation $\rho$ of the algebra $C$ in $C_L$. Let $V=C_L\cap C^+$. Then restricting $\rho$ to $Spin_{10}$, we get the half-spin representation of $G:=Spin_{10}$ in $V$.

An element of $V$ is called a \textit{spinor}. The restriction to $B_{4}$ is the spin representation for $B_{4}$.
By \cite[Prop.~5.4.9]{KL} the module $V_{D_5}(\lambda_5)$ is not self dual, while the restriction to $B_4$ is an orthogonal module.
We first aim to classify the $G$-orbits on $2$-spaces of $V$.
Let $T$ be the maximal torus of $G$ acting diagonally on the standard basis of $V_{10}$. 

The embedding of $G$ in the Clifford algebra gives us root subgroups $X_{i,j}:=\{ 1+ \lambda e_ie_j:\lambda\in K \}$ for $|i-j|\neq 5$. Let $u_1,u_2\in V_{10}$. An element $1+u_1u_2\in G$, in the action on $V_{10}$, sends a vector $v$ to $v+(v,u_2)u_1-(v,u_1)u_2$. We use $x_{i,j}(\lambda)$ to denote the element $1+\lambda e_ie_j\in X_{i,j}$.

Let us recall the orbit structure of $G$ on $1$-spaces of $V$. 

\begin{proposition}\cite[Prop.~2]{igusa}\cite[Lemma~2.11]{finite}\label{orbits on 1-spaces}
There are two $G$-orbits on spinors of $V$. A set of representatives is given by $1$ and $1+e_1e_2e_3e_4$, with stabilizers $P_5(G)'$ and $U_8B_3$ respectively. More precisely $G_{\langle 1 \rangle}=G_{\langle f_1,f_2,f_3,f_4,f_5 \rangle}$ and $G_{1+e_1e_2e_3e_4}=U_8B_3$ where $U_8=\langle X_{i,10}\rangle_i$ for $i\neq 5,10$; and 
\begin{gather*}
    B_3=\langle X_{i,j+5}, x_{6,9}(\lambda)x_{2,3}(\lambda),x_{7,8}(\lambda)x_{1,2}(\lambda),x_{7,9}(-\lambda)x_{1,3}(\lambda),\\x_{6,7}(\lambda)x_{1,4}(\lambda),x_{6,8}(-\lambda)x_{2,4}(\lambda), x_{6,7}(\lambda)x_{3,4}(\lambda):\lambda\in K\rangle_{i,j},
\end{gather*}
for $i, j\leq 4,i\neq j$.
\end{proposition}

We call the points (vectors) with stabilizer $P_5$ ($P_5'$) white points (vectors), and use grey for the other orbit.
An arbitrary spinor in $V$ can be written as $x=x^{(0)}+x^{(2)}+x^{(4)}$, where $x^{(i)}$ is a homogeneous component of degree $i$. We denote by $V^{(i)}$ the subspace of homogeneous spinors of degree $i$.

In order to determine if a non-zero spinor $x$ is in the same orbit as $1$, it is sufficient to perform the following algorithm, as described in the proof of \cite[Lemma~1]{igusa}.

\begin{itemize}
\item Make $x^{(0)}\neq 0$ by acting on $x$ with elements of the form $1+f_if_j\in G$. More precisely if $x^{(0)}\neq 0$ we are done; if  $x^{(0)}= 0$ and $0\neq x^{(2)}=e_ie_j+\dots$, we act on $x$ with $1+f_if_j$; if $x^{(0)}= x^{(2)}=0$ and $0\neq x^{(4)}=e_ie_je_ke_l+\dots$ we first make $x^{(2)}\neq 0$ by hitting $x$ with $1+f_if_j$ and then conclude as in the previous case.
\item Act on $x$ with $\prod_{i,j} 1-\alpha_{i,j}e_ie_j$, where $x^{(2)}=\sum_{i,j} \alpha_{i,j}e_ie_j$.
\item Then $x$ is $G$-equivalent to $1$ if and only if what we get after the previous step is of degree $0$.
\end{itemize}

Let us consider an example. Let $x=e_1e_2e_3e_4$. Then $(1+f_1f_2).x=x-e_3e_4$ and $(1+f_3f_4).(x-e_3e_4)=1-e_3e_4-e_1e_2+e_1e_2e_3e_4$, completing the first step. Finally $(1+e_1e_2)(1+e_3e_4).(1-e_3e_4-e_1e_2+e_1e_2e_3e_4)=1$, showing that $x$ is in the same $G$-orbit of $1$.

\begin{lemma}\label{white}
Let $W_2=\langle v_1,v_2 \rangle $ be a $2$-space of $V$ that contains at least one white point. Then $W_2$ is $G$-equivalent to one of the following $2$-spaces:

\begin{enumerate}
\item $\langle 1,e_1e_2\rangle $,
\item $\langle 1,e_1e_2+e_3e_4 \rangle$,
\item $\langle 1,e_1e_2e_3e_4\rangle $,
\item $\langle 1,e_1e_5+e_1e_2e_3e_4\rangle. $
\end{enumerate}
\end{lemma}

\begin{proof}
We can of course assume that $W_2=\langle 1,x\rangle$. Write $x=x^{(0)}+x^{(2)}+x^{(4)}$. Assume that $x^{(4)}=0$. Our $2$-space $W_2$ is then of the form $\langle 1,\sum_{i,j} \alpha_{i,j}e_ie_j\rangle$. Without loss of generality assume that $\alpha_{1,2}=1$. The root element $1+\lambda f_1e_i\in P_5(D_5)$ sends $e_1e_2$ to $e_1e_2+\lambda e_2e_i$; therefore in terms of $D_5$-equivalence we can assume that $\alpha_{2,3}=\alpha_{2,4}=\alpha_{2,5}=0$ and similarly $\alpha_{1,3}=\alpha_{1,4}=\alpha_{1,5}=0$. Either $x=e_1e_2$ and we are done, or without loss of generality $x=e_1e_2+e_3e_4+\alpha_{3,5}e_3e_5+\alpha_{4,5}e_4e_5$. Again we can assume that $\alpha_{3,5}=\alpha_{4,5}=0$ and by acting with $T$ we get $W_2=\langle 1, e_1e_2+e_3e_4 \rangle$.

Now assume that $x^{(4)}\neq 0$. By the proof of \cite[Lemma~1]{igusa}, the parabolic $P_5(D_5)$ is transitive on $V^{(4)}$. Also, $P_5(D_5).V^{(2)}\subset V^{(2)}+V^{(0)}$, so we can assume that $x=x^{(2)}+e_1e_2e_3e_4$. Acting with $1+\alpha f_if_j$ for $i,j\leq 5$ and then with $1+\alpha e_if_5$ allows us to reduce to the case where either $x^{(2)}=0$ or $x^{(2)}=e_1e_5$.
\end{proof}

\begin{lemma}\label{black}
Let $W=\langle v_1,v_2 \rangle $ be a $2$-space in $V$ such that all non-zero $v\in W$ are $D_5$-equivalent to $1+e_1e_2e_3e_4$. Then $W$ is $D_5$-equivalent to one of the following $2$-spaces:

\begin{enumerate}
\item $\langle 1+e_1e_2e_3e_4,e_1e_2+e_2e_3e_4e_5\rangle $,
\item $\langle 1+e_1e_2e_3e_4,e_1e_5+e_2e_3e_4e_5 \rangle$.
\end{enumerate}
\end{lemma}
\begin{proof}
We can assume that $W=\langle 1+e_1e_2e_3e_4,x\rangle$. Assume that $x^{(4)}=0$, so that $$x=\beta +\sum_{i,j} \alpha_{i,j}e_ie_j.$$ Also assume that $\alpha_{1,5}=\alpha_{2,5}=\alpha_{3,5}=\alpha_{4,5}=0$. Then without loss of generality $\alpha_{1,2}=1$. Using $1+\lambda f_1e_{2,3,4}$ and $1+\lambda f_2e_{1,3,4}$, which are elements of $U_8B_3$ fixing $1+e_1e_2e_3e_4$, we can assume that $\alpha_{i,j}=0$ whenever $i$ or $j$ is either $1$ or $2$ (apart from $\alpha_{1,2}$). So $x=\beta +e_1e_2+\alpha_{3,4}e_3e_4$, with $\alpha_{3,4}\neq 0$, since otherwise $x$ is $D_5$-equivalent to $1$. Acting with the torus we can assume that $\alpha_{3,4}=1$. Now consider the vector $v=1+e_1e_2e_3e_4+\delta x \in W$. Then $(1-\frac{\delta}{1+\delta\beta}(e_1e_2)).v=1+\delta\beta+\delta e_3e_4+(1-\frac{\delta^2}{1+\delta\beta})e_1e_2e_3e_4$. Setting $\delta$  such that $\delta^2=1+\delta\beta$ gives $v$ in the same $D_5$ orbit as $1$.

We can therefore assume that $\alpha_{1,5}=1$, while still dealing with the case $x^{(4)}=0$. Using $1+\lambda f_1e_i$ and $1+\lambda f_5e_i$ we reduce to the case $\alpha_{2,5}=\alpha_{3,5}=\alpha_{4,5}=0$ and $\alpha_{1,2}=\alpha_{1,3}=\alpha_{1,5}=0$. We have therefore reduced to $x$ of the form $x=x^{(0)}+e_1e_5+\alpha_{2,3}e_2e_3+\alpha_{2,4}e_2e_4+\alpha_{3,4}e_2e_4.$ If all the coefficients $ \alpha_{2,3},\alpha_{2,4},\alpha_{3,4}$ are $0$ then $x$ is $D_5$-equivalent to $1$. Therefore one of $ \alpha_{2,3},\alpha_{2,4},\alpha_{3,4}$ is non-zero and we can reduce to the case $x=x^{(0)}+e_1e_5+e_2e_3$, using $1+\lambda e_ie_j$ with $2\leq i,j\leq 4$ and the maximal torus $T$.
The element $1+\lambda f_1f_5$ takes us to the case $x=1+e_1e_5+e_2e_3$. Now $(1- f_1f_4)(1- e_2e_3).x=1+e_1e_5-e_1e_2e_3e_5$. Therefore $W=\langle 1+e_1e_2e_3e_4,e_1e_5-e_1e_2e_3e_5-e_1e_2e_3e_4 \rangle$. Now $(1-e_4f_5).(e_1e_5-e_1e_2e_3e_5-e_1e_2e_3e_4)=e_1e_5-e_1e_4-e_1e_2e_3e_5$ and $(1-f_2f_3).(e_1e_5-e_1e_4-e_1e_2e_3e_5)=-e_1e_4-e_1e_2e_3e_5$, which shows that $W$ is $D_5$-equivalent to the first case in the conclusion of this lemma.

Finally let us consider the case where $x^{(4)}\neq 0$. Say $x^{(4)}=\alpha_1e_2e_3e_4e_5+\alpha_2e_1e_3e_4e_5+\alpha_3e_1e_2e_4e_5+\alpha_4e_1e_2e_3e_5+\alpha_5e_1e_2e_3e_4$. We can of course assume that $\alpha_5=0$. Without loss of generality $\alpha_1=1$. Thanks to $1+\lambda e_if_j$, with $j\leq 4$, we reduce to the case $\alpha_2=\alpha_3=\alpha_4=0$. Now using $1+\alpha f_if_5$ we reduce to the case when $x^{(2)}$ has $\alpha_{2,3}=\alpha_{2,4}=\alpha_{3,4}=0$. Suppose $\alpha_{1,5}\neq 0$. Then we can assume that $\alpha_{1,2}=\alpha_{1,3}=\alpha_{1,4}=\alpha_{2,5}=\alpha_{3,5}=\alpha_{(4,5)}=0$, leaving $x^{(2)}=e_1e_5$. This gives $W$ as in the second case of the lemma.

Now suppose $\alpha_{1,5}=0$. If $\alpha_{1,2}=\alpha_{1,3}=\alpha_{1,4}=0$ then we are in the case $x=x^{(0)}+e_2e_5+e_2e_3e_4e_5$, which we already encountered. Otherwise suppose $\alpha_{1,2}=1$. We can reduce to $\alpha_{1,3}=\alpha_{1,4}=0$. If $\alpha_{2,5}=\alpha_{3,5}=\alpha_{4,5}=0$ we are done, otherwise if either $\alpha_{3,5}\neq 0$ or $\alpha_{4,5}\neq 0$ we can reduce to $x^{(2)}=e_1e_2+e_3e_5$. Here $(1+\lambda f_3f_5).x=x^{(0)}-\lambda+e_1e_2+e_3e_5+e_2e_4+e_2e_3e_4e_5$ and $(1-f_1e_4)(1+\lambda f_3f_5).x=x^{(0)}-\lambda+e_1e_2+e_3e_5+e_2e_3e_4e_5$, so that taking $\lambda=x^{(0)}$ we can assume $x^{(0)}=0$. Now we can reduce to $\alpha_{3,5}=0$ using $(1-\lambda f_2f_4)(1+\lambda e_1e_3)$, and we are done.
The last remaining case is $\alpha_{2,5}\neq 0$ and $\alpha_{3,5}=\alpha_{4,5}=0$. Adding to $x$ a multiple of $1+e_1e_2e_3e_4$ and acting with $(1+\lambda f_5e_1)$ we reduce to $x^{(0)}=0$, so that $x=e_1e_2+e_2e_5+e_2e_3e_4e_5$. Finally $(1+\lambda f_3f_4)(1+\lambda e_1e_2)$ is what allows us to get rid of $e_2e_5$. Here $W$ is as in the first case of the lemma.
\end{proof}

We therefore have the following proposition.

\begin{proposition}
There are six $G$-orbits on $P_2(V)$. A set of representatives and corresponding numbers of white points are as in Table~\ref{tab:d5 reps}.

\begin{center}
 \begin{xltabular}{\textwidth}{l l l} \caption{$D_5$-orbit representatives}
\label{tab:d5 reps}\\
 \hline
 Orbit number & Orbit representative & Number of white points\\ \hline
 \endhead

 1& $\langle 1+e_1e_2e_3e_4,e_1e_5+e_2e_3e_4e_5 \rangle$ & $0$ \\*

  2 &$\langle 1+e_1e_2e_3e_4,e_1e_2+e_2e_3e_4e_5\rangle $ & $0$\\*

  3& $\langle 1,e_1e_5+e_1e_2e_3e_4\rangle$  & $1$\\*

    4&$\langle 1,e_1e_2e_3e_4\rangle$ & $2$\\*

   5& $\langle 1,e_1e_2+e_3e_4 \rangle $ & $ 1$\\*

   6&$\langle 1,e_1e_2\rangle $ & all\\
 \hline
\end{xltabular}

\end{center} 
\end{proposition}

\begin{proof}
This follows from Lemma~\ref{white} and Lemma~\ref{black}.
\end{proof}

We now seek to determine the stabilizers of the given representatives. Let us start with the dense orbit.

\begin{lemma}\label{dense}
The stabilizer in $G$ of $\langle 1+e_1e_2e_3e_4,e_1e_5+e_2e_3e_4e_5 \rangle$ is $A_1G_2$. 
\end{lemma}

\begin{proof}
Pick a subgroup $A_1B_3\leq G$ such that $V\downarrow A_1B_3 =\lambda_1\otimes \lambda_3$ (see \cite[$2.7$]{liebeckreductive}). If we choose a subgroup $G_2$ of the factor $B_3$, then $G_2$ fixes a $1$-space in $V_{B_3}(\lambda_3)$, and hence the subgroup $A_1G_2$ of $G$ fixes a $2$-space in $V$. In \cite[Lemma~3.7]{finite} it is shown that in fact $A_1G_2$ is the full stabilizer of this $2$-space. More precisely, if $A_1B_3$ stabilises $ \langle e_1,f_1,e_5-f_5 \rangle \oplus \langle e_2,e_3,e_4,f_2,f_3,f_4,e_5+f_5 \rangle$, and $G_2$ is taken to be the subgroup of $B_3$ fixing $ v_{-}+v_{+}$, where $v_{-}$ and $v_{+}$ are respectively a lowest and highest weight vector in $ V_{B_3}(\lambda_3)$, we can explicitly determine the $2$-space fixed by this specific $A_1G_2$. We just need to determine vectors in $V$ of weight $ \pm \lambda_1\otimes \pm \lambda_3$, giving us a two space $ \langle e_1e_2e_3e_4+e_1e_5, 1+e_2e_3e_4e_5 \rangle$. This is $G$-equivalent to the $2$-space $\langle 1+e_1e_2e_3e_4,e_1e_5+e_2e_3e_4e_5 \rangle$, as it is mapped to it by the element $ (1-e_5f_1)(1+f_5e_1)$. This concludes the proof.
\end{proof}

For the remaining orbits we need to do more work to find the stabilizers. In general we will first determine the centralizers of the $2$-spaces, and then the full stabilizer. 

\begin{lemma}
The stabilizer of $W_2=\langle 1+e_1e_2e_3e_4,e_1e_2+e_2e_3e_4e_5\rangle $ is $U_{11}A_1A_1T_1$. 
\end{lemma}

\begin{proof}
We start by listing the root subgroups of the pointwise stabilizer. Recall that $U_8B_3$ is the stabilizer of $1+e_1e_2e_3e_4$, where $U_8=\langle X_{i,10}\rangle_i$ for $i\neq 5,10$; and 
\begin{gather*}
    B_3=\langle X_{i,j+5}, x_{6,9}(t)x_{2,3}(t),x_{7,8}(t)x_{1,2}(t),x_{7,9}(-t)x_{1,3}(t),\\x_{6,7}(t)x_{1,4}(t),x_{6,8}(-t)x_{2,4}(t),x_{6,7}(t)x_{3,4}(t)\rangle_{i,j},
\end{gather*}

for $i, j\leq 4,i\neq j$.
The stabilizer of $ e_1e_2+e_2e_3e_4e_5$ is another $U_8B_3$, where $U_8=\langle X_{i,2}\rangle_i$ for $i\neq 2,7$; and 
\begin{gather*}
    B_3=\langle X_{i,j+5}, x_{3,4}(t)x_{1,10}(t),x_{3,5}(t)x_{9,1}(t),x_{4,5}(t)x_{8,1}(t),\\x_{8,9}(t)x_{6,5}(t),x_{8,10}(t)x_{6,4}(t),x_{9,10}(-t)x_{6,3}(t)\rangle_{i,j},
\end{gather*}
where $i, j\in \{1,3,4,5\},i\neq j$.

We then note that $ G_{1+e_1e_2e_3e_4}\leq G_{f_5}$ and $G_{e_1e_2+e_2e_3e_4e_5}\leq G_{e_2}$, since the listed generators have no $e_5$ and $ f_2$ contributions respectively. There is a single $G$-orbit on totally singular $2$-spaces in $V_{10}$, with stabilizer $P_2(D_5)=U_{13}A_1A_3T_1$. The centralizer of $\langle f_5,e_2\rangle$ is then a $U_{13}A_3$, with $$U_{13}=\langle X_{i,j}: |i-j|\neq 5; i \in \{2,10\}; j\neq 5,7\rangle ;\quad A_3=\langle X_{i,j}: i,j\neq 2,5,7,10;  |i-j|\neq 5 \rangle .$$ We first intersect the two $U_{8}B_3$'s with $U_{13}A_3$. Arguments like in Lemma~\ref{f4 orbit 14} show that the two intersections $U_{8}B_3^{(1)}\cap U_{13}A_3$ and $U_{8}B_3^{(2)}\cap U_{13}A_3$ have the structure of two $U_{13}A_2$'s, which on the other hand intersect as a $U_{11}A_1$, with \begin{gather*}
    U_{11}=\langle X_{2,10},X_{3,10},X_{4,10},X_{6,10},X_{2,6},X_{2,8},X_{2,9}, x_{6,9}(t)x_{2,3}(t),\\ x_{6,8}(-t)x_{2,4}(t),x_{9,10}(-t)x_{6,3}(t),x_{8,10}(t)x_{6,4}(t) \rangle,
\end{gather*}
and $A_1=\langle X_{4,8},X_{3,9} \rangle$. 

We now exhibit a faithful $A_1T_1$ action on the $2$-space. The element $ x_{2,5}(t)x_{8,9}(t)x_{6,5}(t)\in G_{e_1e_2+e_2e_3e_4e_5}$ adds $t(e_1e_2+e_2e_3e_4e_5)$ to $ 1+e_1e_2e_3e_4$, inducing a $U_1$ action. The element $ x_{7,10}(t)x_{6,7}(t)x_{3,4}(t)\in G_{1+e_1e_2e_3e_4}$ induces the opposite $U_1$ action, giving an $A_1$-action on the $2$-space. Finally, there is a $3$-dimensional torus in the maximal torus $T$, stabilising the given $2$-space. This means that $G_{W_2}$ induces a faithful $A_1T_1$ action on $W_2$. This completes our proof.

\end{proof}

With the last two propositions we have determined the stabilizers for the orbits number $1$ and $2$, which are the only orbits containing only spinors $G$-equivalent to $1+e_1e_2e_3e_4$. In orbits $3$ and $5$ there is a unique point $G$-equivalent to $\langle 1 \rangle$, so the stabilizers are contained in $P_5(D_5)$.

\begin{lemma}
The stabilizers of $\langle  1,e_1e_5+e_1e_2e_3e_4 \rangle$ and $\langle  1,e_1e_2+e_3e_4 \rangle$ are $ U_{11}A_2T_2$ and $ U_{14}C_2T_2$ respectively. 
\end{lemma}

\begin{proof}
The centralizers of both $2$-spaces are obtained by intersecting a $P_5(D_5)'$ with a $ U_8B_3$. Let us start with $W_2=\langle  1,e_1e_5+e_1e_2e_3e_4 \rangle$. Here $ G_{e_1e_5+e_1e_2e_3e_4}\leq G_{ e_1}$. The intersection of the parabolics $G_{\langle f_1,f_2,f_3,f_4,f_5\rangle}$ and $G_{\langle e_1 \rangle}$ is a parabolic $U_{10}A_3T_2$. The centralizer of $W_2$ is therefore a $U_{10}A_3$. We find that the intersection of $U_8B_3=G_{e_1e_5+e_1e_2e_3e_4}$ with $U_{10}A_3$ is a $U_{10}A_2$. Finally, we exhibit an induced $U_1T_2$ action on $W_2$. This is given by $T_{W_2}$ and $U_1=X_{6,10}$. This gives us the required stabilizer $U_{11}A_2T_2$.

Let $W_2=\langle  1,e_1e_2+e_3e_4 \rangle$. Here $G_{1,e_1e_2+e_3e_4 }\leq G_{f_5}$, and the intersection of $P_5=G_{\langle f_1,f_2,f_3,f_4,f_5\rangle}$ with $ P_1=G_{\langle f_5 \rangle}$ is again a $U_{10}A_3T_2$. The intersection of $U_{10}A_3$ with $U_8B_3=G_{e_1e_2+e_3e_4}$ is a $ U_{13}C_2$. Finally, $T_{W_2}$ and $U_1=X_{6,7}$ induce a $U_1T_2$ faithful action on $W_2$, giving us the full stabilizer $G_{W_2}=U_{14}C_2T_2$. 
\end{proof}

The remaining orbit representatives have at least $2$ points that are $G$-equivalent to $1$. 

\begin{lemma}
The stabilizers of $\langle  1,e_1e_2e_3e_4 \rangle$ and $\langle  1,e_1e_2 \rangle$ are $ U_{8}A_3T_2.2$ and $ U_{15}A_1^2A_2T_1$ respectively. 
\end{lemma}

\begin{proof}
We have $G_{\langle 1 \rangle}=G_{f_1,f_2,f_3,f_4,f_5}$ and $G_{\langle e_1e_2e_3e_4 \rangle}=G_{\langle e_1,e_2,e_3,e_4,f_5\rangle}$. Their intersection is a parabolic $U_8A_3T_2$. Let $w_0$ be the longest element of the Weyl group of $G$.
Since $1$ and $e_1e_2e_3e_4$ are respectively a lowest and highest weight vectors, $\dot{w}_0$ swaps $\langle 1 \rangle$ and $\langle e_1e_2e_3e_4 \rangle$. Since there are precisely two white points in $W_2$, this shows that indeed $G_{W_2}= U_{8}A_3T_2.2$. 

Finally, the intersection of $ G_{\langle 1 \rangle}$ and $G_{\langle e_1e_2 \rangle}$ is $ U_{15}A_1A_2T_2$, which with the $A_1$ action induced by $\langle X_{1,2},X_{5,6} \rangle$ gives the full stabilizer $G_{\langle 1, e_1e_2 \rangle}=U_{15}A_1^2A_2T_1 $.
\end{proof}

Let $H=B_4$ be the stabilizer in $G$ of $e_4+f_4\in V_{10}$. Then $H$ fixes a non degenerate quadratic form on $V$, with hyperbolic pairs given by the pairs of opposite weight vectors $(1,e_1e_2e_3e_5)$, $(e_ie_j,e_ke_l)$, $(e_ie_4,e_je_ke_le_4) $ where $\{i,j,k,l\}=\{1,2,3,5\}$. In particular note that all $2$-spaces listed in Table~\ref{tab:d5 reps} are totally singular.

We are then able to prove the following proposition:

\begin{proposition}
The $B_4$-orbits on totally singular $2$-spaces of $V_{B_4}(\lambda_4)$ are as in Table~\ref{tab:b4 reps}.

\begin{center}
 \begin{xltabular}{\textwidth}{l l l l} \caption{$B_4$-orbit representatives}
\label{tab:b4 reps}\\
 \hline
$D_5$-orbit& $D_5$-orbit representative & $B_4$-orbit representatives & $B_4$-stabs\\ 
 \hline
 \endhead
 1& $\langle 1+e_1e_2e_3e_4,e_1e_5+e_2e_3e_4e_5 \rangle$ &  $(p\neq 2)$ $\langle 1+e_1e_2e_3e_4,e_1e_5+e_2e_3e_4e_5 \rangle$ & $A_1A_2.2$\\
 &  &  $(p = 2)$ $\langle 1+e_1e_2e_3e_4,e_1e_5+e_2e_3e_4e_5 \rangle$ & $U_5A_1$\\
  &  &  $(p = 2)$ $\langle e_1e_2e_3e_4+e_1e_5, 1+e_2e_3e_4e_5 \rangle$ & $A_1G_2$\\

  2 &$\langle 1+e_1e_2e_3e_4,e_1e_2+e_2e_3e_4e_5\rangle $ & $\langle 1+e_1e_2e_3e_4,e_1e_2+e_2e_3e_4e_5\rangle $& $ U_8A_1T_1$\\

  3& $\langle 1,e_1e_5+e_1e_2e_3e_4\rangle$  & $\langle 1,e_1e_5+e_1e_2e_3e_4\rangle$& $U_9A_1T_2$\\

    4&$\langle 1,e_1e_2e_3e_4\rangle$ & $\langle 1,e_1e_2e_3e_4\rangle$& $ U_7A_2T_2.2$\\

   5& $\langle 1,e_1e_2+e_3e_4 \rangle $ & $\langle 1,e_1e_2+e_3e_4 \rangle $& $ U_{13}A_1T_2$\\

   6&$\langle 1,e_1e_2\rangle $ &$\langle 1,e_1e_2\rangle $ & $U_{11}A_1^3T_1$\\
     & &$\langle 1,e_1e_4\rangle $ & $U_{12}A_1A_2T_1$\\
 \hline
\end{xltabular}

\end{center}

\end{proposition}

\begin{proof}
We justify the listed stabilizers and conclude with a counting argument over finite fields. For $D_5$-orbits numbers $2,3,4,5$, simply finding the intersection of the $D_5$-stabilizers with $H=(D_5)_{e_4+f_4}$, gives the listed $B_4$-stabilizers. Consider the dense $D_5$-orbit. By the proof of Proposition~\ref{dense}, one of the representatives of this $G$-orbit is the totally singular $2$-space $W_2=\langle e_1e_2e_3e_4+e_1e_5, 1+e_2e_3e_4e_5 \rangle$, with stabilizer $A_1G_2$ fixing the sum $ \langle e_1,f_1,e_5-f_5 \rangle + \langle e_2,e_3,e_4,f_2,f_3,f_4,e_5+f_5 \rangle$. If $p\neq 2$ this sum is direct and $(A_1G_2)_{e_4+f_4}=A_1A_2.2$, which is then $H_{W_2}$. If $p=2$ then $(G_2)_{e_4+f_4}=U_5A_1$ and therefore $(A_1G_2)_{e_4+f_4}=U_5A_1^2$, which is $H_{W_2}$. Also, when $p=2$, the sum is not direct, and $A_1G_2\leq G_{e_5+f_5}$. The element $g=(1+f_5e_4)(1+e_5f_4)(1+f_5e_4)$ swaps $e_5+f_5$ and $e_4+f_4$. Therefore $g.W_2=\langle e_1e_2e_3e_4+e_1e_5, 1+e_2e_3e_4e_5 \rangle$ is a totally singular $2$-space with stabilizer $A_1G_2\leq H$. 

Let $W_2=\langle 1,e_1e_2\rangle$. We have $G_{W_2}\cap H= U_{11}A_1^3T_1$ as the stabilizer of $W_2$ in $H$. The element $g=(1+f_2e_4)(1+e_2f_4)(1+f_2e_4)$ sends $e_2+f_2$ to $e_4+f_4$. The stabilizer of $e_2+f_2$ in $ G_{W_2}$ is a $U_{12}A_1A_2T_1$, which is then the stabilizer of $g.W_2=\langle 1, e_1e_4\rangle$, which is a totally singular $2$-space.

We now show that the listed $H$-orbits are a complete set of orbits for the $H$-action on totally singular $2$-spaces. To do this let $q=p^e$ for an arbitrary positive integer $e$. Let $\sigma_q$ be the standard Frobenius morphism sending $x_{i,j}(t)$ to $x_{i,j}(t^q)$ and acting in a compatible way on $V$. Then the induced action of $\sigma$ on $P_2^{TS}(V)$ stabilises the $H$-orbits in Table~\ref{tab:b4 reps}, since for each orbit we have a representative given as a linear combination of basis elements with coefficients in $\{0,1\}$.  
The only orbits in Table~\ref{tab:b4 reps}, with a disconnected stabilizer are number $1$ for $p\neq 2$ and number $4$. 

Let $\Gamma_1$ be the $H$-orbit with representative $W_2=\langle 1+e_1e_2e_3e_4,e_1e_5+e_2e_3e_4e_5\rangle$ and stabilizer $H_{W_2}=A_1A_2.2$ for $p\neq 2$. The involution in $H_{W_2}/A_1A_2$ centralises an $ A_1\leq A_1A_2$ and induces a graph automorphism on $A_2$. Therefore by Lang-Steinberg, the fixed points of $\Gamma_1$ under $\sigma_q$, split into two $B_4(q)$ orbits with stabilizers $SL(3,q)SL(2,q).2$ and $SU(3,q)SL(2,q).2$. 
Note that $\Gamma_1$ is contained in the dense $D_5$-orbit on $2$-spaces of $V$, while we have seen that if $p=2$ the dense $D_5$-orbit contains two totally singular $H$-orbits with stabilizers $U_5A_1A_1$ and $A_1G_2$. Indeed we find the following polynomial equation for the orbit sizes:
\begin{gather*}
    [B_4(q):SL(3,q)SL(2,q).2]+[B_4(q):SU(3,q)SL(2,q).2]=\\
    =[B_4(q):SL(2,q)G_2(q)]+[B_4(q):q^5SL(2,q)^2].
\end{gather*}

Finally let $\Gamma_4$ be the $H$-orbit with representative $\langle 1,e_1e_2e_3e_4\rangle$ and stabilizer $U_7A_2T_2.2$. The involution in $H_{W_2}/U_7A_2T_2$ centralises a $ T_1\leq U_7A_2T_2$, induces a graph automorphism on $A_2$ and inverts a $T_1$. Therefore by Lang-Steinberg, the fixed points of $\Gamma_4$ under $\sigma_q$, split into two $B_4(q)$ orbits with stabilizers $q^7SL(3,q)(q-1)^2.2$ and $q^7SU(3,q)(q-1)(q+1).2$. 

Adding up the indices of the $B_4(q)$-stabilizers gives the number of totally singular $2$-spaces in $V_\sigma$. This means that we have found a complete list of orbit representatives for the $B_4$-action on totally singular $2$-spaces of $V_{B_4}(\lambda_4)$.
\end{proof}

\subsection{Remaining cases – no dense orbit}\label{no dense orbit section}

In this section we conclude the proof of Theorem~\ref{main theorem k=2}. This is achieved by proving the following proposition.

\begin{proposition}\label{prop no dense orbit}
Let $H=E_7$ and $\lambda=\lambda_7$, or $H=D_6$ and $\lambda=\lambda_6$, or $H=A_5$ and $\lambda=\lambda_3$. Let $V=V_H(\lambda)$. Then $H$ has no dense orbit on $P_2^{TS}(V)$.
\end{proposition}

In \cite[Proposition~$3.2.20$ ]{generic2} the authors determine the generic stabilizer for the $H$-action on all $2$-spaces. They find generic stabilizers with connected components respectively $D_4$, $A_1^3$ and $T_2$. Therefore by Corollary~\ref{minimum dimension generic} the dimension of the stabilizer of any $2$-space is respectively at least $28,9,2$. 
If $p\neq 2$ the module $V$ is symplectic, and if $p=2$ it is orthogonal. The dimension of $P_2^{TS}(V)$ is therefore respectively $107,59,37$ if $p\neq 2$, and  $105,55,35$ if $p=2$. Hence, if $p\neq 2$ there is no dense orbit for the $H$-action on $P_2^{TS}(V)$. We now focus on the case $p=2$. The strategy is to show that in each case there is an open dense subset of $P_2^{TS}(V)$ where every stabilizer has dimension respectively at least $29,10,3$.  

In the proof of \cite[Proposition~$3.2.20$ ]{generic2} the authors consider a subgroup $A_1^3\leq H$ and look at an $8$-dimensional subspace of $V$ on which $A_1^3$ acts as $\lambda_1\otimes \lambda_1\otimes \lambda_1$. They first find the generic stabilizer for this $A_1^3$ action, and then use this information to determine the generic stabilizer for the $H$-action on $P_2(V)$. Adopting a similar strategy we first analyse the $A_1^3$-action on $P_2^{TS}(V_{A_1^3}(\lambda_1\otimes \lambda_1\otimes \lambda_1)$, and then combine the information with the proof of \cite[Proposition~$3.2.20$ ]{generic2}.

Let $G=(SL_2)^3$ over $K$ of characteristic $p=2$, and let $V_8$ be the $8$-dimensional module $\lambda_1\otimes \lambda_1\otimes \lambda_1$. Let $\{e,f\}$ be a basis for $V_{A_1}(\lambda_1)$ and let $$(v_1,\dots ,v_8)=(e\otimes e \otimes e, e\otimes e\otimes f, \dots , f\otimes f\otimes e, f\otimes f\otimes f)$$ be an ordered basis for $V_8$, given in lexicographical order.

The group $G$ preserves a non-degenerate quadratic form on $V_8$ given by $$Q\left(\sum_{1\leq i \leq 8}\alpha_iv_i\right)=\sum_{1\leq i \leq 4}\alpha_i\alpha_{9-i}.$$

We define a $1$-parameter family of totally singular $2$-spaces given by $$W_2(\theta):=\langle u_1,u_2(\theta) \rangle :=\langle v_1+v_2+v_3+v_4+v_6+v_7, \theta v_1+  (\theta+1)v_2+\theta v_3+v_5 \rangle.$$ Let $$Y=\{W_2(\theta) : \theta\in K,\theta \neq 0,1\}.$$ 

We first show that $Y$ is a variety. Note that $P_2^{TS}(V_8)$ is naturally embedded in $P(\Lambda^2 V_8)$ via the Pl\"ucker embedding. More specifically, a $2$-space $W_2(\theta)$ is sent to a $1$-space $\langle u(\theta) \rangle \in P(\Lambda^2 V_8)$, where the coefficients of $u(\theta)$ are scaled to be in the set $\{0,1,\theta,1+\theta\}$. 

\begin{lemma}
The set $Y$ is a (quasi-projective) variety.
\end{lemma}
\begin{proof}
To show that $Y$ is a quasi-projective variety it suffices to define a list of homogeneous polynomials on $\Lambda^2 V_8$ whose set of common zeroes is precisely the embedding of $Y$ plus a finite set of points. 

There are of course many ways to do this. Let $u(\theta)$ be the image of $W_2(\theta)$ into $\Lambda^2 V_8$ via the Pl\"ucker embedding.
One way to construct the list of polynomials required is to have a polynomial $x_{ij}$ for every coefficient of $u(\theta)$ that is $0$;   $x_{i_1j_1}+x_{i_2j_2}$ when the coefficients are equal, and for every triple of non-zero coefficients that are pairwise different, a polynomial $x_{i_1j_1}+x_{i_2j_2}+x_{i_3j_3}$. We call this collection of polynomials $\mathfrak{S}$.

Clearly $u(\theta)$ is a zero of $\mathfrak{S}$. It remains to be seen that the common zeroes of $\mathfrak{S}$ only contain the image of $Y$ and finitely many other elements.
Let $x\in P(\Lambda^2 V_8)$ be a common zero. Consider an arbitrary polynomial $p$ of the form $x_{i_1j_1}+x_{i_2j_2}+x_{i_3j_3}$ in $\mathfrak{S}$. 
It is easy to see that since $x$ is a common zero of $\mathfrak{S}$, the values of the coefficients $\alpha_{i_1j_1},\alpha_{i_2j_2},\alpha_{i_3j_3}$ of $x$ completely determine the remaining coefficients. In particular, if  $\alpha_{i_1j_1},\alpha_{i_2j_2},\alpha_{i_3j_3}$ are all non-zero, then $x$ is in the image of $Y$.
Without loss of generality assume that $\alpha_{i_1j_1}=0$. Then, since $x$ is a root of $p$, we have $\alpha_{i_2j_2}=\alpha_{i_3j_3}$ and $x$ is the unique element of the set of zeroes of $\mathfrak{S}$ satisfying $\alpha_{i_2j_2}=\alpha_{i_3j_3}$ and $\alpha_{i_1j_1}=0$. 

We have therefore shown that the common zeroes of $\mathfrak{S}$ consist of the image of $Y$ together with three points, corresponding to the cases $\alpha_{i_1j_1}=0,\alpha_{i_2j_2}=0,\alpha_{i_3j_3}=0$.

The set $Y$ is therefore a $1$-dimensional subvariety of $P_2^{TS}(V_8)$.
\end{proof}

In what follows we are going to show that there is an open dense subset of $P_2^{TS}(V_8)$, such that the dimension of the stabilizer of any $2$-space contained in it is $1$.

Let us consider $U\leq G$ given by $$U=\left\{\left(\left(\begin{matrix} 1 &  \alpha \\ 0 & 1 \\ \end{matrix} \right), \left(\begin{matrix} 1 &  \beta \\ 0 & 1 \\ \end{matrix} \right), \left(\begin{matrix} 1 &  \gamma \\ 0 & 1 \\ \end{matrix} \right)\right) : \alpha,\beta,\gamma \in K\right \}.$$ We now determine the stabilizer in $U$ of $y\in Y$. 

\begin{lemma}
Let $y=W_2(\theta) \in Y$. The stabilizer in $U$ of $y$ is isomorphic to $U_1.2$.
\end{lemma}

\begin{proof}
We compute the action of an arbitrary element of $U$ on the given basis for $V_8$:

$v_1\rightarrow v_1$;\\
$v_2 \rightarrow \gamma v_1 +v_2$;\\
$v_3 \rightarrow \beta v_1+v_3$;\\
$v_4 \rightarrow \beta\gamma v_1 +\beta v_2+\gamma v_3+v_4$;\\
$v_5 \rightarrow \alpha v_1 +v_5$;\\
$v_6 \rightarrow \alpha\gamma v_1+\alpha v_2 +\gamma v_5 +v_6$;\\
$v_7 \rightarrow \alpha\beta v_1+\alpha v_3+\beta v_5 +v_7$;\\
$v_8 \rightarrow \alpha\beta\gamma v_1 +\alpha\beta v_2+\alpha\gamma v_3+\alpha v_4 +\beta\gamma v_5 + \beta v_6 + \gamma v_7+v_8$.

Consider the vector $u_1=v_1+v_2+v_3+v_4+v_6+v_7$. Its image $h(u_1)$ is $$(1+\gamma+\beta+\beta\gamma+\alpha\gamma+\alpha\beta)v_1+(1+\beta+\alpha)v_2+(1+\gamma+\alpha)v_3+v_4+(\gamma+\beta)v_5+v_6+v_7.$$
This shows that $h(u_1)\in W_2(\theta)$ if and only if $h(u_1)=u_1+\delta u_2$, for $\delta\in K$. This happens if and only if $(\gamma+\beta+\beta\gamma+\alpha\gamma+\alpha\beta)v_1+(\beta+\alpha)v_2+(\gamma+\alpha)v_3+(\gamma+\beta)v_5$ is multiple of $u_2$. This is equivalent to the following system of equations:
\begin{enumerate}
\item $(\theta+1)\gamma +\theta\beta+\alpha = 0$
\item $\alpha+\beta+\beta\gamma+\alpha\gamma+\alpha\beta = 0$
\end{enumerate}
 We treat this as a quadratic system for the unknowns $\beta,\gamma$. Since $\theta\neq 0$, we can multiply the second equation by $\theta$ to get $$\theta\alpha +(\alpha+(\theta+1)\gamma)(1+\gamma+\alpha)+\theta \alpha\gamma=0.$$ This is equivalent to $$(\theta+1)\gamma^2 +(\theta+1)\gamma + \alpha+\alpha^2=0,$$ which always produces two distinct solutions for $\gamma$. A routine check shows that under the same conditions $h(u_2)\in W_2(\theta)$. 
 
 This shows that $U_y$ is isomorphic to $U_1.2$.
 \end{proof}

We now find the full stabilizer of an element $y\in Y$.
\begin{lemma}
For any $y\in Y$ we have $G_y= U_y$.
\end{lemma}
\begin{proof}
We proceed by contradiction, by assuming that there exists an element $g=g_1g_2g_3\in G_y\setminus U_y$. Suppose that $g_1$ is not upper triangular. Then, using the fact that the projection of $U_y$ onto the first $A_1$ is isomorphic to $U_1$, we can assume that $g_1$ is of the form $g_1=\left(\begin{matrix} 0 & \lambda  \\  \lambda^{-1} & 0 \\ \end{matrix} \right)$. We let $g_2$ and $g_3$ be arbitrary elements $(b_{ij}),(c_{ij})$, respectively. 
 We determine the action of $g=g_1g_2g_3$ on the given basis:
 
 $v_1\rightarrow \lambda(b_{11}c_{11}v_5+b_{11}c_{21}v_6+b_{21}c_{11}v_7+b_{21}c_{21}v_8)$;\\
$v_2 \rightarrow \lambda (b_{11}c_{12}v_5+b_{11}c_{22}v_6+b_{21}c_{12}v_7+b_{21}c_{22}v_8)$;\\
$v_3 \rightarrow \lambda (b_{12}c_{11}v_5+b_{12}c_{21}v_6+b_{22}c_{11}v_7+b_{22}c_{21}v_8)$;\\
$v_4 \rightarrow \lambda (b_{12}c_{12}v_5+b_{12}c_{22}v_6+b_{22}c_{12}v_7+b_{22}c_{22}v_8)$;\\
 $v_5\rightarrow \lambda^{-1} (b_{11}c_{11}v_1+b_{11}c_{21}v_2+b_{21}c_{11}v_3+b_{21}c_{21}v_4)$;\\
$v_6 \rightarrow \lambda^{-1} (b_{11}c_{12}v_1+b_{11}c_{22}v_2+b_{21}c_{12}v_3+b_{21}c_{22}v_4)$;\\
$v_7 \rightarrow \lambda^{-1} (b_{12}c_{11}v_1+b_{12}c_{21}v_2+b_{22}c_{11}v_3+b_{22}c_{21}v_4)$;\\
$v_8 \rightarrow \lambda^{-1} (b_{12}c_{12}v_1+b_{12}c_{22}v_2+b_{22}c_{12}v_3+b_{22}c_{22}v_4)$.

Looking at the coefficient of $v_8$ in $g(u_1)$ and $g(u_2)$ gives $$b_{21}c_{21}+b_{21}c_{22}+b_{22}c_{21}+b_{22}c_{22}=0$$ and $$ \theta b_{21}c_{21}+(\theta+1)b_{21}c_{22}+\theta b_{22}c_{21}=0.$$ This shows that $b_{21}c_{22}+\theta b_{22}c_{22}=c_{22}(b_{21}+\theta b_{22})=0$. Suppose that $c_{22}\neq 0$. Then $b_{21}=\theta b_{22}$ and $c_{21} = c_{22}$. However this implies that the coefficient of $v_6$ in $g(u_1)$, i.e. $b_{11}c_{21}+b_{11}c_{22}+b_{12}c_{21}+b_{12}c_{22} $, is equal to $0$. Therefore $ g(u_1)=\alpha u_2$. But then the coefficient of $v_7$ in $g(u_1)$ is also $0$, which means that $c_{11}=c_{22}$, a contradiction since $c_{ij}\in A_1$.  Therefore $c_{22} =0$ and $b_{21}=b_{22}$. Again, since this implies that the coefficient of $v_7$ is $0$, we find that $g(u_1)=\alpha u_2$, which implies $b_{11}=b_{12}$, a contradiction.

We have therefore shown that $g_1$ is upper triangular, and we can now assume that our element $g=g_1g_2g_3\in G_y\setminus U_y$ has $g_1=\left(\begin{matrix} \lambda &  0 \\ 0 & \lambda^{-1} \\ \end{matrix} \right)$. Again, we let $g_2$ and $g_3$ be arbitrary elements $b_{ij},c_{ij}$, respectively. 
 
 We determine the action of $g=g_1g_2g_3$ on the given basis:
 
$v_1\rightarrow \lambda (b_{11}c_{11}v_1+b_{11}c_{21}v_2+b_{21}c_{11}v_3+b_{21}c_{21}v_4)$;\\
$v_2 \rightarrow \lambda (b_{11}c_{12}v_1+b_{11}c_{22}v_2+b_{21}c_{12}v_3+b_{21}c_{22}v_4)$;\\
$v_3 \rightarrow \lambda (b_{12}c_{11}v_1+b_{12}c_{21}v_2+b_{22}c_{11}v_3+b_{22}c_{21}v_4)$;\\
$v_4 \rightarrow \lambda (b_{12}c_{12}v_1+b_{12}c_{22}v_2+b_{22}c_{12}v_3+b_{22}c_{22}v_4)$;\\
$v_5\rightarrow \lambda^{-1} (b_{11}c_{11}v_5+b_{11}c_{21}v_6+b_{21}c_{11}v_7+b_{21}c_{21}v_8)$;\\
$v_6 \rightarrow \lambda^{-1} (b_{11}c_{12}v_5+b_{11}c_{22}v_6+b_{21}c_{12}v_7+b_{21}c_{22}v_8)$;\\
$v_7 \rightarrow \lambda^{-1} (b_{12}c_{11}v_5+b_{12}c_{21}v_6+b_{22}c_{11}v_7+b_{22}c_{21}v_8)$;\\
$v_8 \rightarrow \lambda^{-1} (b_{12}c_{12}v_5+b_{12}c_{22}v_6+b_{22}c_{12}v_7+b_{22}c_{22}v_8)$.

We now show that $b_{21}=c_{21}=0$ and that $b_{11}=c_{11}=\lambda$.
Consider $g(u_1)$. Requiring the coefficient of $v_8$ to be $0$ gives $b_{21}c_{22}+b_{22}c_{21}=0$. Similarly by looking at $g(u_2)$ we get $b_{21}c_{21}=0$. This forces $b_{21}=c_{21}=0$. Equating the coefficients for $v_4,v_6,v_7$ in $g(u_1)$ gives $\lambda b_{22}c_{22}=\lambda^{-1}b_{22}c_{11}=\lambda^{-1}b_{11}c_{22}$. Since $c_{22}\neq 0$ we get $b_{11}=\lambda^2 b_{22}$ and since $b_{22}\neq 0$ we get $c_{11}=\lambda^2 c_{22}$. Also, since $b_{11}=b_{22}^{-1}$ and $c_{11}=c_{22}^{-1}$, we get $b_{11}=c_{11}=\lambda$. Note that $b_{22}=c_{22}=\lambda^{-1}$. For clarity let us rewrite the images of the basis vectors in light of the new information:

$v_1\rightarrow \lambda^3 v_1$;\\
$v_2 \rightarrow \lambda^2 c_{12}v_1+ \lambda v_2$;\\
$v_3 \rightarrow \lambda^2 b_{12}v_1+ \lambda v_3$;\\
$v_4 \rightarrow \lambda b_{12}c_{12}v_1+b_{12}v_2+c_{12}v_3+ \lambda^{-1}v_4$;\\
$v_5\rightarrow  \lambda v_5$;\\
$v_6 \rightarrow c_{12}v_5+ \lambda ^{-1}v_6$;\\
$v_7 \rightarrow  b_{12}v_5+\lambda^{-1} v_7$;\\

Now consider $g(u_2)=\lambda((\theta\lambda^2+(\theta+1)\lambda c_{12}+\theta\lambda b_{12})v_1+(\theta +1)v_2+\theta v_3 +v_5)$. This must be equal to $\lambda u_2$.  Also, $\theta \lambda^2+(\theta+1)\lambda c_{12}+\theta \lambda b_{12}=\theta$.  Keeping this in mind consider $g(u_1)= A_1 v_1+ (\lambda+b_{12})v_2+(\lambda+c_{12})v_3+\lambda^{-1}v_4+(c_{12}+b_{12})v_5+\lambda^{-1}v_6+\lambda^{-1}v_6$, for some coefficient $A_1$. Therefore $g(u_1)=\lambda^{-1}u_1+(c_{12}+b_{12})u_2$ which, by equating the coefficient of $v_3$ and $v_4$ for $g(u_1)+(c_{12}+b_{12})u_2$, gives $\theta(c_{12}+b_{12})+\lambda+c_{12}=\lambda^{-1}$. Multiplying by $\lambda$ we get $\theta\lambda(c_{12}+b_{12})+\lambda^2+\lambda c_{12}=1$. Adding the equation $\theta \lambda^2+(\theta+1)\lambda c_{12}+\theta \lambda b_{12}=\theta$ we get $\lambda^2+1=\theta\lambda^2+ \theta$, which forces $\lambda^2=1$. Therefore $\lambda=1$ and since we have shown that $b_{21}=c_{21}=0$ and $b_{11}=c_{11}=\lambda=1$, we get that $g\in U$, a contradiction. Hence we do indeed have  $G_y=U_y$. 
\end{proof}

We are now ready to prove the following lemma.

\begin{lemma}
The elements of $Y$ all lie in different $G$-orbits.
\end{lemma}

\begin{proof}  To prove this it suffices to show that the stabilizers of elements of $Y$ are pairwise non-conjugate in $G$. Let $\theta_1, \theta_2\neq 0,1$ and consider $y(\theta_1)$ and $y(\theta_2)$. Suppose that $G_{y(\theta_1)}^g=G_{y(\theta_2)}$ for some $g=g_1g_2g_3\in G$. We are going to show that this implies that $\theta_1=\theta_2$. Consider the element $$\left(\left(\begin{matrix} 1 &  0 \\ 0 & 1 \\ \end{matrix} \right), \left(\begin{matrix} 1 &  1+\theta_1^{-1} \\ 0 & 1 \\ \end{matrix} \right), \left(\begin{matrix} 1 &  1 \\ 0 & 1 \\ \end{matrix} \right)\right)\in G_{y(\theta_1)}.$$ The only element in $G_{y(\theta_2)}$ that this can be conjugate to is $$\left(\left(\begin{matrix} 1 &  0 \\ 0 & 1 \\ \end{matrix} \right), \left(\begin{matrix} 1 &  1+\theta_2^{-1} \\ 0 & 1 \\ \end{matrix} \right), \left(\begin{matrix} 1 &  1 \\ 0 & 1 \\ \end{matrix} \right)\right).$$ Therefore $g_2$ is upper triangular. More precisely we have $$g_2= u_2\mathrm{diag}(x^{-1},x)$$ for $x=\sqrt{\frac{1+\theta_2^{-1}}{1+\theta_1^{-1}}}$ and some strictly upper triangular matrix $u_2$. Now consider the element $$\left(\left(\begin{matrix} 1 &  1 \\ 0 & 1 \\ \end{matrix} \right),\left(\begin{matrix} 1 &  \theta_1^{-1} \\ 0 & 1 \\ \end{matrix} \right), \left(\begin{matrix} 1 &  0 \\ 0 & 1 \\ \end{matrix} \right)\right)\in G_{y(\theta_1)}.$$ The only element in $G_{y(\theta_2)}$ that this can be conjugate to is $$\left(\left(\begin{matrix} 1 &  1 \\ 0 & 1 \\ \end{matrix} \right), \left(\begin{matrix} 1 &  \theta_2^{-1} \\ 0 & 1 \\ \end{matrix} \right), \left(\begin{matrix} 1 &  0 \\ 0 & 1 \\ \end{matrix} \right)\right).$$ This means that $g_2$ is of the form $$g_2= u_2'\mathrm{diag}(x^{-1},x)$$ for $x=\sqrt{\frac{\theta_2^{-1}}{\theta_1^{-1}}}$ and some strictly upper triangular matrix $u_2'$. However this implies that $$\sqrt{\frac{\theta_2^{-1}}{\theta_1^{-1}}}=\sqrt{\frac{1+\theta_2^{-1}}{1+\theta_1^{-1}}}$$ which happens if and only if $\theta_1=\theta_2$. We have therefore shown that the stabilizers of elements of $Y$ are pairwise non-conjugate in $G$. This concludes the proof.

\end{proof}
Let $y\in Y$. This automatically implies that the set of elements of $G$ such that $g.y\in Y$ is $Tran_G(y,Y)=G_y$, since no element in $Y$ can be sent by an element of $G$ to some other element in $Y$. We are ready to prove the following:

\begin{proposition}\label{a1 a1 a1 minimum dimension}
The stabilizer in $G$ of any totally singular $2$-space in $V_8$ is at least $1$-dimensional.
\end{proposition}
\begin{proof}

We have defined a family $Y$ of totally singular $2$-spaces, each with a $1$-dimensional transporter into $Y$. 

Since $P_2^{TS}(V_8)$ is $9$-dimensional we get that $$\dim P_2^{TS}(V_8)-\dim Y = \dim G - \dim \mathrm{Tran}_G(y,Y)$$ for all $y\in Y$. This shows that the set $Y$ is $Y$-exact, and by Lemma~\ref{loc to a subvariety lemma}, we get that the union of the $A_1^3$-orbits containing the elements of $Y$ contains an open dense subset of the variety of totally singular $2$-spaces. By Lemma~\ref{minimum dimension lemma} this gives a lower bound of $1$ for the dimension of the stabilizer in $A_1^3$ of any totally singular $2$-space in $V_8$.
\end{proof}

We conclude with the proof of Proposition~\ref{prop no dense orbit}.
\renewcommand*{\proofname}{Proof of Proposition~\ref{prop no dense orbit}}
\begin{proof}
We have seen at the start of this section that if $p\neq 2$ there is no dense orbit on $P_2^{TS}(V)$.

If $p=2$ \cite[Prop.~3.2.20]{generic2} shows that there is a variety $Y_0$ (called $Y$ in \cite[Prop.~3.2.20]{generic2}) and a dense subvariety $\hat{Y_1}$ of $P_2(V)$, such that every $y\in \hat{Y_1}$ is $Y_0$-exact. The set $Y_0$ is defined as the set of $2$-spaces of an $8$-dimensional subspace of $V$, which is naturally the $8$-dimensional module $V_8$ for $A_1^3\leq H$. 

The set $\hat{Y_1}$ is defined by requiring certain expressions in terms of the coefficients of a $V_8$ basis to be non-zero. We are now going to proceed in the following manner. First we note that $\hat{Y}_1$ contains a totally singular $2$-space. Then we take the intersections of the sets $Y_0$ and $\hat{Y}_1$ with $P_2^{TS}(V)$, in order to be able to apply Lemma~\ref{loc to a subvariety lemma} and conclude as in \cite[Prop.~3.2.20]{generic2}.

If $x$ is a generator for the multiplicative group of $\mathbb{F}_8$, the $2$-space spanned by $ (x^2,1,x,     1,     1,     1,     0,     1)$ and 
$(    1 ,    1  , x, x^2,     1,     0   ,     1)$ is totally singular and contained in $\hat{Y_1}$. 
We now define $Y^{TS}_{0}:=Y_0\cap P_2^{TS}(V)$ and $\hat{Y_1}^{TS}:=\hat{Y_1}\cap P_2^{TS}(V)$. 
Since  $\hat{Y_1}^{TS}\neq \emptyset$, it is a dense subset of $Y^{TS}_{0}$. 

The transporter of an element $y\in \hat{Y_1}^{TS}$ into $Y^{TS}_{0}$ is naturally the same as $\mathrm{Tran}_H(y,Y_0)$, i.e. a group with connected component $D_4A_1^3$ if $H=E_7$, $A_1^3A_1^3$ if $H=D_6$ and $T_2A_1^3$ if $H=A_5$. In each case the codimension of the transporter is equal to the codimension of $Y^{TS}_{0}$ in $P_2^{TS}(V)$. Therefore every $y\in \hat{Y_1}^{TS}$ is $Y^{TS}_0$-exact. 

To conclude it suffices to intersect $\hat{Y_1}^{TS}$ with the open dense subset for the $A_1^3$ action on $Y^{TS}_0$, to get a $Y^{TS}_0$-exact set $\hat{Y}$ where every stabilizer has dimension at least $29,10,3$ respectively. By Lemma~\ref{loc to a subvariety lemma} there is an open dense subset of $P_2^{TS}(V)$, such that the stabilizer of every element is respectively at least $29,10,3$ dimensional. This proves that there is no dense orbit for the action of $H$ on $P_2^{TS}(V)$.
 \end{proof}
\renewcommand*{\proofname}{Proof.}

\printbibliography
\end{document}